\documentclass[10pt]{article}
\usepackage{flafter,amsmath,amssymb,latexsym,psfrag,graphicx,color,indentfirst,bm,amsthm}
\usepackage[a4paper,top=1.5cm,bottom=1.5cm,left=2cm,right=1.6cm,marginparwidth=1.75cm]{geometry}
\usepackage[T1]{fontenc}
\usepackage[utf8]{inputenc}
\usepackage[colorlinks,
linktocpage=true,
linkcolor=red,
citecolor=red,
urlcolor=black,
anchorcolor=black
]{hyperref}
\usepackage[makeroom]{cancel}

\usepackage[numbers,sort&compress]{natbib}
\usepackage{appendix}
\usepackage{multirow}
\usepackage{caption}
\usepackage{subfigure}
\usepackage{graphicx}

\DeclareGraphicsExtensions{.eps,.pdf,.jpg,.png}
\allowdisplaybreaks

\newtheorem{theorem}{Theorem}[section]
\newtheorem{lemma}{Lemma}[section] %[theorem]
\newtheorem{corollary}{Corollary}[section] %[theorem]
\newtheorem{proposition}{Proposition}[section] %[theorem]
\newtheorem{assumption}{Assumption} %[theorem]
\numberwithin{equation}{section}
\numberwithin{equation}{section}
\numberwithin{figure}{section}

\DeclareRobustCommand{\rchi}{{\mathpalette\irchi\relax}}
\newcommand{\irchi}[2]{\raisebox{\depth}{$#1\chi$}}
\newcommand\obullet[1]{\ThisStyle{\ensurestackMath{%
  \stackon[1pt]{\SavedStyle#1}{\SavedStyle\kern.6\LMpt\bullet}}}}
\newcommand\ocirc[1]{\ThisStyle{\ensurestackMath{%
  \stackon[1pt]{\SavedStyle#1}{\SavedStyle\kern.6\LMpt\circ}}}}
\title{Time-Dependent Acoustic Waves Generated by\\  Multiple Resonant Bubbles:\\ Application to Acoustic Cavitation}

\author{Arpan Mukherjee\footnote{Radon Institute (RICAM), Austrian Academy of
Sciences, Altenbergerstrasse 69, A-4040, Linz, Austria (arpan.mukherjee@oeaw.ac.at). This author is supported by the Austrian Science Fund (FWF): P32660.} \ and Mourad Sini\footnote{Radon Institute (RICAM), Austrian Academy of
Sciences, Altenbergerstrasse 69, A-4040, Linz, Austria (mourad.sini@oeaw.ac.at). This author is partially supported by the Austrian Science Fund (FWF): P32660.}}

\begin{document}
\maketitle
\begin{abstract}

We analyse the ultrasound waves reflected by multiple bubbles in the linearized time-dependent acoustic model. The generated time-dependent wave field is estimated close to the bubbles. The motivation of this study comes from the therapy modality using acoustic cavitation generated by injected bubbles into the region of interest. The goal is to create enough, but not too much, pressure in the region of interest to eradicate anomalies in that region. In a previous work, we already showed that, in case of single bubble, the dominant part of the acoustic pressure near it splits into two main echos. The primary one is the incident field shifted and amplified at certain order. The secondary one is of periodic form which is related to the resonant frequency (i.e. the Minnaert one) created by the single bubble. This secondary wave can be amplified at will, at certain specific times, by tuning properly the material characteristics of the used bubble.

\noindent
Here, we derive the dominant part of the generated acoustic field by a cluster of bubbles taking into account the (high) contrasts of their mass density and bulk as well as their general distribution in the given region. 

\noindent
As consequences of these approximations, we highlight the following features:
\begin{enumerate}
\item If we use dimers (two close bubbles), or generally polymers, then  both the primary and the secondary waves can be amplified resulting in a remarkable enhancement of the whole echo in the whole time. The main reason for that is the closeness of the bubbles which translates the fact that the polymers resonate (even if each bubble don't). This feature is shown also when we use a set of separated polymers. Therefore, one can generate desired amount of pressure by injecting such a set of polymers. 

\item If we distribute the bubbles every where in the region of interest, in a periodic way for instance, then we can derive the effective acoustic model which turns out to be a dispersive one (due to the resonant behavior of the bubbles). Therefore, the original question of generating desired acoustic pressure can be related to the effective model. We show that, for a given desired pressure, we can tune the effective model to generate it.

\end{enumerate}
\end{abstract}

\maketitle
\section{Introduction and Statement of the Results}
\subsection{Introduction}

The propagation of ultrasound waves in a fluid containing spherically-shaped bubbles is described, through the evolution of their radii in time, by the Rayleigh-Plesset mathematical model, or the Keller-Miksis model, see \cite{Brennen, Y-2018}. A more general model is given in \cite{C-M-P-T:1985, C-M-P-T:1986} which allows for general shaped bubbles under different scales and distributions. Those models are non-linear in nature. Our motivation in studying such models is related to the so-called acoustic cavitation, see \cite{Y-2018}, which initially meant to describe the growth and eventually collapse of the bubbles when they are subject to high acoustic pressure. As it is reported in the literature, this is possible when we use high incident frequencies to generate a high pressure around the bubbles. Such phenomenon is used in many industrial applications as well as in imaging and therapy modalities, see \cite{Stride-Coussios-2010}.

Our work is based on the linearized model derived in \cite{C-M-P-T:1985, C-M-P-T:1986}, which is valid if the shapes of the bubbles are not varying very much. Such assumption might make sense if the used incident ultrasound is not that pronounce. Therefore, we use moderate incident ultrasound. Contrary to the original acoustic cavitation principle for which the goal is to inject very highly incident ultrasound to collapse the bubbles (and then make delivery for instance), our goal is to generate enough pressure, but not too much, around the bubbles. To ensure this, we rely on the resonant character of the used bubbles.

As a first attempt towards this goal, in \cite{AM2} we have considered the case of acoustic wave propagation in the presence of a single bubble. The bubble is assumed to have a small radius (as compared to a typical wavelength, or the interval of time $(0, \mathrm{T})$, or the incident wave, etc).  However, it enjoys contrasting properties of its mass density and bulk modulus (as compare to the ones of the background where the bubble is injected). With these properties, the bubble generates local spots in the time domain. These local spots are time-domain translations of the resonant character of the bubble while excited by time-harmonic incident plane waves at incident frequencies close to the Minnaert resonant frequency. Such a resonant frequency was already observed in the time-harmonic regime, see \cite{A-F-G-L-Z-2018}.\par
The outcome of the analysis in \cite{AM2}, is that the reflected pressure has a dominant part (estimated up to a distance to the bubble of the order its radius) splits into a primary wave, which is the incident wave but shifted and amplified, and a secondary wave, which is the local spot described above. Tuning properly the bubble, we showed that the primary wave can be amplified to a certain extent. However, the secondary wave can have any desired amplitude at certain times. Therefore, at those times, we can expect to generate any desire amount of pressure. This local spot appears as a resonant (i.e. oscillating) field which is high amplitude around the bubble which decays as we get away from it.

In the current work, we aim at extending that approach to the case where we have multiple bubbles. Each bubble has similar size/contrast properties as described above. We derive the dominant part of the generated acoustic field by a cluster of bubbles taking into account the (high) contrasts of their mass density and bulk as well as their general distribution in the given region. As consequences of these approximations, we highlight the following features:
\begin{enumerate}
\item If we use dimers (two close bubbles), or generally polymers, then  both the primary and the secondary waves can be amplified resulting in a remarkable enhancement of the whole echo in the whole time. The main reason for that is the closeness of the bubbles which translates the fact the polymer resonate (even if each bubble don't). This feature is shown also when we use a set of separated polymers. Therefore, one can generate desired amount of pressure by injecting such a set of polymers. 

\item If we distribute the bubbles every where in the region of interests, in a periodic way for instance, then we can derive the effective acoustic model which turns out to be a dispersive one (due to the resonant behavior of the bubbles). Therefore, the original question of generating desired acoustic pressure can be related to the effective model. We show that, for a given desired pressure, we can tune the effective model to generate it. This is possible at the expense of a numerical differentiation.

\end{enumerate}

There is a difference in the two situations above. In the first case, while we use a set of Dimers, the generated pressure can fit to the desired one in the vicinity of the Dimers. Between the Dimers, the error between the desired and generated pressure might depend on the distance between the Dimers. This is not taken into account in our our asymptotic expansion. In the second situation, we do provide the desired pressure every where in the region of interest. However, since we use numerical differentiation, then one needs to face the induced instability (which is known to be a mild instability, of the Holder type). One way to overcome this numerical differentiation step is to look at this issue as an optimal control problem or inverse problem using internal data.

%Acoustic scattering is a formidable tool employed in a myriad of applications, ranging from medical imaging to sonar and non-destructive testing. It facilitates the detection and characterization of objects that would otherwise be imperceptible. One strategy to augment acoustic scattering is by utilizing multiple resonating micro-bubbles, diminutive gas-filled bubbles commonly utilized as contrast agents in medical ultrasound imaging. Under the influence of an acoustic field, the micro-bubbles oscillate, generating a scattering effect that an ultrasound transducer can capture. However, a solitary micro-bubble may not produce a potent enough signal to be discernible. To surmount this constraint, researchers have devised approaches to use multiple resonating micro-bubbles that resonate synchronously, resulting in a more pronounced scattering signal. In comparison to traditional contrast agents, the application of multiple resonating micro-bubbles proffers numerous benefits. For instance, the frequency of the micro-bubble cluster can be finely adjusted to match the acoustic field frequency, thereby amplifying the scattering effect. Furthermore, the deployment of multiple micro-bubbles escalates the scattering cross-section, facilitating the detection of smaller targets.
%\newline

\subsection{The main result}
Let us now closely examine the mathematical model that elucidates the scattering of acoustic waves by $\mathrm{M}$ identical resonating bubbles in a homogeneous medium in $\mathbb{R}^3$:
\begin{align}\label{math-modelmulti}
\begin{cases}   \mathrm{k}_\mathrm{m}^{-1}u_{\mathrm{t}\mathrm{t}}- \text{div} \rho_\mathrm{m}^{-1}\nabla u = 0 & \text{in}\;\mathbb{R}^{3}\setminus \overline{\cup_{i=1}^\mathrm{M}\Omega_i} \times (0,\mathrm{T})\\
\mathrm{k}_\mathrm{c}^{-1}u_{\mathrm{t}\mathrm{t}}- \text{div} \rho_\mathrm{c}^{-1}\nabla u = 0 & \text{in}\;\Omega_i \times (0,\mathrm{T})\; \text{for}\ i=1,2,...,\mathrm{M}\\
 u\big|_{+} = u\big|_{-}  & \text{on} \;\partial\Omega_i\; \text{for}\ i=1,2,...,\mathrm{M}\\
 \rho_{\mathrm{m}}^{-1}\partial_\nu u\big|_{+} = \rho_{\mathrm{c}}^{-1} \partial_\nu u\big|_{-} & \text{on} \;\partial \Omega_i\; \text{for}\ i=1,2,...,\mathrm{M} \\
 u(\mathrm{x},0)= u_\mathrm{t}(\mathrm{x},0) = 0 & \text{for}\; \mathrm{x} \in \mathbb{R}^3.
\end{cases}
\end{align}
The mass density, denoted by $\rho$, is defined as $\rho = \rho_{\mathrm{c}}\rchi_{\Omega_i} + \rho_{m}\rchi_{\mathbb{R}^{3}\setminus\overline{\cup_{i=1}^\mathrm{M}\Omega_i}}$, where $\rho_{\mathrm{c}}$ and $\rho_{m}$ represent the mass density of the bubble and acoustic medium, respectively. Similarly, the bulk modulus of the bubble and acoustic medium, denoted by $\mathrm{k}$, are defined as $\mathrm{k} = \mathrm{k}_{\mathrm{c}}\rchi_{\Omega_i} + \mathrm{k}_{\mathrm{m}}\rchi_{\mathbb{R}^{3}\setminus\overline{\cup_{i=1}^\mathrm{M}\Omega_i}}$ for $i=1,2,...,\mathrm{M}$. Here, $\partial_\nu$ denotes the outward normal vector, and we use the notation $\partial_\nu \big|{\pm}$ to indicate $
\partial_\nu u \big|_{\pm}(\mathrm{x},\mathrm{t}) = \lim_{\mathrm{h}\to 0}\nabla u(\mathrm{x}\pm \mathrm{h}\nu_\mathrm{x},\mathrm{t})\cdot \nu_\mathrm{x},$ where $\nu$ is the outward normal vector to $\partial\Omega$.

Let $u:=u^\textbf{in}+u^\mathrm{s}$ denote the solution to the hyperbolic problem (\ref{math-modelmulti}). The existence and uniqueness of the solution to the direct scattering problem (\ref{math-modelmulti}) have been studied in the literature. The retarded boundary integral equation method has been employed for this purpose in \cite{sayas-transmission}, while the volume integral equation method has been utilized in \cite{AM2}.

Let us now consider the bubbles, which can be represented in the form of $\Omega_i= \delta\mathrm{B}_i+\mathrm{z}_i$ with $i=1,2,\ldots,\mathrm{M}$. Here, $\delta \in \mathbb R^+$, and the small value of $\delta\ll1$ denotes the relative size of $\Omega_i$ compared to the size of $\mathrm{B}_i$. The domains of $\mathrm{B}_i$ are bounded and have $\mathcal{C}^2$-regularity, containing the origin, while the parameter $\mathrm{z}_i\in \mathbb R^3$ indicates the respective locations of the bubbles.
\newline
To simplify the analysis, we make the further assumption that $\rho_{\mathrm{c}},\rho_{m},\mathrm{k}_\mathrm{c},\mathrm{k}_\mathrm{m}$ are positive constants. However, we impose certain scaling properties on the constants $\rho_\mathrm{c}$ and $\mathrm{k}_\mathrm{c}$ in $\Omega_i$ for $i=1,2,...,\mathrm{M}$. Specifically, we assert that
\begin{align}\label{cond-bubblemulti}
\rho_\mathrm{c} = \overline{\rho}_\mathrm{c}\delta^2, \quad \mathrm{k}_\mathrm{c} = \overline{\mathrm{k}}_\mathrm{c}\delta^2 \quad \text{and} \quad \frac{\rho_\mathrm{c}}{\mathrm{k}_\mathrm{c}} \sim 1 \ \text{as}\ \delta\ll1.
\end{align}
This enables us to establish a more streamlined framework for the analysis.
\newline
We define $\mathrm{d}$ as the minimum distance between the bubbles, denoted as $\mathrm{d}_{\mathrm{ij}} = \textbf{dist}(\Omega_i,\Omega\mathrm{j})$ for $i\ne \mathrm{j}$, where $\textbf{dist}$ denotes the distance function. Specifically, we have
\begin{align}
\mathrm{d}:= \min_{1\le i,\mathrm{j}\le \mathrm{M}} \mathrm{d}_{\mathrm{ij}}.
\end{align}
Furthermore, we assume $a$ as the maximum diameter among the resonating micro-bubbles, given by
\begin{align}
a := \max_{1\le i,\mathrm{j}\le \mathrm{M}} \textbf{diam}(\Omega_j)=\delta \max_{1\le i,\mathrm{j}\le \mathrm{M}} \textbf{diam}(\mathrm{B}_j),
\end{align}
where $\textbf{diam}$ represents the diameter function.

In order state the final result, we introduce the grad-harmonic sub-space
\begin{align}\nonumber
    \nabla \mathbb{H}_{\text{arm}} := \Big\{ u \in \big(\mathrm{L}^{2}(\mathrm{B}_i)\big)^3: \exists \  \varphi \ \text{s.t.} \ u = \nabla \varphi,\; \varphi\in \mathrm{H}^1(\mathrm{B}_i)\; \text{and} \ \Delta \varphi = 0 \Big\}
\end{align}
and we define the Magnetization operator as follows:
\begin{align}\nonumber
    \bm{\mathrm{M}}^{(0)}_{\mathrm{B}_i}\big[f\big](\mathrm{x}) := \nabla \int_{\mathrm{B}_i}\mathop{\nabla}\limits_{\mathrm{y}} \frac{1}{4\pi|\mathrm{x}-\mathrm{y}|} \cdot f(\mathrm{y})d\mathrm{y}.
\end{align}
It is well known, see for instance \cite{friedmanI}, that the  Magnetization operator $\mathbb{M}^{(0)}_{\mathrm{B}_i}: \nabla \mathbb{H}_{\text{arm}}\rightarrow \nabla \mathbb{H}_{\text{arm}}$ induces a complete orthonormal basis namely $\big(\lambda^{(3)}_{\mathrm{n}_{\textcolor{black}{i}}},\mathrm{e}^{(3)}_{\mathrm{n}_{\textcolor{black}{i}}}\big)_{\mathrm{n} \in \mathbb{N}}$.
\newline

We hereby present the main result of this work as follows:
\begin{theorem}\label{mainthmulti}
We posit that the incident wave field $u^\textbf{in}$ originating from a point source located at $\mathrm{x}_0\in \mathbb{R}^{3}\setminus\overline{\cup_{i=1}^\mathrm{M}\Omega_i}$ takes the form
\begin{align}\label{incident-wave}
u^\textbf{in}(\mathrm{x},\mathrm{t}, \mathrm{x}_0):=\frac{\lambda(\mathrm{t}-\mathrm{c}_0^{-1}\vert \mathrm{x}-\mathrm{x}_0\vert)}{\vert \mathrm{x}-\mathrm{x}_0\vert},
\end{align}
where $\lambda\in \mathcal{C}^9(\mathbb R)$,\footnote{The reason for considering the necessary degree of time-regularity is discussed in \cite{AM2}.} is a causal signal that vanishes for all $\mathrm{t}<0$. Then, under the condition
\begin{align}\label{inversion-cond}
    \frac{\rho_\mathrm{m}}{4\pi}\textbf{Vol}(\mathrm{B}_j)\;\Big(\frac{\delta}{\mathrm{d}}\Big)^6\; \Big(\frac{1}{\lambda_1^{(3)}}\Big)^2 <1,\; \text{where}\; \lambda_1^{(3)} = \textcolor{black}{\min_i}\max_{\mathrm{n}\in \mathbb N}\ \lambda^{(3)}_{\mathrm{n}_{\textcolor{black}{i}}},\; \textcolor{black}{and}\; \textcolor{black}{\max\limits_{1\le i\le M}\sum\limits_{\substack{j=1 \\ j\neq i}}^M\mathrm{q}_{ij}<d_i,\ i=1,2,\ldots, M,}
\end{align}\vspace{-5pt}
we have the asymptotic expansion
\begin{align}\label{assymptotic-expansion-us}
    u^\mathrm{s}(\mathrm{x},\mathrm{t}) = \sum_{i=1}^\mathrm{M}\frac{\alpha_i \rho_\mathrm{m}}{4\pi} \frac{1}{|\partial\Omega_i|}\int_{\partial\Omega_i}\frac{1}{|\mathrm{x}-\mathrm{y}|}d\sigma_\mathrm{y}\mathrm{Y}_i\big(\mathrm{t}-\mathrm{c}_0^{-1}|\mathrm{x}-\mathrm{z}_i|\big) + \mathcal{O}(\mathrm{M}\bm{\xi}^{-1}\delta^2)\; \text{as}\; \delta\to 0,
\end{align}\vspace{-5pt}
for $(\mathrm{x},\mathrm{t})\in \mathbb R^3\setminus\overline{\cup_{i=1}^\mathrm{M}\Omega_i}\times(0,\mathrm{T})$, where $\bm{\xi}= \max\limits_{1\le i\le \mathrm{M}}\textbf{dist}(\mathrm{x},\mathrm{z}_i),$ and $\bm{\mathrm{Y}}:=\big(\mathrm{Y}_i\big)_{i=1}^\mathrm{M}$ is the vector solution to the following non-homogeneous second-order matrix differential equation with initial conditions:
\begin{align}\label{matrixmulti}
\begin{cases}\displaystyle
    \mathrm{d}_i\frac{\mathrm{d}^2}{\mathrm{d}\mathrm{t}^2}\mathrm{Y}_i(\mathrm{t}) + \mathrm{Y}_i(\mathrm{t}) \textcolor{black}{+} \sum\limits_{\substack{j=1 \\ j\neq i}}^M\mathrm{q}_{ij} \frac{\mathrm{d}^2}{\mathrm{d}\mathrm{t}^2}\mathrm{Y}_j(\mathrm{t}-\mathrm{c}_0^{-1}|\mathrm{z}_i-\mathrm{z}_j|) = \frac{\rho_\mathrm{c}}{\mathrm{k}_\mathrm{c}}\mathrm{c}_0^2\int_{\partial\Omega_i} \partial_\nu u^\textbf{in} \mbox{ in } (0, \mathrm{T}),
     \\ \mathrm{Y}_i(\mathrm{0}) = \frac{\mathrm{d}}{\mathrm{d}\mathrm{t}}\mathrm{Y}_i(\mathrm{0}) = 0,
\end{cases}
\end{align}
where $\mathrm{d}_i$ is defined as $\mathrm{d}_i:= \frac{\rho_\mathrm{m}}{2}\alpha_i\frac{\rho_\mathrm{c}}{\mathrm{k}_\mathrm{c}} \mathrm{A}_{\partial\Omega_i}$, \textcolor{black}{ with $\alpha_i := \rho_\mathrm{c}^{-1}-\rho_\mathrm{m}^{-1}$ is the contrast between the inner and the outer acoustic
coefficient}. Furthermore, $\bm{\mathrm{Q}}=\big(\mathrm{q}_{ij}\big)_{i,j=1}^\mathrm{M}$ is given by
\begin{align}\label{interaction-matrix}
\mathrm{q}_{ij}=\begin{cases}
0 & \text{for}\ i=j,\\
\frac{\mathrm{b}_j}{|\mathrm{z}_i-\mathrm{z}_j|} & \text{for}\ i\ne j,
\end{cases}
\end{align}
where $\mathrm{b}_j$ is defined as $\mathrm{b}_j:= \textcolor{black}{\rho_\mathrm{m}}\alpha_j\delta^3\frac{\rho_\mathrm{c}}{\mathrm{k}_\mathrm{c}}$. We also define $\bm{\mathrm{B}}=\big(\mathrm{B}_i\big)_{i=1}^\mathrm{M}:= \displaystyle\Big(\int_{\partial\Omega_i} \partial_\nu u^\textbf{in}\Big)_{i=1}^\mathrm{M}$, and \\ $\displaystyle\mathrm{A}_{\partial\Omega_i} = \frac{1}{|\partial \Omega_i|}\int_{\partial\Omega_i}\int_{\partial\Omega_i}\frac{(\mathrm{x}-\mathrm{y})\cdot\nu_\mathrm{x}}{|\mathrm{x}-\mathrm{y}|}d\sigma_\mathrm{x}d\sigma_\mathrm{y}.$
\end{theorem}
\subsection{Discussion about the obtained result}
 
\begin{assumption}\label{assumption}
    In order to simplify the technicalities, we restrict our analysis to the scenario in which the material properties of the micro-bubbles being injected are identical, with bulk modulus ($\mathrm{k}_\mathrm{c}$) and mass density ($\rho_\mathrm{c}$) being uniform for all $i=1,2,\ldots,\mathrm{M}$. 
\end{assumption}

%--------------------------------------------------------------------------------
%--------------------------------------------------------------------------------

\subsubsection{The pressure generated by a Dimer of bubbles, i.e. \texorpdfstring{$M=2$}{M}}

As $\vert \mathrm{z}_1-\mathrm{z}_2 \vert$ is of the order $\delta$, then the system (\ref{matrixmulti}) can be approximately by the following one:

\begin{align}\label{matrixmulti-M-2}
\begin{cases}\displaystyle
    {\bm{\mathrm{A}}}\frac{\mathrm{d}^2}{\mathrm{d}\mathrm{t}^2}\mathrm{Y}_i(\mathrm{t}) + \mathrm{Y}_i(\mathrm{t}) = \frac{\rho_\mathrm{c}}{\mathrm{k}_\mathrm{c}}\mathrm{c}_0^2\ \mathrm{B}_i\mbox{ in } (0, \mathrm{T}), \\ 
     \mathrm{Y}_i(\mathrm{0}) = \frac{\mathrm{d}}{\mathrm{d}\mathrm{t}}\mathrm{Y}_i(\mathrm{0}) = 0,
\end{cases}
\end{align}
where 
\begin{equation}\label{A--M=2}
\bm{\mathrm{A}}:=diag(\mathrm{d}_i)-\bm{\mathrm{Q}}.
\end{equation}

As the matrix $\bm{\mathrm{A}}$ is real and symmetric, it can be diagonalized as $\bm{\mathrm{A}} = \bm{\mathrm{P}}\bm{\mathrm{D}}\bm{\mathrm{P}}^{-1}$, where $\bm{\mathrm{D}} = diag(\lambda_1,\lambda_2)$, 
\begin{align}
\begin{cases}
\bm{\mathrm{D}}\frac{\mathrm{d}^2}{\mathrm{d}\mathrm{t}^2} \bm{\mathrm{P}}^{-1}\bm{\mathrm{Y}}(\mathrm{t}) + \bm{\mathrm{P}}^{-1}\bm{\mathrm{Y}}(\mathrm{t}) = \frac{\rho_\mathrm{c}}{\mathrm{k}_\mathrm{c}}\mathrm{c}_0^2\ \bm{\mathrm{P}}^{-1} {\bm{\mathrm{B}}} \mbox{ in } (0, \mathrm{T}),\\
\bm{\mathrm{P}}^{-1}\bm{\mathrm{Y}}(0) = \frac{\mathrm{d}}{\mathrm{d}\mathrm{t}}\bm{\mathrm{P}}^{-1}\bm{\mathrm{Y}}(0) = 0.
\end{cases}
\end{align}
Simple computations give
\begin{align}
\bm{\mathrm{D}} = \begin{bmatrix}
\mathrm{d}_1 - \mathrm{b}_1\mathrm{q}_{12} & 0 \\
0 & \mathrm{d}_1 + \mathrm{b}_1\mathrm{q}_{12}
\end{bmatrix}, \; \text{as}\ \mathrm{d}_1=\mathrm{d}_2, \mathrm{b}_1=\mathrm{d}_2,\text{and}\; \mathrm{q}_{12}=\mathrm{q}_{21}
\end{align}
and $\bm{\mathrm{P}}:=\big(\mathrm{P}_{ij}\big)_{i,j=1}^\mathrm{M}$ as
\begin{align}
\bm{\mathrm{P}} = \begin{bmatrix}
1 & -1 \\
1 & 1
\end{bmatrix}\;
\text{ with }\;
\bm{\mathrm{P}}^{-1} = \frac{1}{2} \begin{bmatrix}
1 & 1 \\
-1 & 1
\end{bmatrix}.
\end{align}
We proceed by utilizing the change of variable $\bm{\mathrm{Z}}(\mathrm{t}) = \bm{\mathrm{P}}^{-1}\bm{\mathrm{Y}}(\mathrm{t})$ to obtain the ensuing non-homogeneous second-order matrix differential equation, along with initial conditions:
\begin{align}
\begin{cases}
     \bm{\mathrm{D}}\frac{\mathrm{d}^2}{\mathrm{d}\mathrm{t}^2} \bm{\mathrm{Z}}(\mathrm{t}) + \bm{\mathrm{Z}}(\mathrm{t}) = \frac{\rho_\mathrm{c}}{\mathrm{k}_\mathrm{c}}\mathrm{c}_0^2\ \bm{\mathrm{P}}^{-1} \bm{\mathrm{B}}\; \text{in}\ (0,\mathrm{T}),
     \\ \bm{\mathrm{Z}}(0) = \frac{\mathrm{d}}{\mathrm{d}\mathrm{t}}\bm{\mathrm{Z}}(0) = 0.
\end{cases}    
\end{align}
Thus, the solution of the problem can be expressed as:
\begin{align}\label{lamdamulti}
\mathrm{Z}_i(\mathrm{t}) = \lambda_i^{-\frac{1}{2}}\int_0^\mathrm{t}\sin \Big( \lambda_i^{-\frac{1}{2}}(\mathrm{t}-\tau)\Big) \mathrm{g}_{i1}(\tau) d\tau,
\end{align}
where $\lambda_i$ represents the eigenvalues of $\bm{\mathrm{D}}$, and $\displaystyle\mathrm{g}_{11}(\mathrm{t}) := \frac{1}{2}\Big(\int_{\partial\Omega_1} \partial_\nu u^\textbf{in}+\int_{\partial\Omega_2} \partial_\nu u^\textbf{in}\Big) $ and \\ $\displaystyle\mathrm{g}_{21}(\mathrm{t}) := \frac{1}{2}\Big(-\int_{\partial\Omega_1} \partial_\nu u^\textbf{in}+\int_{\partial\Omega_2} \partial_\nu u^\textbf{in}\Big) $ .
\newline
We state the following corollary under the assumptions of Theorem \ref{mainthmulti}:
\begin{corollary} \label{cor1}
As two micro-bubbles are spaced apart by a distance $\bm{\xi}:= |\mathrm{z}_1-\mathrm{z}_2|\sim \delta; \delta\ll 1$, where $\mathrm{z}_1$ and $\mathrm{z}_2$ denote the location points of the micro-bubbles,
\begin{align} \label{deri}
    u^\mathrm{s}(\mathrm{x},\mathrm{t}) = \sum_{i=1}^\mathrm{M}\frac{\alpha_i \rho_\mathrm{m}}{4\pi} \frac{\rho_\mathrm{c}}{\mathrm{k}_\mathrm{c}}\mathrm{c}_0^2\frac{1}{|\partial\Omega_i|}\int_{\partial\Omega_i}\frac{1}{|\mathrm{x}-\mathrm{y}|}d\sigma_\mathrm{y}\mathrm{Y}_i\big(\mathrm{t}-\mathrm{c}_0^{-1}|\mathrm{x}-\mathrm{z}_i|\big) + \mathcal{O}\big(\mathrm{M}\bm{\xi}^{-1}\delta^2\big),\; \text{as}\; \delta \to 0,
\end{align}
where,
\begin{align}
  \mathrm{Y}_1(\mathrm{t}) = \lambda_1^{-\frac{1}{2}}\int_0^\mathrm{t}\sin \Big( \lambda_1^{-\frac{1}{2}}(\mathrm{t}-\tau)\Big) \mathrm{g}_{11}(\tau) d\tau - \lambda_2^{-\frac{1}{2}}\int_0^\mathrm{t}\sin \Big( \lambda_2^{-\frac{1}{2}}(\mathrm{t}-\tau)\Big) \mathrm{g}_{21}(\tau) d\tau,  
\end{align}
and
\begin{align}
  \mathrm{Y}_2(\mathrm{t}) = \lambda_1^{-\frac{1}{2}}\int_0^\mathrm{t}\sin \Big( \lambda_1^{-\frac{1}{2}}(\mathrm{t}-\tau)\Big) \mathrm{g}_{11}(\tau) d\tau + \lambda_2^{-\frac{1}{2}}\int_0^\mathrm{t}\sin \Big( \lambda_2^{-\frac{1}{2}}(\mathrm{t}-\tau)\Big) \mathrm{g}_{21}(\tau) d\tau.  
\end{align}
\end{corollary}

Under this assumption of identical material properties of the injected micro-bubbles, we observe that the eigenvalues of the matrix $\bm{\mathrm{D}}$, denoted by $\lambda_i$, for $i=1,2$, can be expressed as follows:
\begin{align}
    \lambda_i^\mp &\nonumber = \mathrm{d}_i \mp \mathrm{b}_i \mathrm{q}_{12}
    \\ \nonumber &= \frac{\rho_\mathrm{m}}{2}\alpha_2\frac{\rho_\mathrm{c}}{\mathrm{k}_\mathrm{c}}\mathrm{A}_{\partial\Omega_i} \mp \frac{\rho_\mathrm{m}}{2}\alpha_2\delta^3\frac{\rho_\mathrm{c}}{\mathrm{k}_\mathrm{c}} \mathrm{q}_{12}
    \\ \nonumber &= \frac{\rho_\mathrm{m}}{2} \frac{\overline{\rho}_\mathrm{c}}{\overline{\mathrm{k}}_\mathrm{c}}\frac{1}{\overline{\rho}_\mathrm{c}} \Big( \mathrm{A}_{\partial\mathrm{B}_i}\mp\frac{\delta}{|\mathrm{z}_1-\mathrm{z}_2|}\Big),\; \big[\text{as}\; \alpha_2 = \delta^{-2}\frac{1}{\overline{\rho}_\mathrm{c}} + \mathcal{O}(1)\big],
    \\ \nonumber &= \frac{\rho_\mathrm{m}}{2{\overline{\mathrm{k}}_\mathrm{c}}} \Big( \mathrm{A}_{\partial\mathrm{B}_i}\mp\frac{\delta}{|\mathrm{z}_1-\mathrm{z}_2|}\Big)
    \\ \nonumber &= \frac{\rho_\mathrm{m}\mathrm{A}_{\partial\mathrm{B}_i}}{2{\overline{\mathrm{k}}_\mathrm{c}}} \Big( 1 \mp \frac{\delta}{\mathrm{A}_{\partial\mathrm{B}_i}|\mathrm{z}_1-\mathrm{z}_2|}\Big).
\end{align}
Let us then denote  $\omega_\mathrm{M} := \sqrt{
\frac{2\overline{\mathrm{k}_\mathrm{c}}}{\mathrm{A}_{\partial\mathrm{B}}\rho_\mathrm{m}}}.$ Therefore, we have
\begin{align}
    \lambda_i^\mp = \frac{1}{\omega_\mathrm{M}^2}\Big( 1 \mp \frac{\delta}{\mathrm{A}_{\partial\mathrm{B}_i}|\mathrm{z}_1-\mathrm{z}_2|}\Big).
\end{align}
\begin{assumption}
As elucidated in \cite{Ahcene-Mourad-JDE}, only the eigenvalue $\frac{1}{3}$ permits the corresponding eigenfunctions $\big(\mathrm{e}^{(3)}_{\mathrm{n}}\big)_{n \in \mathbb{N}}$ of the Magnetization operator to have non-zero averages, if our computations are limited to the case of a unit sphere. Upon examining the condition (\ref{inversion-cond}), it is therefore apparent that when we consider two micro-bubbles i.e. for $\mathrm{M}=2$, the condition is simplified to
\begin{align}
\frac{\delta}{|\mathrm{z}_1-\mathrm{z}_1|} < \Big(\frac{1}{3\rho_\mathrm{m}}\Big)^\frac{1}{6}.
\end{align}    
\end{assumption}
Utilizing Taylor's series expansion and integration by parts, we can derive the subsequent estimates for $ i=1,2,\ldots,\mathrm{M}:$
\begin{align}\label{toy}
\mathrm{B}_i:=\int_{\partial\Omega_i} \partial_\nu u^\textbf{in}(\mathrm{y},\mathrm{\tau})d\sigma_\mathrm{y} = \int_{\Omega_i} \Delta u^\textbf{in}(\mathrm{y},\mathrm{\tau})d\mathrm{y}  
= \frac{\rho_\mathrm{m}}{\mathrm{k}_\mathrm{m}}\int_{\Omega_i} u_{\mathrm{t}\mathrm{t}}^\textbf{in}(\mathrm{y},\mathrm{\tau})d\mathrm{y}
 = \frac{\rho_\mathrm{m}}{\mathrm{k}_\mathrm{m}}|\Omega_i| u_{\mathrm{t}\mathrm{t}}^\textbf{in}(\mathrm{z}_i,\mathrm{\tau}) + \mathcal{O}(\delta^4),
\end{align}
and
\begin{align}\label{e-multi-1}
  u_{\mathrm{t}\mathrm{t}}^\textbf{in}(\mathrm{z}_1,\mathrm{\tau}) = u_{\mathrm{t}\mathrm{t}}^\textbf{in}(\mathrm{z}_2,\mathrm{\tau}) + \mathcal{O}(\vert z_1-z_2 \vert).
\end{align}

Therefore $\mathrm{B}_1+\mathrm{B}_2=2 u_{\mathrm{t}\mathrm{t}}^\textbf{in}(\mathrm{z}_1,\mathrm{\tau}) +\mathcal{O}(\delta)$.
%\begin{assumption}
%Let us begin by assuming that $\displaystyle\frac{1}{|\partial\Omega_1|}\int_{\partial\Omega_1}\frac{1}{|\mathrm{x}-\mathrm{y}|}d\sigma_\mathrm{y} = \frac{1}{|\partial\Omega_2|}\int_{\partial\Omega_2}\frac{1}{|\mathrm{x}-\mathrm{y}|}d\sigma_\mathrm{y}$. To simplify matters, let us assume that $\Omega_1$ and $\Omega_2$ are spheres centered at $\mathrm{z}_1$ and $\mathrm{z}_2$ respectively, with radius $\delta$. Under these assumptions, we show that the function $\displaystyle\frac{1}{|\partial\Omega_i|}\int_{\partial\Omega_i}\frac{1}{|\mathrm{x}-\mathrm{y}|}d\sigma_\mathrm{y}=\vert \mathrm{x}-\mathrm{z}_i\vert^{-1}$ for $\mathrm{x}\in \mathbb R^3\setminus\Omega_i$ for $i=1,2.$ Consequently, if we consider the two bubbles are equidistant from $\mathrm{x}$, it is evident that $\frac{1}{|\mathrm{x}-\mathrm{z}_1|}=\frac{1}{|\mathrm{x}-\mathrm{z}_2|}.$
% \end{assumption}
Consequently, relying on the aforementioned observation, the expressions for $\lambda_i^\mp$, the estimates presented in (\ref{toy}), (\ref{e-multi-1}), and Corollary \ref{cor1}, we derive the resulting asymptotic expansion as $\delta\to 0$
\begin{align}\label{main-formula-multi}
    u^\mathrm{s}(\mathrm{x},\mathrm{t}) &= \frac{\omega_\mathrm{M}\rho_\mathrm{m}|\mathrm{B}_1|}{4\pi\overline{\mathrm{k}_\mathrm{c}}}\ \delta\ \mathrm{J}_1^{-\frac{1}{2}}\big( \frac{1}{|\mathrm{x}-\mathrm{z_1}|}+ \frac{1}{|\mathrm{x}-\mathrm{z_2}|}\big) \int_0^{\mathrm{t}-\mathrm{c}_0^{-1}|\mathrm{x}-\mathrm{z}_1|} &\nonumber\sin\big(\omega_\mathrm{M}\mathrm{J}_1^{-\frac{1}{2}}(\mathrm{t}-\mathrm{c}_0^{-1}|\mathrm{x}-\mathrm{z}_1|-\tau)\big)\; u_{\mathrm{t}\mathrm{t}}^\textbf{in}(\mathrm{z}_1,\mathrm{\tau})\;d\tau 
    \\ &+ \mathcal{O}\big( \delta \mathrm{J}_1^{-\frac{1}{2}} \frac{\vert z_1-z_2\vert}{|\mathrm{x}-\mathrm{z_2}|}\big)+ \mathcal{O}\big(\mathrm{M}\bm{\xi}^{-1}\delta^2\big),
\end{align}
where $\mathrm{J}_1 = 1 - \frac{\delta}{\mathrm{A}_{\partial\mathrm{B}_1}|\mathrm{z}_1-\mathrm{z}_2|}.$
\newline

Let us now closely scrutinize the form of the time-dependent term that appears in the dominant term (\ref{main-formula-multi}). By employing integration by parts and exploiting the zero initial conditions satisfied by $u^\textbf{in}$, we demonstrate that:
\begin{align}
    &\nonumber \int_0^{\mathrm{t}-\mathrm{c}_0^{-1}|\mathrm{x}-\mathrm{z}_1|} \sin\big(\omega_\mathrm{M}\mathrm{J}_1^{-\frac{1}{2}}(\mathrm{t}-\mathrm{c}_0^{-1}|\mathrm{x}-\mathrm{z}_1|-\tau)\big)\;u_{\mathrm{t}\mathrm{t}}^\textbf{in}(\mathrm{z}_1,\mathrm{\tau})\;d\tau 
    \\ &= \omega_\mathrm{M}^2\mathrm{J}_1^{-1}\int_0^{\mathrm{t}-\mathrm{c}_0^{-1}|\mathrm{x}-\mathrm{z}_1|} \sin\big(\omega_\mathrm{M}\mathrm{J}_1^{-\frac{1}{2}}(\mathrm{t}-\mathrm{c}_0^{-1}|\mathrm{x}-\mathrm{z}_1|-\tau)\big)\;u^\textbf{in}(\mathrm{z}_1,\mathrm{\tau})\;d\tau -\omega_\mathrm{M}\mathrm{J}_1^{-\frac{1}{2}}u^\textbf{in}(\mathrm{z}_1,\mathrm{t}-\mathrm{c}_0^{-1}|\mathrm{x}-\mathrm{z}_1|) 
\end{align}
Upon applying the aforementioned decomposition, it becomes apparent that the dominant term expressed in (\ref{main-formula-multi}) can be resolved into two distinct reflected waves. Specifically, we define the primary reflected wave as $\mathrm{U}_1(\mathrm{x},\mathrm{t})$, which can be expressed as follows:
$$
\mathrm{U}_1(\mathrm{x},\mathrm{t}):=\frac{\omega^2_\mathrm{M}\rho_\mathrm{m}|\mathrm{B}_1|}{4\pi\overline{\mathrm{k}_\mathrm{c}}}\;\mathrm{J}_1^{-1}\; \delta\; \big( \frac{1}{|\mathrm{x}-\mathrm{z_1}|}+ \frac{1}{|\mathrm{x}-\mathrm{z_2}|}\big)\; u^\textbf{in}(\mathrm{z}_1,t-\mathrm{c}_0^{-1}|\mathrm{x}-\mathrm{z}_1|),
$$
In addition, we introduce the secondary reflected wave as $\mathrm{U}_2(\mathrm{x},\mathrm{t})$, which is defined as:
$$
\mathrm{U}_2(\mathrm{x},\mathrm{t}):=\frac{\omega^3_\mathrm{M}\rho_\mathrm{m}|\mathrm{B}_1|}{4\pi\overline{\mathrm{k}_\mathrm{c}}}\;\mathrm{J}_1^{-\frac{3}{2}}\; \delta\; \big( \frac{1}{|\mathrm{x}-\mathrm{z_1}|}+ \frac{1}{|\mathrm{x}-\mathrm{z_2}|}\big)\;\int_0^{\mathrm{t}-\mathrm{c}_0^{-1}|\mathrm{x}-\mathrm{z}_1|} \sin\big(\mathrm{J}_1^{-\frac{1}{2}}\omega_\mathrm{M}(\mathrm{t}-\mathrm{c}_0^{-1}|\mathrm{x}-\mathrm{z}_1|-\tau)\big)u^\textbf{in}(\mathrm{z}_1,\mathrm{\tau}) d\tau
$$
It is important to note that $\mathrm{U}_1(\mathrm{x},\mathrm{t})$ is the primary wave, while $\mathrm{U}_2(\mathrm{x},\mathrm{t})$ is the secondary wave.
\begin{enumerate}
\item The primary reflected wave can be described as the incident wave that has been time-shifted by $\mathrm{c}_0^{-1}|\mathrm{x}-\mathrm{z}|$, and "amplified" by the amplitude $\displaystyle\frac{\omega^2_\mathrm{M}\rho_\mathrm{m}|\mathrm{B}_1|}{4\pi\overline{\mathrm{k}}_\mathrm{c}}\; \delta\;\mathrm{J}_1^{-1} \big( \frac{1}{|\mathrm{x}-\mathrm{z_1}|}+ \frac{1}{|\mathrm{x}-\mathrm{z_2}|}\big)$. This amplification is computed using the integral expression, which can be simplified as follows: $\displaystyle\frac{|\mathrm{B}_1|}{A_{\partial B_1}}\;\mathrm{J}_1^{-1} \;\frac{\delta}{2\pi}\; \big( \frac{1}{|\mathrm{x}-\mathrm{z_1}|}+ \frac{1}{|\mathrm{x}-\mathrm{z_2}|}\big)$
\bigskip

\item The secondary reflected wave is characterized by its resonant, oscillating field. Its amplitude is obtained by the coefficient $\displaystyle\frac{\omega^3_\mathrm{M}\rho_\mathrm{m}|\mathrm{B}_1|}{4\pi\overline{\mathrm{k}_\mathrm{c}}}\;\delta\;\mathrm{J}_1^{-\frac{3}{2}}\; \big( \frac{1}{|\mathrm{x}-\mathrm{z_1}|}+ \frac{1}{|\mathrm{x}-\mathrm{z_2}|}\big)$  
\newline

To gain a better understanding of the behavior of this oscillating field, we examine the scenario where the incident field $u^\textbf{in}$ is given by a wavefront $u^\textbf{in}( \mathrm{y},\mathrm{t}):=\frac{\delta(\mathrm{t}- \mathrm{c}_{0}^{-1}\vert \mathrm{y}-\mathrm{x}_0\vert)}{\vert \mathrm{y}-\mathrm{x}_0\vert}$. Under this condition, we obtain the expression for $\mathrm{U}_2(\mathrm{x},\mathrm{t})$ as follows:
$$
\mathrm{U}_2(\mathrm{x},\mathrm{t}):=\frac{\omega^3_\mathrm{M}\rho_\mathrm{m}|\mathrm{B}_1|}{4\pi\overline{\mathrm{k}_\mathrm{c}}} \;\delta\;\mathrm{J}_1^{-\frac{3}{2}} \big( \frac{1}{|\mathrm{x}-\mathrm{z_1}|}+ \frac{1}{|\mathrm{x}-\mathrm{z_2}|}\big) \frac{\sin\big(\omega_\mathrm{M}\mathrm{J}_1^{-\frac{1}{2}}(\mathrm{t}-\mathrm{c}_0^{-1}|\mathrm{x}-\mathrm{z}|-\mathrm{c}_0^{-1}|\mathrm{z}-\mathrm{x}_0|)\big)}{|\mathrm{z}-\mathrm{x}_0|}.
$$
\end{enumerate}
\bigskip

The analysis above shows that the amount of pressure can be increased if we put two bubbles, close to each other, than only one. A natural question then arises. What happen if we increase the number of bubbles, close to each other, i.e. in the form of a polymer? At the technical analysis level, the ideas in the previous subsection extend to such ensemble of bubbles which are equidistant and with a distance of the order $\delta$. It is obvious that, in $\mathbb{R}^3$, this is possible if we have $4$ equidistant bubbles, at maximum, i.e. $M=4$. In the next subsection, we state the corresponding results.
\subsubsection{The pressure generated by a tetramer of bubbles i.e. \texorpdfstring{$\mathrm{M}=4$}{M}}\label{M-4}
Let consider $4$ bubbles to be equidistant from each other, and 'centered' at the points $z_i, i=1, 2, 3, 4$, with  $|z_i - z_j| \sim \delta$ for $i \neq j$, where $i,j = 1,2,3,4$. Then, similar to the case for a dimer, the matrix $\bm{\mathrm{A}}$ takes the following form \begin{align}
     \bm{\mathrm{A}}:=diag(\mathrm{d}_i)-\bm{\mathrm{Q}}
 \end{align}
 where $\mathrm{d}_i:= \frac{\rho_\mathrm{m}}{2}\alpha_i\frac{\rho_\mathrm{c}}{\mathrm{k}_\mathrm{c}} \mathrm{A}_{\partial\Omega_i}$ and $\bm{\mathrm{Q}} :=\big(\mathrm{q}_{ij}\big)_{i,j=1}^\mathrm{M}$ is given by
\begin{align}
\mathrm{q}_{ij}=\begin{cases}
0 & \text{for}\ i=j,\\
\frac{\mathrm{b}_j}{|\mathrm{z}_i-\mathrm{z}_j|} & \text{for}\ i\ne j,
\end{cases}
\end{align}
where $\mathrm{b}_j$ is defined as $\mathrm{b}_j:= \frac{\rho_\mathrm{m}}{2}\alpha_j\delta^3\frac{\rho_\mathrm{c}}{\mathrm{k}_\mathrm{c}}$. Therefore, due to the Assumption \ref{assumption} and the diagonalizability of the matrix $\bm{\mathrm{A}}$, we have $\bm{\mathrm{A}} = \bm{\mathrm{P}}\bm{\mathrm{D}}\bm{\mathrm{P}}^{-1}$, where
\begin{align}
    \bm{\mathrm{D}}:= diag(\lambda_1,\lambda_2,\lambda_3,\lambda_4)= \begin{pmatrix}
        \mathrm{d}_1 - 3 \mathrm{q}_{12} & 0 &  0\\
        0 & \mathrm{d}_1 + \mathrm{q}_{12} & 0 & 0 \\
        0 & 0 & \mathrm{d}_1 + \mathrm{q}_{12} & 0\\
        0 & 0 & 0 & \mathrm{d}_1 + \mathrm{q}_{12} 
     \end{pmatrix},
\end{align}
\begin{align}
     \bm{\mathrm{P}} := \begin{pmatrix}
        1 & -1 &  -1 & -1\\
        1 & 0 & 0 & 1 \\
        1 & 0 & 1 & 0\\
        1 & 1 & 0 & 0 
     \end{pmatrix}\; \text{and} \; 
     \bm{\mathrm{P}}^{-1}:= \frac{1}{4}\begin{pmatrix}
        1 & 1 &  1 & 1\\
        -1 & -1 & -1 & 3 \\
       -1 & -1 & 3 & -1\\
       -1 & 3 & -1  & -1 
     \end{pmatrix}.
 \end{align}
Then, similar to the asymptotic expansion given in (\ref{deri}), we obtain the following expressions:
\begin{align}
  \mathrm{Y}_1(\mathrm{t}) \nonumber&= \lambda_1^{-\frac{1}{2}}\int_0^\mathrm{t}\sin \Big( \lambda_1^{-\frac{1}{2}}(\mathrm{t}-\tau)\Big) \mathrm{g}_{11}(\tau) d\tau - \lambda_2^{-\frac{1}{2}}\int_0^\mathrm{t}\sin \Big( \lambda_2^{-\frac{1}{2}}(\mathrm{t}-\tau)\Big) \mathrm{g}_{21}(\tau) d\tau
  \\ &-\lambda_3^{-\frac{1}{2}}\int_0^\mathrm{t}\sin \Big( \lambda_1^{-\frac{1}{2}}(\mathrm{t}-\tau)\Big) \mathrm{g}_{31}(\tau) d\tau - \lambda_4^{-\frac{1}{2}}\int_0^\mathrm{t}\sin \Big( \lambda_2^{-\frac{1}{2}}(\mathrm{t}-\tau)\Big) \mathrm{g}_{41}(\tau) d\tau,  
\end{align}
\begin{align}
  \mathrm{Y}_2(\mathrm{t}) = \lambda_1^{-\frac{1}{2}}\int_0^\mathrm{t}\sin \Big( \lambda_1^{-\frac{1}{2}}(\mathrm{t}-\tau)\Big) \mathrm{g}_{11}(\tau) d\tau + \lambda_4^{-\frac{1}{2}}\int_0^\mathrm{t}\sin \Big( \lambda_2^{-\frac{1}{2}}(\mathrm{t}-\tau)\Big) \mathrm{g}_{41}(\tau) d\tau, 
\end{align}
\begin{align}
  \mathrm{Y}_3(\mathrm{t}) = \lambda_1^{-\frac{1}{2}}\int_0^\mathrm{t}\sin \Big( \lambda_1^{-\frac{1}{2}}(\mathrm{t}-\tau)\Big) \mathrm{g}_{11}(\tau) d\tau + \lambda_3^{-\frac{1}{2}}\int_0^\mathrm{t}\sin \Big( \lambda_2^{-\frac{1}{2}}(\mathrm{t}-\tau)\Big) \mathrm{g}_{31}(\tau) d\tau, 
\end{align}
and 
\begin{align}
  \mathrm{Y}_4(\mathrm{t}) = \lambda_1^{-\frac{1}{2}}\int_0^\mathrm{t}\sin \Big( \lambda_1^{-\frac{1}{2}}(\mathrm{t}-\tau)\Big) \mathrm{g}_{11}(\tau) d\tau + \lambda_2^{-\frac{1}{2}}\int_0^\mathrm{t}\sin \Big( \lambda_2^{-\frac{1}{2}}(\mathrm{t}-\tau)\Big) \mathrm{g}_{21}(\tau) d\tau. 
\end{align}
Also, we have
\begin{align}
\begin{cases}
    \displaystyle\mathrm{g}_{11}(\mathrm{t}) := \frac{1}{4}\Big(\int_{\partial\Omega_1} \partial_\nu u^\textbf{in}+\int_{\partial\Omega_2} \partial_\nu u^\textbf{in}+ \int_{\partial\Omega_3} \partial_\nu u^\textbf{in}+\int_{\partial\Omega_4} \partial_\nu u^\textbf{in}\Big) \\
    \displaystyle\mathrm{g}_{21}(\mathrm{t}) := \frac{1}{4}\Big(-\int_{\partial\Omega_1} \partial_\nu u^\textbf{in}-\int_{\partial\Omega_2} \partial_\nu u^\textbf{in} - \int_{\partial\Omega_3} \partial_\nu u^\textbf{in}+3\int_{\partial\Omega_4} \partial_\nu u^\textbf{in}\Big)
    \\ 
   \displaystyle\mathrm{g}_{31}(\mathrm{t}) := \frac{1}{4}\Big(-\int_{\partial\Omega_1} \partial_\nu u^\textbf{in}-\int_{\partial\Omega_2} \partial_\nu u^\textbf{in}+ 3\int_{\partial\Omega_3} \partial_\nu u^\textbf{in}-\int_{\partial\Omega_4} \partial_\nu u^\textbf{in}\Big)
   \\ 
   \displaystyle\mathrm{g}_{41}(\mathrm{t}) := \frac{1}{4}\Big(-\int_{\partial\Omega_1} \partial_\nu u^\textbf{in}+3\int_{\partial\Omega_2} \partial_\nu u^\textbf{in} - \int_{\partial\Omega_3} \partial_\nu u^\textbf{in} -\int_{\partial\Omega_4} \partial_\nu u^\textbf{in}\Big).
\end{cases}
\end{align}
Consequently, with the aid of similar estimates as those derived in (\ref{toy}) and (\ref{e-multi-1}), and following the same approach discussed for a dimer of bubbles, we deduce the resulting asymptotic expansion as $\delta\to 0$
\begin{align}
    &u^\mathrm{s}(\mathrm{x},\mathrm{t}) \nonumber 
    \\&= \frac{\omega_\mathrm{M}\rho_\mathrm{m}|\mathrm{B}_1|}{4\pi\overline{\mathrm{k}}_\mathrm{c}}\ \delta\ \mathrm{J}_1^{-\frac{1}{2}}\big( \frac{1}{|\mathrm{x}-\mathrm{z_1}|}+ \frac{1}{|\mathrm{x}-\mathrm{z_2}|} + \frac{1}{|\mathrm{x}-\mathrm{z_3}|} + \frac{1}{|\mathrm{x}-\mathrm{z_4}|}\big) \mspace{-35mu} \int\limits_0^{\mathrm{t}-\mathrm{c}_0^{-1}|\mathrm{x}-\mathrm{z}_1|}&\nonumber \mspace{-45mu}\sin\big(\omega_\mathrm{M}\mathrm{J}_1^{-\frac{1}{2}}(\mathrm{t}-\mathrm{c}_0^{-1}|\mathrm{x}-\mathrm{z}_1|-\tau)\big)\; u_{\mathrm{t}\mathrm{t}}^\textbf{in}(\mathrm{z}_1,\mathrm{\tau})\;d\tau 
    \\ &+ \mathcal{O}\big( \delta \mathrm{J}_1^{-\frac{1}{2}} \frac{\vert z_1-z_j\vert}{|\mathrm{x}-\mathrm{z_j}|}\big) + \mathcal{O}\big(\bm{\xi}^{-1}\delta^2\big).
\end{align}
where $\mathrm{J}_1 = 1 - 3\frac{\delta}{\mathrm{A}_{\partial\mathrm{B}_1}|\mathrm{z}_1-\mathrm{z}_2|}, \; \text{for}\ \mathrm{j} = 2,3,4.$
\newline

As previously discussed for a dimer, the main part of the above asymptotic expansion can be broken down into a sum of primary reflected waves and secondary reflected waves generated by a tetramer. Hence, we observe that the distribution of a tetramer of bubbles allows us to generate more pressure near its domain compared to a dimer of bubbles. Precisely, we first see the appearance of the factor $3$ in the definition of $\mathrm{J}_1$ which makes larger than in the case of dimer. In addition, we have the appearance of the sum of four terms, i.e. $\big( \frac{1}{|\mathrm{x}-\mathrm{z_1}|}+ \frac{1}{|\mathrm{x}-\mathrm{z_2}|} + \frac{1}{|\mathrm{x}-\mathrm{z_3}|} + \frac{1}{|\mathrm{x}-\mathrm{z_4}|}\big)$, instead of only two terms of the dimer.
\newline
Based on these observations, it is legitimate to consider that by adding more and more bubbles, in the polymer, we could increase the amount of created pressure near it while decaying away from it.

\subsubsection{The pressure generated by a finite collection of Dimers of bubbles}
Here, we consider the situation where the set of $M=2 \aleph$ bubbles are distributed as a collection of dimers and those dimers are apart from each other. 
Precisely, they are of the form $D_{j_1}, D_{j_2}, j=1, ..., \aleph $ and the minimum distance between all the dimers $d_{\aleph}$ is large as compared to the maximum diameter of the bubbles, i.e. 
\begin{equation}\label{dimers-separation}
\delta \ll d_{\aleph}.
\end{equation} 
Under this last assumption, we clearly see that the problem (\ref{matrixmulti}) reduces, approximately in terms of the ratio $\frac{\delta}{d_{\aleph}}$, to the following form:
\begin{align}
\begin{cases}
     \bm{\mathrm{A}}\frac{\mathrm{d}^2}{\mathrm{d}\mathrm{t}^2} \bm{\mathrm{Y}}(\mathrm{t}) + \bm{\mathrm{Y}}(\mathrm{t}) = \frac{\rho_\mathrm{c}}{\mathrm{k}_\mathrm{c}}\mathrm{c}_0^2\ \bm{\mathrm{B}} \; \text{in}\ (0,\mathrm{T})
     \\ \bm{\mathrm{Y}}(0) = \frac{\mathrm{d}}{\mathrm{d}\mathrm{t}}\bm{\mathrm{Y}}(0) = 0.
\end{cases}
\end{align}
Here, $\bm{\mathrm{A}}$ is the block diagonal matrix defined as $\bm{\mathrm{A}} := diag\big(\bm{\mathrm{A}}_1,\bm{\mathrm{A}}_2,\ldots,\bm{\mathrm{A}}_\aleph\big),$ where $\bm{\mathrm{A}}_i$ is a $2\times 2$ matrix, given by:
\begin{align}
    \bm{\mathrm{A}}_i := 
    \begin{bmatrix}
     \mathrm{d}_i & \mathrm{b}_i\;\mathrm{q}_{i2}\\
     \mathrm{b}_i\;\mathrm{q}_{2i} & \mathrm{d}_i
    \end{bmatrix},
\end{align}
and $\bm{\mathrm{Y}} := \big(\bm{\mathrm{Y}}_{1},\bm{\mathrm{Y}}_{2},\ldots,\bm{\mathrm{Y}}_{\aleph}\big)^\mathrm{T}$ represents the vector-solution of the above differential equation with $\bm{\mathrm{Y}}_i := \big(\bm{\mathrm{Y}}_{i1},\bm{\mathrm{Y}}_{i2}\big)^\mathrm{T}.$ To see this, it is enough to observe that, under the condition (\ref{dimers-separation}), the interaction matrix (\ref{interaction-matrix}) reduces to a block diagonal matrix. 
\newline

As a consequence, we derive the resulting asymptotic expansion as $\delta\to 0$
\begin{align}\label{dimers}
    \nonumber&u^\mathrm{s}(\mathrm{x},\mathrm{t}) \\ \nonumber&= \sum_{\mathrm{i}=1}^\aleph\frac{\omega_\mathrm{M}\rho_\mathrm{m}|\mathrm{B}_{2i-1}|}{4\pi\overline{\mathrm{k}_\mathrm{c}}}\ \delta\ \mathrm{J}_i^{-\frac{1}{2}}\big( \frac{1}{|\mathrm{x}-\mathrm{z_{2i-1}}|}+ \frac{1}{|\mathrm{x}-\mathrm{z_{2i}}|}\big)\mspace{-30mu}\int\limits_0^{\mathrm{t}-\mathrm{c}_0^{-1}|\mathrm{x}-\mathrm{z}_{2i-1}|} \mspace{-40mu}\sin\big(\omega_\mathrm{M}\mathrm{J}_i^{-\frac{1}{2}}(\mathrm{t}-\mathrm{c}_0^{-1}|\mathrm{x}-\mathrm{z}_i|-\tau)\big)\; u_{\mathrm{t}\mathrm{t}}^\textbf{in}(\mathrm{z}_{2i-1},\mathrm{\tau})\;d\tau 
    \\ &+ \mathcal{O}\big( \delta \mathrm{J}_i^{-\frac{1}{2}} \frac{\vert \mathrm{z}_{2i-1}-\mathrm{z}_{2i}\vert}{|\mathrm{x}-\mathrm{z}_{2i}|}\big) + \mathcal{O}\big(\mathrm{M}\bm{\xi}^{-1}\delta^2\big),
\end{align}
where $\mathrm{J}_i = 1 - \frac{\delta}{\mathrm{A}_{\partial\mathrm{B}_{2i-1}}|\mathrm{z}_{2i-1}-\mathrm{z}_{2i}|}$, with $\mathrm{z}_{2i-1}$ and $\mathrm{z}_{2i}$ represents the centres of the corresponding dimer.
\newline

As previously discussed, the dominant part of the above asymptotic expansion can be decomposed into a sum of primary reflected waves and secondary reflected waves generated by each dimer. Therefore, this way, we can generate desired pressure around each dimer and hence near the domain where the dimers were distributed. 
\newline
As described in Section \ref{M-4}, following the steps described in the current section, we can distribute polymers, which are well seperated, all around the domain of interest $\Omega$ to create the desired pressure.

\subsubsection{The pressure generated by a periodic cluster of bubbles}
\bigskip

Now, we distribute the bubbles periodically  inside a given smooth domain $\mathbf{\Omega}$ with a period given by $\delta$, which means that the distance between close bubbles is of the order $\delta$. For simplicity, here also, we take the bubbles to be spherically shaped and having all the same contrasts. In this case, the pressure derived in (\ref{incident-wave})-(\ref{inversion-cond})-(\ref{assymptotic-expansion-us})-(\ref{matrixmulti}) converges to the solution of the following kind of Lippmann-Schwinger equation
\begin{equation}\label{effective-equation}
d\; \rchi_{\bm{\Omega}}\; \frac{\partial ^2}{\partial t^2} \bm{\mathrm{Y}} (\mathrm{x},\mathrm{t}) + \bm{\mathrm{Y}} (\mathrm{x},\mathrm{t}) \textcolor{black}{+}\int_{\mathbf{\Omega}}\frac{b}{4\pi\vert x-y\vert} \frac{\partial ^2}{\partial t^2}{\mathrm{Y} (x, t-c\vert x-y\vert)} dy =a \frac{\partial ^2}{\partial t^2}u^\textbf{in}(\mathrm{x},\mathrm{t}),\mbox{ for } \mathrm{x} \in \mathbb{R}^3, \mathrm{t} \in (0, \mathrm{T}),
\end{equation}
with appropriate coefficients $a, b, c$ and $d$. To this equation, we add the homogeneous initial conditions till the order $1$ for $\mathrm{Y}$.
\bigskip

Applying the differential operation $\big( c^{-1}\partial^2_t-\Delta \big)$ to (\ref{effective-equation}), we get 
\begin{equation}\label{dispersive-acoustic-model}
\big( c^{-1}\partial^2_t-\Delta \big)\big(d\; \rchi_{\bm{\Omega}}\;  \partial^2_{t} \mathrm{Y} +\mathrm{Y}\big)+b\rchi_{\mathbf{\Omega}}\; \partial_{t^2} \mathrm{Y}=0.
\end{equation}
We set $\mathrm{P}:=d\; \rchi_{\mathbf{\Omega}}\; \partial^2_{t} \mathrm{Y}+\mathrm{Y}$, then we get the following dispersive acoustic model

\begin{align}\label{math-model-dispersive-media}
\begin{cases}   
\partial_{t} \mathrm{U}+\nabla \mathrm{P}& = 0, \\
c^{-1}\partial_{t} \mathrm{P}+\nabla \cdot \mathrm{U} & = -b\; \rchi_{\mathbf{\Omega}}\; \partial_{t} \mathrm{Y}, \\
 d\; \rchi_{\mathbf{\Omega}}\; \partial^2_{t} \mathrm{Y}+\mathrm{Y}-\mathrm{P} &=0.
\end{cases}
\end{align}
Here $\mathrm{P}$ and $\mathrm{U}$ play the roles of the acoustic pressure and velocity fields respectively. The third relation plays a similar role as the electromagnetic susceptibility in modeling the electromagnetic dispersive media. Observe that when $d\ll 1$, then the model (\ref{math-model-dispersive-media}) reduces to the classical non-dispersive acoustic model
\begin{align}\label{math-model-non-dispersive-media}
\begin{cases}   
\partial_{t} \mathrm{U}+\nabla \mathrm{P}& = 0, \\
\big(c^{-1}+b\; \rchi_{\mathbf{\Omega}}\big)\partial_{t} \mathrm{P}+\nabla \cdot \mathrm{U} & =0.
\end{cases}
\end{align}
Therefore the coefficient $d$ is responsible for the dispersion effect. Looking closely at the definition of $d$, we see that $d=\omega^2_M+o(1)$ as $\delta \ll 1$, where $\omega_M$ is the Minnaert resonant frequency that is generated by the injected bubbles. The bubble's properties can be tuned to generate small Minnaert resonant frequency.

Let us now deal with our original goal of generating a desired pressure $\mathrm{P}(\mathrm{x},\mathrm{t})$ in $\mathbf{\Omega}\times (0, \mathrm{T})$. We rephrase the question as follows: Given $\mathrm{P}_0(\mathrm{x},\mathrm{t})$ in $\mathbf{\Omega}\times (0, \mathrm{T})$, can we find a coefficient $b$ so that the pressure $\mathrm{P}(\mathrm{x},\mathrm{t})$, generated through the model (\ref{math-model-dispersive-media}), equals the desired pressure $\mathrm{P}_0(\mathrm{x},\mathrm{t})$ in $\mathbf{\Omega}\times (0, \mathrm{T})$.

Indeed, solving the Cauchy problem $d\; \partial^2_{t} \mathrm{Y}+\mathrm{Y}=\mathrm{P}_0$ with zero initial values for $\partial^k_{t} \mathrm{Y}$, $k=0, 1$, we derive the values of $Y$. From (\ref{dispersive-acoustic-model}), i.e.
\begin{equation}\label{dispersive-acoustic-model-P0}
\big( c^{-1}\partial^2_t-\Delta \big) \mathrm{P}_0 +b\;\partial^2_{t} \mathrm{Y}=0,
\end{equation}
and at the expense of a numerical differentiation, we can derive the value of $b$.

The generation of the dispersive effective model (\ref{effective-equation}) (or (\ref{dispersive-acoustic-model})) from the cluster of bubbles is justified in \cite{AM4} under more general conditions on the cluster. In addition to handling general shaped bubbles, the periodicity in distributing these bubbles is not required.

The rest of the paper is divided into two sections. In Section \ref{mainthmulti}, we start by providing the needed Sobolev spaces and few of their properties followed with the used volume and surface integral operators, see Section \ref{Preliminaries-text}. In Section \ref{integral-equations}, we state the main time-domain system of integral equations with which we represent the acoustic field. The main part of the proof is given in Section \ref{dominant-fields} and Section \ref{End-proof}, where we derive the dominant acoustic fields generated by the cluster of bubbles. The needed a priori estimates are derived in Section \ref{apri}.

%--------------------------------------------------------------------------------
%--------------------------------------------------------------------------------

\section{Proof of Theorem \ref{mainthmulti}}

In this section, we present a comprehensive justification of the asymptotic expansion for the solution to (\ref{math-modelmulti}) in the limit of $\delta \ll 1$. Additionally, we establish the unique solvability of the linear algebraic system (\ref{algebricsystem}) in the case of multiple bubbles.

%--------------------------------------------------------------------------------
%--------------------------------------------------------------------------------

\subsection{Mathematical Preliminaries}\label{Preliminaries-text}
Let us introduce the appropriate function spaces that will be utilized in this article. For $\mathrm{r}\in \mathbb{R}$ and $\mathrm{T}\in (0,\infty]$, we shall define the space $\mathrm{H}_0^\mathrm{r}(0,\mathrm{T})$ as follows:
\begin{align}\nonumber
    \mathrm{H}_0^\mathrm{r}(0,\mathrm{T}) := \Big\{\mathrm{g}|_{(0,\mathrm{T})}: \mathrm{g} \in \mathrm{H}^\mathrm{r}(\mathbb{R})\; \text{with} \ \mathrm{g}_{(-\infty,0)}\equiv 0\Big\}.
\end{align}
Similarly, we can introduce a similar notion of a generalized space consisting of $\mathrm{X}$-valued function where $\mathrm{X}$ is a Hilbert space and denote it by $\mathrm{H}^\mathrm{r}_{0}(0,\mathrm{T};\mathrm{X})$. For $\sigma>0$ and $\mathrm{r}\in \mathbb{Z}_+$, we define
\begin{align}\nonumber
    \mathrm{H}^\mathrm{r}_{0,\sigma}(0,\mathrm{T};\mathrm{X}) := \Big\{\mathrm{f} \in \mathrm{H}^\mathrm{r}_{0}(0,\mathrm{T};\mathrm{X}): \sum_{\mathrm{n}=0}^\mathrm{r}\int_0^\mathrm{T}\mathrm{e}^{-2\sigma\mathrm{t}} \Vert\partial_\mathrm{t}^\mathrm{n}\mathrm{f}(\cdot,\mathrm{t})\Vert_\mathrm{X} d\mathrm{t}<\infty\Big\},
\end{align}
where the norm is given by
\begin{align}
    \mathrm{H}^\mathrm{r}_{0,\sigma}(0,\mathrm{T};\mathrm{X}) := \Big(\int_0^\mathrm{T}\mathrm{e}^{-2\sigma \mathrm{t}}\Big[\Vert \mathrm{f}\Vert_\mathrm{X}^2 + \sum_{\mathrm{n}=1}^\mathrm{r} \mathrm{T}^{2\mathrm{n}}\Vert \partial_\mathrm{t}^\mathrm{n}\mathrm{f} \Vert_\mathrm{X}^2\Big]d\mathrm{t}\Big)^\frac{1}{2}.
\end{align}
As a next step, we present an equivalent norm for $f \in \mathrm{H}^{\mathrm{r}}(\partial\Omega_i)$ based on the Slobodeckij norm for $\mathrm{r}\in (0,1)$, as described in \cite[pp. 20]{Grisvard}. This norm is given by
\begin{align}\label{normd1}
    \Vert \mathrm{f}\Vert_{\mathrm{H}^{\mathrm{r}}(\partial\Omega_i)} := \Vert \mathrm{f} \Vert_{\mathrm{L}^2(\partial \Omega_i)}^2 + \int_{\partial\Omega}\int_{\partial\Omega_i}\frac{|\mathrm{f}(\mathrm{x})-\mathrm{f}(\mathrm{y})|^2}{|\mathrm{x}-\mathrm{y}|^{2+2\mathrm{r}}}d\sigma_\mathrm{x}d\sigma_\mathrm{y}.
\end{align}
We also define the following dual norm for $\varphi \in \mathrm{H}^{-\frac{1}{2}}(\partial\Omega_i)$:
\begin{align}\label{normd2}
 \Vert \varphi \Vert_{\mathrm{H}^{-\frac{1}{2}}(\partial\Omega_i)} &= \sup_{0\neq \mathrm{f} \in \mathrm{H}^{\frac{1}{2}}(\partial\Omega_i )} \dfrac{|\langle \varphi, \mathrm{f} \rangle_{\partial \Omega_i}|}{\Vert \mathrm{f} \Vert_{\mathrm{H}^{\frac{1}{2}}(\partial\Omega_i )}},
 \end{align}
where $\langle\cdot,\cdot\rangle_{\partial\Omega_i}$ denotes the duality pairing between $\mathrm{H}^{\frac{1}{2}}(\partial\Omega_i)$ and $\mathrm{H}^{-\frac{1}{2}}(\partial\Omega_i)$. 
\newline

In our analysis, we require several function spaces denoted using blackboard bold font to represent vector fields within $\mathbb R^3$. Firstly, we introduce the following function spaces:
\begin{align}\nonumber
   \mathbb{H}(\text{div},\Omega):= \Big\{ \mathrm{f}\in \big(\mathrm{L}^{2}(\Omega)\big)^3 :\; \text{div}\; \mathrm{f} \in \mathrm{L}^{2}(\Omega)\Big\} \; \text{and}\
   \mathbb{H}(\text{curl},\Omega):= \Big\{ \mathrm{f}\in \big(\mathrm{L}^{2}(\Omega)\big)^3 :\; \text{curl}\; \mathrm{f} \in \big(\mathrm{L}^{2}(\Omega)\big)^3\Big\}.
\end{align}
We then consider the space of divergence-free as well as the space of irrotational vector fields 
\begin{align}\nonumber
   \mathbb{H}(\text{div}\;0,\Omega):= \Big\{ \mathrm{f}\in \mathbb{H}(\text{div},\Omega) :\; \text{div}\; \mathrm{f}= 0 \Big\} \; \text{and}\
   \mathbb{H}(\text{curl}\; 0,\Omega):= \Big\{ \mathrm{f}\in \mathbb{H}(\text{curl},\Omega) :\; \text{curl}\; \mathrm{f} = 0\Big\},
\end{align}
and their sub-spaces
\begin{align}\nonumber
   \mathbb{H}_0(\text{div}\;0,\Omega):= \Big\{ \mathrm{f}\in \mathbb{H}(\text{div}\;0,\Omega) :\; \nu\cdot\mathrm{f}= 0\; \text{on}\ \partial\Omega \Big\} \; \text{and}\
   \mathbb{H}_0(\text{curl}\; 0,\Omega):= \Big\{ \mathrm{f}\in \mathbb{H}(\text{curl}\;0,\Omega) :\; \nu \times \mathrm{f} = 0 \; \text{on}\ \partial\Omega\Big\},
\end{align}
respectively. 
We then define the retarded volume potential $\bm{\mathrm{V}}_{\Omega_i}$ by
\begin{align}
    \bm{\mathrm{V}}^{(ii)}\big[f\big](\mathrm{x},\mathrm{t}) := \int_{\Omega_i}\frac{\rho_\mathrm{m} }{4\pi|\mathrm{x}-\mathrm{y}|}f(\mathrm{y},\mathrm{t}-\mathrm{c}_0^{-1}|\mathrm{x}-\mathrm{y}|)d\mathrm{y}, \quad (\mathrm{x},\mathrm{t}) \in \Omega_i \times(0,\mathrm{T}).
\end{align}
and the adjoint double layer retarded wave operator $\bm{\mathcal{K}}^\mathrm{t}$ is defined for $(\mathrm{x},\mathrm{t}) \in \partial\Omega_i \times(0,\mathrm{T})$ by
\begin{align}
    \bm{\mathcal{K}}^{(ii)}\Big[f\Big](\mathrm{x},\mathrm{t}) := -\rho_\mathrm{m}\Big[\mathrm{c}_0^{-1}\int_{\partial \Omega_i} \frac{f_\mathrm{t}(\mathrm{y}, \mathrm{t}- \mathrm{c}_0^{-1}|\mathrm{x}-\mathrm{y}|)}{4\pi|\mathrm{x}-\mathrm{y}|}\frac{(\mathrm{x}-\mathrm{y})\cdot\nu_\mathrm{x}}{|\mathrm{x}-\mathrm{y}|}d\sigma_\mathrm{y} + \int_{\partial \Omega_i} \frac{f(\mathrm{y}, \mathrm{t}- \mathrm{c}_0^{-1}|\mathrm{x}-\mathrm{y}|)}{4\pi|\mathrm{x}-\mathrm{y}|}\frac{(\mathrm{x}-\mathrm{y})\cdot\nu_\mathrm{x}}{|\mathrm{x}-\mathrm{y}|^2}d\sigma_\mathrm{y}\Big].
\end{align}
Throughout this paper, we use the standard Sobolev space of order $\mathrm{r}$ on $\Omega_i$, which we denote as $\mathrm{H}^\mathrm{r}(\Omega_i)$. Let us finally mention that the operators we employ are defined using a bold symbol.

\subsection{Integral Representation of the Solution} \label{integral-equations}
We begin by defining the wave number $\mathrm{c}$ as the square root of the ratio between the bulk modulus $\mathrm{k}$ and the mass density $\rho$. This is essential as it allows for a distinction in propagation speed of the pressure field between the interior and exterior of the bubble $\Omega_i$. Additionally, we take into account the background velocity $\mathrm{c}_0$, which is determined by the ratio between the wave number $\mathrm{k}_\mathrm{m}$ and the mass density $\rho_\mathrm{m}$. We denote the unperturbed Green function of the background medium $(\rho_\mathrm{m},\mathrm{k}_\mathrm{m})$ by $\mathrm{G}^{\textcolor{black}{(m)}}(\mathrm{x},\mathrm{t})$, which can be expressed as follows
\begin{align}
\mathrm{G}^{\textcolor{black}{(m)}}(\mathrm{x},\mathrm{t}) := \rho_\mathrm{m}\frac{\delta_0(\mathrm{t}-\mathrm{c}_0^{-1} |\mathrm{x}|)}{4\pi |\mathrm{x}|}\; \text{in}\; \mathbb{R}^3\times \mathbb{R},
\end{align}
Here, $\delta_0$ denotes the Dirac delta distribution, and $\mathrm{G}^w$ is commonly known as the fundamental solution of the wave equation. It is imperative to consider the above factors while studying the behavior of the system. 
\newline
Let $u$ denote the solution to the equation (\ref{math-modelmulti}). Utilizing the Lippmann-Schwinger representation, as derived in \cite{AM2}, we obtain the following integral equation system of the total acoustic field $u^{(i)}$:
\begin{align}\label{ls-multi}
    u^{(i)} + \sum_{i=1}^\mathrm{M}\beta_i\int_{\mathbb R}\int_{\Omega_i} \mathrm{G}^{\textcolor{black}{(m)}}(\mathrm{x}-\mathrm{y};\mathrm{t}-\tau) u^{(i)}_{\mathrm{t}\mathrm{t}}\ d\mathrm{y}d\tau - \sum_{i=1}^\mathrm{M}\alpha_i\;\text{div}\int_{\mathbb R}\int_{\Omega_i} \mathrm{G}^{\textcolor{black}{(m)}}(\mathrm{x}-\mathrm{y};\mathrm{t}-\tau)\nabla u^{(i)} \ d\mathrm{y}d\tau = u^\textbf{in},
\end{align}
where $\beta_i := \mathrm{k}_\mathrm{c}^{-1}- \mathrm{k}_\mathrm{m}^{-1}\; \text{and}\; \alpha_i: = \rho_\mathrm{c}^{-1} - \rho_\mathrm{m}^{-1} $ are the contrasts between the inner and the outer acoustic coefficients, respectively.
\newline
By utilizing integration by parts and the fundamental principles of retarded acoustic layer potentials, it is feasible to convert the aforementioned integro-differential equation into an exclusively integral equation as follows: (for more in-depth computations, we refer to \cite{AM2})
\begin{align}\label{system}
    \Big(1+\frac{\alpha_i\rho_\mathrm{m}}{2}\Big)\partial_\nu u^{(i)} + \sum_{i=1}^\mathrm{M}\gamma_i \partial_\nu\bm{\mathrm{V}}_\Omega\Big[u^{(i)}_{\mathrm{t}\mathrm{t}}\Big]+ \sum_{i=1}^\mathrm{M}\alpha_i \bm{\mathcal{K}}^\mathrm{t}\Big[\partial_\nu u^{(i)}\Big] = \partial_\nu u^\textbf{in}\quad \text{on}\; \partial\Omega_i,
\end{align}
which can be reformulated as:
\begin{align}\label{ls1-multi}
    \Big(1+\frac{\alpha_i\rho_\mathrm{m}}{2}\Big)\partial_\nu u^{(i)} +\gamma_i \partial_\nu\bm{\mathrm{V}}^{(ii)}\Big[u^{(i)}_{\mathrm{t}\mathrm{t}}\Big]+ \alpha_i \bm{\mathcal{K}}^{(ii)}\Big[\partial_\nu u^{(i)}\Big] = \partial_\nu u^{\textbf{in}}-\sum\limits_{\substack{j=1 \\ j\neq i}}^M\gamma_j \partial_\nu\bm{\mathrm{V}}^{(ij)}\Big[u^{(j)}_{\mathrm{t}\mathrm{t}}\Big] - \sum\limits_{\substack{j=1 \\ j\neq i}}^M\alpha_j \bm{\mathcal{K}}^{(ij)}\Big[\partial_\nu u^{(j)}\Big],
\end{align}\vspace{-5pt}
where we denote by $\gamma_i := \beta-\alpha_i \frac{\rho_\mathrm{c}}{\mathrm{k}_\mathrm{c}}$. Due to the scaling properties of $\rho_\mathrm{c}$ and $\mathrm{k}_\mathrm{c}$, we observe that $\gamma_i \sim 1$.
\newline

Define the operator $\bm{\mathrm{V}}^{(ij)}\Big[f^{(j)}\Big]$ by
\begin{align}\vspace{-5pt}
 \bm{\mathrm{V}}^{(ij)}\big[f^{(j)}\big](\mathrm{x},\mathrm{t}) := \int_{\Omega_j}\frac{\rho_\mathrm{m} }{4\pi|\mathrm{x}-\mathrm{y}|}f^{(j)}(\mathrm{y},\mathrm{t}-\mathrm{c}_0^{-1}|\mathrm{x}-\mathrm{y}|)d\mathrm{y}, \quad (\mathrm{x},\mathrm{t}) \in \partial\Omega_i \times(0,\mathrm{T}),   
\end{align}
and for $(\mathrm{x},\mathrm{t}) \in \partial\Omega_i \times(0,\mathrm{T})$ by
\begin{align}
    &\nonumber\bm{\mathcal{K}}^{(ij)}\Big[f^{(j)}\Big](\mathrm{x},\mathrm{t}) \\ \nonumber&:= -\rho_\mathrm{m}\Big[\mathrm{c}_0^{-1}\int_{\partial \Omega_j} \frac{f^{(j)}_\mathrm{t}(\mathrm{y}, \mathrm{t}- \mathrm{c}_0^{-1}|\mathrm{x}-\mathrm{y}|)}{4\pi|\mathrm{x}-\mathrm{y}|}\frac{(\mathrm{x}-\mathrm{y})\cdot\nu_\mathrm{x}}{|\mathrm{x}-\mathrm{y}|}d\sigma_\mathrm{y} + \int_{\partial \Omega_j} \frac{f^{(j)}(\mathrm{y}, \mathrm{t}- \mathrm{c}_0^{-1}|\mathrm{x}-\mathrm{y}|)}{4\pi|\mathrm{x}-\mathrm{y}|}\frac{(\mathrm{x}-\mathrm{y})\cdot\nu_\mathrm{x}}{|\mathrm{x}-\mathrm{y}|^2}d\sigma_\mathrm{y}\Big].
\end{align}

%-----------------------------------------------------------------------------------
%-----------------------------------------------------------------------------------

\subsection{Estimating the Dominating Component of the Acoustic Pressure Field}\label{dominant-fields}

In order to proceed with the subsequent steps, it is essential to attain higher degrees of temporal regularity for the solution $u$ to (\ref{math-modelmulti}). For this purpose, we refer to \cite{AM2} for a detailed discussion on this matter. In brief, we consider a function $\lambda \in \mathrm{C}^9(\mathbb R)$ with its support in the positive real line. It was demonstrated in the aforementioned reference that $u \in \mathrm{H}^8_{0,\sigma}\big(0,\mathrm{T};\mathrm{H}^{1}(\mathbb{R}^3)\big)$ and consequently in $\mathcal{C}^7\big(0,\mathrm{T};\mathrm{H}^{1}(\mathbb{R}^3)\big)$. Moreover, it was shown that $\partial_\nu u^{(i)}$ belongs to $\mathrm{H}^6_{0,\sigma}\big(0,\mathrm{T};\mathrm{H}^{-\frac{1}{2}}(\partial\Omega_i)\big)$ and eventually, $\partial_\nu u^{(i)}\in \mathcal{C}^5\big(0,\mathrm{T};\mathrm{H}^{-\frac{1}{2}}(\partial\Omega_i)\big)$.\par
\bigskip
Subsequently, we start with integrating the expression (\ref{ls1-multi}) with respect to $\partial\Omega_i$, thereby yielding the following equation:
\begin{align}\label{aprmulti}
    \Big(1+\frac{\alpha_i\rho_\mathrm{m}}{2}\Big)\int_{\partial\Omega_i}\partial_\nu u^{(i)} +\gamma_i\int_{\partial\Omega_i} \partial_\nu\bm{\mathrm{V}}^{(ii)}\Big[u^{(i)}_{\mathrm{t}\mathrm{t}}\Big]+ \alpha_i \int_{\partial\Omega_i}\bm{\mathcal{K}}^{(ii)}\Big[\partial_\nu u^{(i)}\Big] \nonumber &= \int_{\partial\Omega_i}\partial_\nu u^{\textbf{in}}-\sum\limits_{\substack{j=1 \\ j\neq i}}^M\gamma_j \int_{\partial\Omega_i}\partial_\nu\bm{\mathrm{V}}^{(ij)}\Big[u^{(j)}_{\mathrm{t}\mathrm{t}}\Big] \\ &- \sum\limits_{\substack{j=1 \\ j\neq i}}^M\alpha_j \int_{\partial\Omega_i}\bm{\mathcal{K}}^{(ij)}\Big[\partial_\nu u^{(j)}\Big].
\end{align}
Using Taylor's expansion with respect to the time variable, the following expression can be obtained, which gives the value of the operator $\bm{\mathcal{K}}^{(ii)}$ acting on the function $\partial_\nu u^{(i)}$ as follows:
\begin{align}\label{ap1multi}
  \bm{\mathcal{K}}^{(ii)}\Big[\partial_\nu u^{(i)}\Big] 
    &=\nonumber\rho_\mathrm{m}\Bigg[ -\cancel{\mathrm{c}^{-1}_0\int_{\partial\Omega_i} \partial_\mathrm{t}\partial_\nu u^{(i)}(\mathrm{y}, \mathrm{t}) \frac{(\mathrm{x}-\mathrm{y}).\nu_\mathrm{x}}{|\mathrm{x}-\mathrm{y}|^2}d\sigma_\mathrm{x}d\sigma_\mathrm{y}}
    + \mathrm{c}^{-2}_0\int_{\partial\Omega_i}\partial_\mathrm{t}^2\partial_\nu u^{(i)}(\mathrm{y}, \mathrm{t}) \frac{(\mathrm{x}-\mathrm{y})\cdot\nu_\mathrm{x}}{|\mathrm{x}-\mathrm{y}|}d\sigma_\mathrm{y} 
    \\ \nonumber &- \mathrm{c}^{-3}_0\int_{\partial\Omega_i}\partial_\mathrm{t}^3\partial_\nu u^{(i)}(\mathrm{y}, \mathrm{t}_1) (\mathrm{x}-\mathrm{y})\cdot\nu_\mathrm{x} d\sigma_\mathrm{x} 
    - \int_{\partial\Omega_i}\partial_\nu u^{(i)}(\mathrm{y}, \mathrm{t}) \frac{(\mathrm{x}-\mathrm{y})\cdot\nu_\mathrm{x}}{|\mathrm{x}-\mathrm{y}|^3}d\sigma_\mathrm{y}
    \\ \nonumber &+ \cancel{\mathrm{c}^{-1}_0\int_{\partial\Omega_i} \partial_\mathrm{t}\partial_\nu u^{(i)}(\mathrm{y}, \mathrm{t}) \frac{(\mathrm{x}-\mathrm{y})\cdot\nu_\mathrm{x}}{|\mathrm{x}-\mathrm{y}|^2}d\sigma_\mathrm{x}d\sigma_\mathrm{y}} 
    - \frac{1}{2} \mathrm{c}^{-2}_0\int_{\partial\Omega_i}\partial_\mathrm{t}^2\partial_\nu u^{(i)}(\mathrm{y}, \mathrm{t}) \frac{(\mathrm{x}-\mathrm{y})\cdot\nu_\mathrm{x}}{|\mathrm{x}-\mathrm{y}|}d\sigma_\mathrm{y}  
    \\ &+ \frac{1}{3!}\mathrm{c}^{-3}_0\int_{\partial\Omega_i}\partial_\mathrm{t}^3\partial_\nu u^{(i)}(\mathrm{y}, \mathrm{t}_2) (\mathrm{x}-\mathrm{y})\cdot\nu_\mathrm{x} d\sigma_\mathrm{y}\Bigg],
\end{align}
where $\mathrm{t}_1, \mathrm{t}_2 \in \big(\mathrm{t}-\mathrm{c}^{-1}_0|\mathrm{x}-\mathrm{y}|,\mathrm{t}\big)$ and we denote $\mathrm{t}_{\mathrm{m}_1} := \max(\mathrm{t}_1, \mathrm{t}_2)$ and similarly we define $\mathrm{t}_{\mathrm{m}_2}$.
\newline
This expression involves multiple integrals over the boundary $\partial\Omega_i$ and has a complex structure. Applying further mathematical simplifications, we can obtain a more concise form of the expression
\begin{align}\label{fir1}
   \bm{\mathcal{K}}^{(ii)}\Big[\partial_\nu u^{(i)}\Big] 
    =\nonumber\rho_\mathrm{m}\Bigg[ \frac{1}{2} \mathrm{c}^{-2}_0\int_{\partial\Omega_i}\ \frac{(\mathrm{x}-\mathrm{y})\cdot\nu_\mathrm{x}}{|\mathrm{x}-\mathrm{y}|}\partial_\mathrm{t}^2\partial_\nu u^{(i)}(\mathrm{y}, \mathrm{t}) d\sigma_\mathrm{y} \nonumber&- \int_{\partial\Omega_i}\frac{(\mathrm{x}-\mathrm{y})\cdot\nu_\mathrm{x}}{|\mathrm{x}-\mathrm{y}|^3}\partial_\nu u^{(i)}(\mathrm{y}, \mathrm{t}) d\sigma_\mathrm{y} \\ &+ \frac{5}{6}\mathrm{c}^{-3}_0\int_{\partial\Omega_i} (\mathrm{x}-\mathrm{y})\cdot\nu_\mathrm{x}\partial_\mathrm{t}^3\partial_\nu u^{(i)}(\mathrm{y}, \mathrm{t}_{\mathrm{m}_1}) d\sigma_\mathrm{y} \Bigg].  
\end{align}
Upon performing the integration with respect to $\partial\Omega_i$ for the expression (\ref{fir1}), we obtain the following equations:
\begin{align}\label{apr1multi}
   \int_{\partial\Omega_i}\bm{\mathcal{K}}^{(ii)}\Big[\partial_\nu u^{(i)}\Big] 
    &=\nonumber\rho_\mathrm{m}\Bigg[ \frac{1}{2} \mathrm{c}^{-2}_0\int_{\partial\Omega_i}\partial_\mathrm{t}^2\partial_\nu u^{(i)}(\mathrm{y}, \mathrm{t})\int_{\partial\Omega_i}\ \frac{(\mathrm{x}-\mathrm{y})\cdot\nu_\mathrm{x}}{|\mathrm{x}-\mathrm{y}|} d\sigma_\mathrm{y} - \int_{\partial\Omega_i}\partial_\nu u^{(i)}(\mathrm{y}, \mathrm{t})\int_{\partial\Omega_i}\frac{(\mathrm{x}-\mathrm{y})\cdot\nu_\mathrm{x}}{|\mathrm{x}-\mathrm{y}|^3} d\sigma_\mathrm{y} \\ &+ \frac{5}{6}\mathrm{c}^{-3}_0\int_{\partial\Omega_i}\partial_\mathrm{t}^3\partial_\nu u^{(i)}(\mathrm{y}, \mathrm{t}_{\mathrm{m}_1})\int_{\partial\Omega_i} (\mathrm{x}-\mathrm{y})\cdot\nu_\mathrm{x} d\sigma_\mathrm{y} \Bigg].  
\end{align}
Moreover, we infer that:
\begin{align}\label{sec1}
\partial_\nu\bm{\mathrm{V}}^{(ii)}\Big[u^{(i)}_{\mathrm{t}\mathrm{t}}\Big] &:= \nonumber \partial_\nu\int_\Omega\frac{\rho_\mathrm{m}}{4\pi|\mathrm{x}-\mathrm{y}|}\partial_\mathrm{t}^2u^{(i)}(\mathrm{y},\mathrm{t}-\mathrm{c}^{-1}_0|\mathrm{x}-\mathrm{y}|)d\mathrm{y}
\\ \nonumber &= \rho_\mathrm{m}\Bigg[\partial_\nu\int_\Omega\frac{\partial_\mathrm{t}^2u^{(i)}(\mathrm{y},\mathrm{t})}{4\pi|\mathrm{x}-\mathrm{y}|}d\mathrm{y} - \underbrace{\partial_\nu\int_\Omega\frac{\partial_\mathrm{t}^{3}u(\mathrm{y},\mathrm{t})}{4\pi|\mathrm{x}-\mathrm{y}|}\mathrm{c}^{-1}_0|\mathrm{x}-\mathrm{y}|d\mathrm{y}}_{=0} + \partial_\nu\int_\Omega\frac{\partial_\mathrm{t}^{4}u^{(i)}(\mathrm{y},\mathrm{t}_3)}{4\pi|\mathrm{x}-\mathrm{y}|}\mathrm{c}^{-2}_0|\mathrm{x}-\mathrm{y}|^2d\mathrm{y}\Bigg]
\\ &= \rho_\mathrm{m}\Bigg[\partial_\nu\int_\Omega\frac{\partial_\mathrm{t}^2u^{(i)}(\mathrm{y},\mathrm{t})}{4\pi|\mathrm{x}-\mathrm{y}|}d\mathrm{y} + \partial_\nu\int_\Omega\frac{\partial_\mathrm{t}^{4}u^{(i)}(\mathrm{y},\mathrm{t}_3)}{4\pi|\mathrm{x}-\mathrm{y}|}\mathrm{c}^{-2}_0|\mathrm{x}-\mathrm{y}|^2d\mathrm{y}\Bigg],
\end{align}
where $\mathrm{t}_3 \in \big(\mathrm{t}-\mathrm{c}^{-1}_0|\mathrm{x}-\mathrm{y}|,\mathrm{t}\big)$.
\newline
After taking the integration with respect to $\partial\Omega_i$ for (\ref{sec1}), we get
\begin{align}\label{apr3}
\int_{\partial\Omega_i}\partial_\nu\bm{\mathrm{V}}^{(ii)}\Big[u^{(i)}_{\mathrm{t}\mathrm{t}}\Big] \nonumber&= \rho_\mathrm{m}\int_{\partial\Omega_i}\partial_\nu\int_{\Omega_i}\frac{\partial_\mathrm{t}^2u^{(i)}(\mathrm{y},\mathrm{t})}{4\pi|\mathrm{x}-\mathrm{y}|}d\mathrm{y}d\sigma_\mathrm{x} + \rho_\mathrm{m}\int_{\partial\Omega_i}\partial_\nu\int_{\Omega_i}\frac{\mathrm{c}^{-2}_0}{4\pi}|\mathrm{x}-\mathrm{y}|\partial_\mathrm{t}^{4}u^{(i)}(\mathrm{y},\mathrm{t}_3)d\mathrm{y}
\\ \nonumber&= \rho_\mathrm{m}\int_{\Omega_i}\Delta\int_{\Omega_i}\frac{\partial_\mathrm{t}^2u^{(i)}(\mathrm{y},\mathrm{t})}{4\pi|\mathrm{x}-\mathrm{y}|}d\mathrm{y}d\mathrm{x} + \rho_\mathrm{m}\int_{\partial\Omega_i}\partial_\nu\int_{\Omega_i}\frac{\mathrm{c}^{-2}_0}{4\pi}|\mathrm{x}-\mathrm{y}|\partial_\mathrm{t}^{4}u^{(i)}(\mathrm{y},\mathrm{t}_3)d\mathrm{y}
\\ \nonumber&= \rho_\mathrm{m}\int_{\Omega_i}\Delta\bm{\mathcal{N}_{\Omega_i,\textbf{lap}}}\Big[\partial_\mathrm{t}^2u^{(i)}\Big]d\mathrm{x} + \rho_\mathrm{m}\int_{\partial\Omega_i}\partial_\nu\int_{\Omega_i}\frac{\mathrm{c}^{-2}_0}{4\pi}|\mathrm{x}-\mathrm{y}|\partial_\mathrm{t}^{4}u^{(i)}(\mathrm{y},\mathrm{t}_3)d\mathrm{y}
\\ &= \rho_\mathrm{m}\frac{\mathrm{k}_\mathrm{c}}{\rho_\mathrm{c}}\int_{\partial\Omega_i}\partial_\nu u^{(i)}(\mathrm{y},\mathrm{t})d\sigma_\mathrm{y} + \rho_\mathrm{m}\int_{\partial\Omega_i}\partial_\nu\int_{\Omega_i}\frac{\mathrm{c}^{-2}_0}{4\pi}|\mathrm{x}-\mathrm{y}|\partial_\mathrm{t}^{4}u^{(i)}(\mathrm{y},\mathrm{t}_3)d\mathrm{y}.
\end{align}
We then give a close look on the following expression for $(\mathrm{x},\mathrm{t}) \in \partial\Omega_i \times(0,\mathrm{T})$
\begin{align}
    \sum\limits_{\substack{j=1 \\ j\neq i}}^M\alpha_j\int_{\partial\Omega_i}\bm{\mathcal{K}}^{(ij)}\Big[\partial_\nu u^{(j)}\Big](\mathrm{x},\mathrm{t})d\sigma_\mathrm{x} \nonumber&= -\frac{\rho_\mathrm{m}}{\mathrm{c}_0}\sum\limits_{\substack{j=1 \\ j\neq i}}^M\alpha_j\int_{\partial\Omega_i}\int_{\partial \Omega_j} \frac{\partial_\mathrm{t}\partial_\nu u^{(j)}(\mathrm{y}, \mathrm{t}- \mathrm{c}_0^{-1}|\mathrm{x}-\mathrm{y}|)}{4\pi|\mathrm{x}-\mathrm{y}|}\frac{(\mathrm{x}-\mathrm{y})\cdot\nu_\mathrm{x}}{|\mathrm{x}-\mathrm{y}|}d\sigma_\mathrm{y} d\sigma_\mathrm{x}
    \\ &-\textcolor{black}{{\rho_m}}\sum\limits_{\substack{j=1 \\ j\neq i}}^M\alpha_j \int_{\partial \Omega_j}\int_{\partial \Omega_j} \frac{\partial_\nu u^{(j)}(\mathrm{y}, \mathrm{t}- \mathrm{c}_0^{-1}|\mathrm{x}-\mathrm{y}|)}{4\pi|\mathrm{x}-\mathrm{y}|}\frac{(\mathrm{x}-\mathrm{y})\cdot\nu_\mathrm{x}}{|\mathrm{x}-\mathrm{y}|^2}d\sigma_\mathrm{y}d\sigma_\mathrm{x}.
\end{align}
Next, we first consider the first term of the above expression and we do the following using Taylor's expansion:
\begin{align}\label{firstestimate}
    \mathrm{S}_1 \nonumber&:= \frac{\rho_\mathrm{m}}{4\pi\mathrm{c}_0}\sum\limits_{\substack{j=1 \\ j\neq i}}^M\alpha_j\int_{\partial\Omega_i}\int_{\partial \Omega_j} \partial_\mathrm{t}\partial_\nu u^{(j)}(\mathrm{y}, \mathrm{t}- \mathrm{c}_0^{-1}|\mathrm{x}-\mathrm{y}|)\frac{(\mathrm{x}-\mathrm{y})\cdot\nu_\mathrm{x}}{|\mathrm{x}-\mathrm{y}|^2}d\sigma_\mathrm{y}d\sigma_\mathrm{x}
    \\ \nonumber&= \frac{\rho_\mathrm{m}}{4\pi}\mathrm{c}_0^{-1}\sum\limits_{\substack{j=1 \\ j\neq i}}^M\alpha_j\int_{\partial\Omega_i}\int_{\partial \Omega_j} \partial_\mathrm{t}\partial_\nu u^{(j)}(\mathrm{y}, \mathrm{t}- \mathrm{c}_0^{-1}|\mathrm{z}_i-\mathrm{z}_j|)\frac{(\mathrm{x}-\mathrm{y})\cdot\nu_\mathrm{x}}{|\mathrm{x}-\mathrm{y}|^2}d\sigma_\mathrm{y}d\sigma_\mathrm{x}
    \\ \nonumber&+ \frac{\rho_\mathrm{m}}{4\pi}\mathrm{c}_0^{-2}\sum\limits_{\substack{j=1 \\ j\neq i}}^M\alpha_j\int_{\partial\Omega_i}\int_{\partial \Omega_j} \partial_\mathrm{t}\partial_\nu u^{(j)}\big(\mathrm{y}, \mathrm{t}- \mathrm{c}_0^{-1}|\mathrm{z}_i-\mathrm{z}_j|\big)|\mathrm{z}_i-\mathrm{z}_j|\frac{(\mathrm{x}-\mathrm{y})\cdot\nu_\mathrm{x}}{|\mathrm{x}-\mathrm{y}|^2}d\sigma_\mathrm{y}d\sigma_\mathrm{x}
    \\ \nonumber&- \frac{\rho_\mathrm{m}}{4\pi}\mathrm{c}_0^{-2}\sum\limits_{\substack{j=1 \\ j\neq i}}^M\alpha_j\int_{\partial\Omega_i}\int_{\partial \Omega_j} \partial_\mathrm{t}\partial_\nu u^{(j)}\big(\mathrm{y}, \mathrm{t}- \mathrm{c}_0^{-1}|\mathrm{z}_i-\mathrm{z}_j|\big)\frac{(\mathrm{x}-\mathrm{y})\cdot\nu_\mathrm{x}}{|\mathrm{x}-\mathrm{y}|}d\sigma_\mathrm{y}d\sigma_\mathrm{x}
    \\ \nonumber&+ \frac{\rho_\mathrm{m}}{4\pi}\frac{\mathrm{c}_0^{-3}}{2}\sum\limits_{\substack{j=1 \\ j\neq i}}^M\alpha_j\int_{\partial\Omega_i}\int_{\partial \Omega_j} \partial_\mathrm{t}\partial_\nu u^{(j)}(\mathrm{y}, \eta_1)\;|\mathrm{z}_i-\mathrm{z}_j|^2\;\frac{(\mathrm{x}-\mathrm{y})\cdot\nu_\mathrm{x}}{|\mathrm{x}-\mathrm{y}|^2}d\sigma_\mathrm{y}d\sigma_\mathrm{x}
    \\ \nonumber&- \frac{\rho_\mathrm{m}}{4\pi}\mathrm{c}_0^{-3}\sum\limits_{\substack{j=1 \\ j\neq i}}^M\alpha_j\int_{\partial\Omega_i}\int_{\partial \Omega_j} \partial_\mathrm{t}\partial_\nu u^{(j)}\big(\mathrm{y}, \eta_1\big)\;|\mathrm{z}_i-\mathrm{z}_j|\;\frac{(\mathrm{x}-\mathrm{y})\cdot\nu_\mathrm{x}}{|\mathrm{x}-\mathrm{y}|}d\sigma_\mathrm{y}d\sigma_\mathrm{x}
    \\ &+ \frac{\rho_\mathrm{m}}{4\pi}\frac{\mathrm{c}_0^{-3}}{2}\sum\limits_{\substack{j=1 \\ j\neq i}}^M\alpha_j\int_{\partial\Omega_i}\int_{\partial \Omega_j} \partial^3_\mathrm{t}\partial_\nu u^{(j)}\big(\mathrm{y}, \eta_1\big)\;(\mathrm{x}-\mathrm{y})\cdot\nu_\mathrm{x}\;d\sigma_\mathrm{y}d\sigma_\mathrm{x},
\end{align}
where $\eta_1 \in \big(\mathrm{t}- \mathrm{c}_0^{-1}|\mathrm{x}-\mathrm{y}|,\mathrm{t}- \mathrm{c}_0^{-1}|\mathrm{z}_i-\mathrm{z}_j|\big).$\par
Following that, we consider the second term and do the following:
\begin{align}\label{secondestimate}
    \mathrm{S}_2 \nonumber&:= \frac{\rho_\mathrm{m}}{4\pi}\sum\limits_{\substack{j=1 \\ j\neq i}}^M\alpha_j\int_{\partial\Omega_i}\int_{\partial \Omega_j} \partial_\nu u^{(j)}(\mathrm{y}, \mathrm{t}- \mathrm{c}_0^{-1}|\mathrm{x}-\mathrm{y}|)\frac{(\mathrm{x}-\mathrm{y})\cdot\nu_\mathrm{x}}{|\mathrm{x}-\mathrm{y}|^3}d\sigma_\mathrm{y}d\sigma_\mathrm{x}
    \\ \nonumber&= \underbrace{\frac{\rho_\mathrm{m}}{4\pi}\sum\limits_{\substack{j=1 \\ j\neq i}}^M\alpha_j\int_{\partial\Omega_i}\int_{\partial \Omega_j} \partial_\nu u^{(j)}(\mathrm{y}, \mathrm{t}- \mathrm{c}_0^{-1}|\mathrm{z}_i-\mathrm{z}_j|)\frac{(\mathrm{x}-\mathrm{y})\cdot\nu_\mathrm{x}}{|\mathrm{x}-\mathrm{y}|^3}d\sigma_\mathrm{y}d\sigma_\mathrm{x}}_{=\;0}
    \\ \nonumber&+ \underbrace{\frac{\rho_\mathrm{m}}{4\pi}\mathrm{c}_0^{-1}\sum\limits_{\substack{j=1 \\ j\neq i}}^M\alpha_j\int_{\partial\Omega_i}\int_{\partial \Omega_j} \partial_\mathrm{t}\partial_\nu u^{(j)}\big(\mathrm{y}, \mathrm{t}- \mathrm{c}_0^{-1}|\mathrm{z}_i-\mathrm{z}_j|\big)\;|\mathrm{z}_i-\mathrm{z}_j|\;\frac{(\mathrm{x}-\mathrm{y})\cdot\nu_\mathrm{x}}{|\mathrm{x}-\mathrm{y}|^3}d\sigma_\mathrm{y}d\sigma_\mathrm{x}}_{=\; 0}
    \\ \nonumber&- \frac{\rho_\mathrm{m}}{4\pi}\mathrm{c}_0^{-1}\sum\limits_{\substack{j=1 \\ j\neq i}}^M\alpha_j\int_{\partial\Omega_i}\int_{\partial \Omega_j} \partial_\mathrm{t}\partial_\nu u^{(j)}\big(\mathrm{y}, \mathrm{t}- \mathrm{c}_0^{-1}|\mathrm{z}_i-\mathrm{z}_j|\big)\frac{(\mathrm{x}-\mathrm{y})\cdot\nu_\mathrm{x}}{|\mathrm{x}-\mathrm{y}|^2}d\sigma_\mathrm{y}d\sigma_\mathrm{x}
    \\ \nonumber&+ \underbrace{\frac{\rho_\mathrm{m}}{4\pi}\frac{\mathrm{c}_0^{-2}}{2}\sum\limits_{\substack{j=1 \\ j\neq i}}^M\alpha_j\int_{\partial\Omega_i}\int_{\partial \Omega_j} \partial^2_\mathrm{t}\partial_\nu u^{(j)}(\mathrm{y}, \mathrm{t}- \mathrm{c}_0^{-1}|\mathrm{z}_i-\mathrm{z}_j|)\;|\mathrm{z}_i-\mathrm{z}_j|^2\;\frac{(\mathrm{x}-\mathrm{y})\cdot\nu_\mathrm{x}}{|\mathrm{x}-\mathrm{y}|^3}d\sigma_\mathrm{y}d\sigma_\mathrm{x}}_{=\; 0}
    \\ \nonumber&- \frac{\rho_\mathrm{m}}{4\pi}\mathrm{c}_0^{-2}\sum\limits_{\substack{j=1 \\ j\neq i}}^M\alpha_j\int_{\partial\Omega_i}\int_{\partial \Omega_j} \partial^2_\mathrm{t}\partial_\nu u^{(j)}\big(\mathrm{y}, \mathrm{t}- \mathrm{c}_0^{-1}|\mathrm{z}_i-\mathrm{z}_j|\big)\;|\mathrm{z}_i-\mathrm{z}_j|\;\frac{(\mathrm{x}-\mathrm{y})\cdot\nu_\mathrm{x}}{|\mathrm{x}-\mathrm{y}|^2}d\sigma_\mathrm{y}d\sigma_\mathrm{x}
    \\ \nonumber&+ \frac{\rho_\mathrm{m}}{4\pi}\frac{\mathrm{c}_0^{-2}}{2}\sum\limits_{\substack{j=1 \\ j\neq i}}^M\alpha_j\int_{\partial\Omega_i}\int_{\partial \Omega_j} \partial^2_\mathrm{t}\partial_\nu u^{(j)}\big(\mathrm{y}, \mathrm{t}- \mathrm{c}_0^{-1}|\mathrm{z}_i-\mathrm{z}_j|\big)\;\frac{(\mathrm{x}-\mathrm{y})\cdot\nu_\mathrm{x}}{|\mathrm{x}-\mathrm{y}|}\;d\sigma_\mathrm{y}d\sigma_\mathrm{x}
    \\ \nonumber&+ \underbrace{\frac{\rho_\mathrm{m}}{4\pi}\frac{\mathrm{c}_0^{-3}}{3!}\sum\limits_{\substack{j=1 \\ j\neq i}}^M\alpha_j\int_{\partial\Omega_i}\int_{\partial \Omega_j} \partial^3_\mathrm{t}\partial_\nu u^{(j)}(\mathrm{y}, \eta_2)\;|\mathrm{z}_i-\mathrm{z}_j|^3\;\frac{(\mathrm{x}-\mathrm{y})\cdot\nu_\mathrm{x}}{|\mathrm{x}-\mathrm{y}|^3}d\sigma_\mathrm{y}d\sigma_\mathrm{x}}_{=\; 0}
    \\ \nonumber&- \frac{\rho_\mathrm{m}}{4\pi}\frac{\mathrm{c}_0^{-3}}{2}\sum\limits_{\substack{j=1 \\ j\neq i}}^M\alpha_j\int_{\partial\Omega_i}\int_{\partial \Omega_j} \partial^2_\mathrm{t}\partial_\nu u^{(j)}\big(\mathrm{y}, \eta_2\big)\;|\mathrm{z}_i-\mathrm{z}_j|^2\;\frac{(\mathrm{x}-\mathrm{y})\cdot\nu_\mathrm{x}}{|\mathrm{x}-\mathrm{y}|^2}d\sigma_\mathrm{y}d\sigma_\mathrm{x}
    \\ \nonumber&+ \frac{\rho_\mathrm{m}}{4\pi}\frac{\mathrm{c}_0^{-3}}{2}\sum\limits_{\substack{j=1 \\ j\neq i}}^M\alpha_j\int_{\partial\Omega_i}\int_{\partial \Omega_j} \partial^2_\mathrm{t}\partial_\nu u^{(j)}\big(\mathrm{y}, \eta_2\big)\;|\mathrm{z}_i-\mathrm{z}_j|\;\frac{(\mathrm{x}-\mathrm{y})\cdot\nu_\mathrm{x}}{|\mathrm{x}-\mathrm{y}|}\;d\sigma_\mathrm{y}d\sigma_\mathrm{x}
    \\ &+ \frac{\rho_\mathrm{m}}{4\pi}\frac{\mathrm{c}_0^{-3}}{3!}\sum\limits_{\substack{j=1 \\ j\neq i}}^M\alpha_j\int_{\partial\Omega_i}\int_{\partial \Omega_j} \partial^2_\mathrm{t}\partial_\nu u^{(j)}\big(\mathrm{y}, \eta_2\big)\;(\mathrm{x}-\mathrm{y})\cdot\nu_\mathrm{x}\;d\sigma_\mathrm{y}d\sigma_\mathrm{x},
\end{align}
where $\eta_2 \in \big(\mathrm{t}- \mathrm{c}_0^{-1}|\mathrm{x}-\mathrm{y}|,\mathrm{t}- \mathrm{c}_0^{-1}|\mathrm{z}_i-\mathrm{z}_j|\big).$
\newline
The reason behind taking the above terms to be zero is that for $\mathrm{x}\in \partial \Omega_i$, we have:
\begin{align}
 \int_{\partial\Omega_i}\int_{\partial\Omega_j}\frac{(\mathrm{x}-\mathrm{y})\cdot\nu_\mathrm{x}}{|\mathrm{x}-\mathrm{y}|^3}\partial_\nu u^{(j)}(\mathrm{y}, \mathrm{t}-\mathrm{c}_0^{-1}|\mathrm{z}_i-\mathrm{z}_j|) d\sigma_\mathrm{y}d\sigma_\mathrm{x} =  \Big\langle \bm{\mathcal{K}}^{*^{(ij)}}_{\textbf{lap}} \Big[\partial_\nu u^{(j)}\Big]; 1\Big\rangle_{\partial\Omega_i}
 =  \Big\langle\partial_\nu u^{(j)};\underbrace{\bm{\mathcal{K}}^{(ij)}_{\textbf{lap}} \Big[1\Big]}_{=0}\Big\rangle = 0.
\end{align}
Therefore, from the estimates (\ref{firstestimate}) and (\ref{secondestimate}), we derive that
\begin{align}\label{result}
    \sum\limits_{\substack{j=1 \\ j\neq i}}^M\alpha_j\int_{\partial\Omega_i}\bm{\mathcal{K}}^{(ij)}\Big[\partial_\nu u^{(j)}\Big](\mathrm{x},\mathrm{t})d\sigma_\mathrm{x} \nonumber&= \frac{\rho_\mathrm{m}}{4\pi}\frac{\mathrm{c}_0^{-2}}{2}\sum\limits_{\substack{j=1 \\ j\neq i}}^M\alpha_j\int_{\partial\Omega_i}\int_{\partial \Omega_j} \partial^2_\mathrm{t}\partial_\nu u^{(j)}\big(\mathrm{y}, \mathrm{t}- \mathrm{c}_0^{-1}|\mathrm{z}_i-\mathrm{z}_j|\big)\;\frac{(\mathrm{x}-\mathrm{y})\cdot\nu_\mathrm{x}}{|\mathrm{x}-\mathrm{y}|}\;d\sigma_\mathrm{y}d\sigma_\mathrm{x} 
    \\ &- \frac{\rho_\mathrm{m}}{4\pi}\frac{\mathrm{c}_0^{-2}}{2}\sum\limits_{\substack{j=1 \\ j\neq i}}^M\alpha_j\int_{\partial\Omega_i}\int_{\partial \Omega_j} \partial^3_\mathrm{t}\partial_\nu u^{(j)}\big(\mathrm{y}, \eta\big)\;(\mathrm{x}-\mathrm{y})\cdot\nu_\mathrm{x}\;d\sigma_\mathrm{y}d\sigma_\mathrm{x},
\end{align}
where $\eta = \max(\eta_1,\eta_2).$\par
We state the following two Propositions that will be useful to estimate the reminder part of the previous expression.
\begin{proposition}\label{p1}
For $u^{(i)} = u^\textbf{in}+u^\mathrm{s}$ as the solution of (\ref{math-modelmulti}) we have the following estimates
\begin{align}
\Vert \partial_\mathrm{t}^\mathrm{n}u^{(i)}(\cdot,\mathrm{t})\Vert_{\mathrm{L}^2(\Omega_i)} \lesssim \delta^{\frac{3}{2}}, \quad \mathrm{t}\in [0,\mathrm{T}], \quad \mathrm{n}=0,1,\ldots.
\end{align}
\end{proposition}
\begin{proof}
See \cite[Sec. 3]{AM2} for the proof.
\end{proof}
and
\begin{proposition}\label{p3}
We have the following estimate
\begin{align}
\Vert \partial_\mathrm{t}^\mathrm{n}\partial_\nu u^{(i)}(\cdot, \mathrm{t})\Vert_{\mathrm{H}^{-\frac{1}{2}}(\partial\Omega_i)} \sim \delta^2, \quad \mathrm{t}\in [0,\mathrm{T}], \quad \text{for}\quad \mathrm{n} =0,1,\ldots.
\end{align}
\end{proposition}
\begin{proof}
See Section \ref{apri} for the proof.
\end{proof}
\noindent
By utilizing the equations (\ref{apr1multi}), (\ref{apr3}), and (\ref{result}) in (\ref{aprmulti}) and employing similar mathematical calculations as presented in \cite{AM2}, we arrive at the following algebraic system:
\begin{align}\label{algebricsystem}
&\nonumber\Big(1 - \gamma_i\frac{\mathrm{k}_\mathrm{c}\rho_\mathrm{m}}{\rho_\mathrm{c}}\Big)\int_{\partial\Omega_i} \partial_\nu u^{(i)} + \frac{\alpha_i\rho_\mathrm{m}}{2} \mathrm{A}_{\partial\Omega_i} \mathrm{c}_0^{-2}\int_{\partial\Omega_i}\partial_t^2\partial_\nu u^{(i)} \\ \nonumber&= \int_{\partial\Omega_i} \partial_\nu u^\textbf{in} \textcolor{black}{-} \frac{1}{2} \frac{\rho_\mathrm{m}}{\mathrm{c}^{2}_0}\sum\limits_{\substack{j=1 \\ j\neq i}}^M\alpha_j\int_{\partial\Omega_i}\int_{\partial\Omega_j}\ \frac{(\mathrm{x}-\mathrm{y})\cdot\nu_\mathrm{x}}{|\mathrm{x}-\mathrm{y}|}\partial_\mathrm{t}^2\partial_\nu u^{(j)}(\mathrm{y}, \mathrm{t}- \mathrm{c}_0^{-1}|\mathrm{z}_i-\mathrm{z}_j|) d\sigma_\mathrm{y}d\sigma_\mathrm{x} \\ &- \underbrace{\frac{1}{3} \frac{\rho_\mathrm{m}}{\mathrm{c}^{3}_0}\sum\limits_{\substack{j=1 \\ j\neq i}}^M\alpha_j\int_{\partial\Omega_i}\int_{\partial\Omega_j} (\mathrm{x}-\mathrm{y})\cdot\nu_\mathrm{x}\partial_\mathrm{t}^3\partial_\nu u^{(j)}(\mathrm{y}, \eta)d\sigma_\mathrm{y}d\sigma_\mathrm{x}}_{:=\; \mathrm{S}_5} +  \underbrace{\sum\limits_{\substack{j=1 \\ j\neq i}}^M\gamma_j \int_{\partial\Omega_i}\partial_\nu\bm{\mathrm{V}}^{(ij)}\Big[u^{(j)}_{\mathrm{t}\mathrm{t}}\Big]}_{:=\; \mathrm{S}_6} + \mathcal{O}(\delta^4).
\end{align}
Next, we turn our attention to the second term on the right-hand side of the equation. Hence we obtain with an exchange of  the order of integration :
\begin{align}\label{coupledde}
 \nonumber& \frac{1}{2} \frac{\rho_\mathrm{m}}{\mathrm{c}^{2}_0}\sum\limits_{\substack{j=1 \\ j\neq i}}^M\alpha_j\int_{\partial\Omega_i}\int_{\partial\Omega_j}\ \frac{(\mathrm{x}-\mathrm{y})\cdot\nu_\mathrm{x}}{|\mathrm{x}-\mathrm{y}|}\partial_\mathrm{t}^2\partial_\nu u^{(j)}(\mathrm{y}, \mathrm{t}- \mathrm{c}_0^{-1}|\mathrm{z}_i-\mathrm{z}_j|) d\sigma_\mathrm{y}d\sigma_\mathrm{x} \\ \nonumber &=   \frac{1}{2} \frac{\rho_\mathrm{m}}{\mathrm{c}^{2}_0}\sum\limits_{\substack{j=1 \\ j\neq i}}^M\alpha_j \int_{\partial\Omega_j}\partial_\mathrm{t}^2\partial_\nu u^{(j)}(\mathrm{y}, \mathrm{t}- \mathrm{c}_0^{-1}|\mathrm{z}_i-\mathrm{z}_j|)\int_{\partial\Omega_i}\ \frac{(\mathrm{x}-\mathrm{y})\cdot\nu_\mathrm{x}}{|\mathrm{x}-\mathrm{y}|} d\sigma_\mathrm{x} d\sigma_\mathrm{y}
 \\ \nonumber&= \frac{1}{2} \frac{\rho_\mathrm{m}}{\mathrm{c}^{2}_0}\sum\limits_{\substack{j=1 \\ j\neq i}}^M\alpha_j \int_{\partial\Omega_j}\partial_\mathrm{t}^2\partial_\nu u^{(j)}(\mathrm{y}, \mathrm{t}- \mathrm{c}_0^{-1}|\mathrm{z}_i-\mathrm{z}_j|)\int_{\Omega_i}\ \nabla\cdot\frac{(\mathrm{x}-\mathrm{y})}{|\mathrm{x}-\mathrm{y}|} d\mathrm{x} d\sigma_\mathrm{y}
 \\ \nonumber&=\frac{\rho_\mathrm{m}}{\mathrm{c}^{2}_0}\sum\limits_{\substack{j=1 \\ j\neq i}}^M\alpha_j \int_{\partial\Omega_j}\partial_\mathrm{t}^2\partial_\nu u^{(j)}(\mathrm{y}, \mathrm{t}- \mathrm{c}_0^{-1}|\mathrm{z}_i-\mathrm{z}_j|)\int_{\Omega_i}\frac{1}{|\mathrm{x}-\mathrm{y}|} d\mathrm{x} d\sigma_\mathrm{y}
 \\ \nonumber &=  \frac{\rho_\mathrm{m}}{\mathrm{c}^{2}_0}\sum\limits_{\substack{j=1 \\ j\neq i}}^M\alpha_j\Bigg[\delta^3\frac{1}{|\mathrm{z}_i-\mathrm{z}_j|} \int_{\partial\Omega_j}\partial_\mathrm{t}^2\partial_\nu u^{(j)}(\mathrm{y},\mathrm{t}- \mathrm{c}_0^{-1}|\mathrm{z}_i-\mathrm{z}_j|)d\sigma_\mathrm{y} \\ \nonumber& + \int_{\partial\Omega_j}\partial_\mathrm{t}^2\partial_\nu u^{(j)}(\mathrm{y}, \mathrm{t}- \mathrm{c}_0^{-1}|\mathrm{z}_i-\mathrm{z}_j|)\int_{\Omega_i}\int_0^1\mathop{\nabla}\limits_{\mathrm{x}}g(\mathrm{z}_i+\vartheta_1(x-\mathrm{z}_i);\mathrm{z}_\mathrm{j})\cdot(x-\mathrm{z}_i)\vartheta d\mathrm{x} d\sigma_\mathrm{y}
 \\ &+ \int_{\partial\Omega_j}\partial_\mathrm{t}^2\partial_\nu u^{(j)}(\mathrm{y}, \mathrm{t}- \mathrm{c}_0^{-1}|\mathrm{z}_i-\mathrm{z}_j|)\int_{\Omega_i}\int_0^1\mathop{\nabla}\limits_{\mathrm{y}}g(\mathrm{z}_\mathrm{i};\mathrm{z}_j+\vartheta_1(y-\mathrm{z}_j))\cdot(y-\mathrm{z}_j)d\vartheta d\mathrm{x} d\sigma_\mathrm{y}\Bigg].
\end{align}\vspace{-5pt}
Next, we proceed to estimate the second term of the above expression. Let us define $g(x,y):= \frac{1}{|x-y|}$ and denote the estimation of $S_3$ as
\begin{align}
    \nonumber&\textbf{Estimation of $S_3$}\\ :=&\nonumber\Bigg|\int_{\partial\Omega_j}\partial_\mathrm{t}^2\partial_\nu u^{(j)}(\mathrm{y}, \mathrm{t}- \mathrm{c}_0^{-1}|\mathrm{z}_i-\mathrm{z}_j|)\int_{\Omega_i}\int_0^1\mathop{\nabla}\limits_{\mathrm{x}}g(\mathrm{z}_i+\vartheta_1(x-\mathrm{z}_i);\mathrm{z}_\mathrm{j})\cdot(x-\mathrm{z}_i)\vartheta d\mathrm{x} d\sigma_\mathrm{y}\Bigg|
    \\ \nonumber &\lesssim \Vert\partial_\mathrm{t}^2\partial_\nu u^{(j)}(\cdot, \mathrm{t}- \mathrm{c}_0^{-1}|\mathrm{z}_i-\mathrm{z}_j|)\Vert_{\mathrm{H}^{-\frac{1}{2}}(\partial\Omega_i)}\Big\Vert\int_{\Omega_i}\int_0^1\mathop{\nabla}\limits_{\mathrm{x}}g(\mathrm{z}_i+\vartheta_1(x-\mathrm{z}_i);\mathrm{z}_\mathrm{j})\cdot(x-\mathrm{z}_i)d\vartheta d\mathrm{x}\Big\Vert_{\mathrm{H}^{\frac{1}{2}}(\partial\Omega_i)} 
     \\ \nonumber &\lesssim \Vert\partial_\mathrm{t}^2\partial_\nu u^{(j)}(\cdot, \mathrm{t}- \mathrm{c}_0^{-1}|\mathrm{z}_i-\mathrm{z}_j|)\Vert_{\mathrm{H}^{-\frac{1}{2}}(\partial\Omega_i)}\Big\Vert\int_{\Omega_i}\int_0^1\mathop{\nabla}\limits_{\mathrm{x}}g(\mathrm{z}_i+\vartheta_1(x-\mathrm{z}_i);\mathrm{z}_\mathrm{j})\cdot(x-\mathrm{z}_i)d\vartheta d\mathrm{x}\Big\Vert_{\mathrm{L}^2(\partial\Omega_i)} 
     \\  &\lesssim \Vert\partial_\mathrm{t}^2\partial_\nu u^{(j)}(\cdot, \mathrm{t}- \mathrm{c}_0^{-1}|\mathrm{z}_i-\mathrm{z}_j|)\Vert_{\mathrm{H}^{-\frac{1}{2}}(\partial\Omega_i)} \delta^5 \mathrm{d}_{ij}^{-2} \sim \delta^7\mathrm{d}_{ij}^{-2}.
\end{align}
Likewise, we define the estimation of $S_4$ as
\begin{align}
    \nonumber\textbf{Estimation of $S_4$}:=& \Bigg| \int_{\partial\Omega_j}\partial_\mathrm{t}^2\partial_\nu u^{(j)}(\mathrm{y}, \mathrm{t}- \mathrm{c}_0^{-1}|\mathrm{z}_i-\mathrm{z}_j|) \int_{\Omega_i}\int_0^1\mathop{\nabla}\limits_{\mathrm{y}}g(\mathrm{z}_\mathrm{i};\mathrm{z}_j+\vartheta_1(y-\mathrm{z}_j))\cdot(y-\mathrm{z}_j)d\vartheta d\mathrm{x}d\sigma_\mathrm{y} \Bigg| \\ &\lesssim \delta^7\mathrm{d}_{ij}^{-2}. 
\end{align}
Hence, by substituting the two previously estimated quantities into equation (\ref{coupledde}), we obtain the following equation:
\begin{align}\label{emul2.28}
    \nonumber &\frac{1}{2} \frac{\rho_\mathrm{m}}{\mathrm{c}^{2}_0}\sum\limits_{\substack{j=1 \\ j\neq i}}^M\alpha_j\int_{\partial\Omega_i}\int_{\partial\Omega_j}\ \frac{(\mathrm{x}-\mathrm{y})\cdot\nu_\mathrm{x}}{|\mathrm{x}-\mathrm{y}|}\partial_\mathrm{t}^2\partial_\nu u^{(j)}(\mathrm{y}, \mathrm{t}- \mathrm{c}_0^{-1}|\mathrm{z}_i-\mathrm{z}_j|) d\sigma_\mathrm{y}d\sigma_\mathrm{x} \\ &=  \textcolor{black}{\rho_m} \mathrm{c}^{-2}_0\sum\limits_{\substack{j=1 \\ j\neq i}}^M\alpha_j\delta^3\frac{1}{|\mathrm{z}_i-\mathrm{z}_j|} \int_{\partial\Omega_j}\partial_\mathrm{t}^2\partial_\nu u^{(j)}(\mathrm{y}, \mathrm{t}- \mathrm{c}_0^{-1}|\mathrm{z}_i-\mathrm{z}_j|)d\sigma_\mathrm{y} + \mathcal{O}\Big(\delta^5 \sum\limits_{\substack{j=1 \\ j\neq i}}^M \mathrm{d}_{ij}^{-2}\Big).
\end{align}

We proceed by estimating $S_5$ as follows:
\begin{align}\nonumber
    \textbf{Estimation of $S_5$}:=&\Big|\int_{\partial\Omega_i}\int_{\partial\Omega_j} (\mathrm{x}-\mathrm{y})\cdot\nu_\mathrm{x}\partial_\mathrm{t}^3\partial_\nu u^{(j)}(\mathrm{y}, \eta)d\sigma_\mathrm{y}d\sigma_\mathrm{x}\Big|
    \\ \nonumber&= \Big|\int_{\partial\Omega_j}\partial_\mathrm{t}^3\partial_\nu u^{(j)}(\mathrm{y}, \eta) \int_{\partial\Omega_i}(\mathrm{x}-\mathrm{y})\cdot\nu_\mathrm{x}d\sigma_\mathrm{x} d\sigma_\mathrm{y}\Big|
    \\ \nonumber&= \Big|\int_{\partial\Omega_j}\partial_\mathrm{t}^3\partial_\nu u^{(j)}(\mathrm{y}, \eta) \underbrace{\int_{\Omega_i}\nabla\cdot(\mathrm{x}-\mathrm{y})d\mathrm{x}}_{=\ 3|\Omega_i|} d\sigma_\mathrm{y}\Big|
    \\ &\lesssim \delta^3\Big\langle 1;\partial_\mathrm{t}^3\partial_\nu u^{(j)}(\cdot, \eta)\Big\rangle_{\partial\Omega_j} \sim \delta^6.
\end{align}
Thus, we obtain the following expression for the final term:
\begin{align}
 \textbf{Estimation of $S_5$} = \frac{1}{3} \frac{\rho_\mathrm{m}}{\mathrm{c}^{3}_0}\sum\limits_{\substack{j=1 \\ j\neq i}}^M\alpha_j\int_{\partial\Omega_i}\int_{\partial\Omega_j} (\mathrm{x}-\mathrm{y})\cdot\nu_\mathrm{x}\partial_\mathrm{t}^3\partial_\nu u^{(j)}(\mathrm{y}, \eta)d\sigma_\mathrm{y}d\sigma_\mathrm{x} \sim \delta^4.   
\end{align}
\begin{align}\label{s_4}
&\nonumber\textbf{Estimation of $S_6$}
\\ \nonumber&:=
 \sum\limits_{\substack{j=1 \\ j\neq i}}^M\gamma_j\int_{\partial\Omega_i}\partial_\nu\bm{\mathrm{V}}^{(ij)}\big[\partial^2_\mathrm{t}u^{(j)}\big](\mathrm{x},\mathrm{t})d\sigma_\mathrm{x} 
 \\ \nonumber&= \sum\limits_{\substack{j=1 \\ j\neq i}}^M\gamma_j\int_{\partial\Omega_i}\partial_{\nu_\mathrm{x}} \int_{\Omega_j}\frac{\rho_\mathrm{m} }{4\pi|\mathrm{x}-\mathrm{y}|}\partial_\mathrm{t}^{2}u^{(j)}(\mathrm{y},\mathrm{t}-\mathrm{c}_0^{-1}|\mathrm{x}-\mathrm{y}|)d\mathrm{y}d\sigma_\mathrm{x}
 \\ \nonumber&= \sum\limits_{\substack{j=1 \\ j\neq i}}^M\gamma_j\int_{\Omega_i}\nabla_\mathrm{x}\cdot \int_{\Omega_j}\frac{\rho_\mathrm{m} }{4\pi|\mathrm{x}-\mathrm{y}|}\partial_\mathrm{t}^{2}u^{(j)}(\mathrm{y},\mathrm{t}-\mathrm{c}_0^{-1}|\mathrm{x}-\mathrm{y}|)d\mathrm{y}d\mathrm{x}
   \\ \nonumber&=\Bigg|\frac{\rho_\mathrm{m}}{4\pi}\sum\limits_{\substack{j=1 \\ j\neq i}}^M\gamma_j\int_{\Omega_i}\nabla_\mathrm{x}\cdot\int_{\Omega_j}\frac{1}{|\mathrm{x}-\mathrm{y}|}\partial_\mathrm{t}^{2}u^{(j)}(\mathrm{y},\mathrm{t}-\mathrm{c}_0^{-1}|\mathrm{z}_i-\mathrm{z}_j|)d\mathrm{y}d\mathrm{s}\Bigg| 
   \\ &+ \Bigg|\frac{\rho_\mathrm{m}}{4\pi}\sum\limits_{\substack{j=1 \\ j\neq i}}^M\gamma_j\int_{\Omega_i}\nabla_\mathrm{x}\cdot\int_{\Omega_j}\frac{1}{|\mathrm{x}-\mathrm{y}|}\Big(\partial_\mathrm{t}^{2}u^{(j)}(\mathrm{y},\mathrm{t}-\mathrm{c}_0^{-1}|\mathrm{z}_i-\mathrm{z}_j|)d\mathrm{y}-\partial_\mathrm{t}^{2}u^{(j)}(\mathrm{y},\mathrm{t}-\mathrm{c}_0^{-1}|\mathrm{x}-\mathrm{y}|)\Big)d\mathrm{y}d\mathrm{x}.
\end{align}
Let us now, give a close look on the second term of the above expression to deduce
\begin{align}
  &\nonumber\Bigg|\frac{\rho_\mathrm{m}}{4\pi}\sum\limits_{\substack{j=1 \\ j\neq i}}^M\gamma_j\int_{\Omega_i}\nabla_\mathrm{x}\cdot\int_{\Omega_j}\frac{1}{|\mathrm{x}-\mathrm{y}|}\Big(\partial_\mathrm{t}^{2}u^{(j)}(\mathrm{y},\mathrm{t}-\mathrm{c}_0^{-1}|\mathrm{z}_i-\mathrm{z}_j|)d\mathrm{y}-\partial_\mathrm{t}^{2}u^{(j)}(\mathrm{y},\mathrm{t}-\mathrm{c}_0^{-1}|\mathrm{x}-\mathrm{y}|)\Big)d\mathrm{y}d\mathrm{x} \Bigg|
  \\ \nonumber&= \Bigg|\frac{\rho_\mathrm{m}}{4\pi}\sum\limits_{\substack{j=1 \\ j\neq i}}^M\gamma_j\int_{\Omega_i}\nabla_\mathrm{x}\cdot\int_{\Omega_j}\frac{1}{|\mathrm{x}-\mathrm{y}|}\Big(\partial_\mathrm{t}^{2}u^{(j)}(\mathrm{y},\mathrm{t}-\mathrm{c}_0^{-1}|\mathrm{z}_i-\mathrm{z}_j|)d\mathrm{y}-\partial_\mathrm{t}^{2}u^{(j)}(\mathrm{y},\mathrm{t}-\mathrm{c}_0^{-1}|\mathrm{x}-\mathrm{z}_j|)\Big)d\mathrm{y}d\mathrm{x}\Bigg|
  \\ \nonumber&+ \Bigg|\frac{\rho_\mathrm{m}}{4\pi}\sum\limits_{\substack{j=1 \\ j\neq i}}^M\gamma_j\int_{\Omega_i}\nabla_\mathrm{x}\cdot\int_{\Omega_j}\frac{1}{|\mathrm{x}-\mathrm{y}|}\Big(\partial_\mathrm{t}^{2}u^{(j)}(\mathrm{y},\mathrm{t}-\mathrm{c}_0^{-1}|\mathrm{x}-\mathrm{z}_j|)d\mathrm{y}-\partial_\mathrm{t}^{2}u^{(j)}(\mathrm{y},\mathrm{t}-\mathrm{c}_0^{-1}|\mathrm{x}-\mathrm{y}|)\Big)d\mathrm{y}d\mathrm{x}\Bigg|
  \\ &\lesssim \delta^7 \sum\limits_{\substack{j=1 \\ j\neq i}}^M \mathrm{d}_{ij}^{-2}.
\end{align}
Therefore, based on the above estimate and (\ref{s_4}) we obtain
\begin{align}
    \textbf{Estimation of $S_6$} = \sum\limits_{\substack{j=1 \\ j\neq i}}^M\gamma_j\int_{\partial\Omega_i}\partial_\nu\bm{\mathrm{V}}^{(ij)}\big[\partial^2_\mathrm{t}u^{(j)}\big](\mathrm{x},\mathrm{t})d\sigma_\mathrm{x} \lesssim \delta^7 \sum\limits_{\substack{j=1 \\ j\neq i}}^M \mathrm{d}_{ij}^{-2}.
\end{align}

By substituting the estimated values of (\ref{emul2.28}), $S_5$, and $S_6$ into (\ref{algebricsystem}), we derive the following equation:
\begin{align}
\nonumber&\Big(1 - \gamma_i\frac{\mathrm{k}_\mathrm{c}\rho_\mathrm{m}}{\rho_\mathrm{c}}\Big)\int_{\partial\Omega_i} \partial_\nu u^{(i)} + \frac{\alpha_i\rho_\mathrm{m}}{2} \mathrm{A}_{\partial\Omega_i} \mathrm{c}_0^{-2}\int_{\partial\Omega_i}\partial_t^2\partial_\nu u^{(i)} \\ & \textcolor{black}{+} \frac{1}{2} \frac{\rho_\mathrm{m}}{\mathrm{c}^{2}_0}\sum\limits_{\substack{j=1 \\ j\neq i}}^M\alpha_j\delta^3\frac{1}{|\mathrm{z}_i-\mathrm{z}_j|} \int_{\partial\Omega_j}\partial_\mathrm{t}^2\partial_\nu u^{(j)}(\mathrm{y}, \mathrm{t}-\mathrm{c}_0^{-1}|\mathrm{x}-\mathrm{z}_j|)d\sigma_\mathrm{y} = \int_{\partial\Omega_i} \partial_\nu u^\textbf{in} + \delta^5 \sum\limits_{\substack{j=1 \\ j\neq i}}^M \mathrm{d}_{ij}^{-2} + \delta^7 \sum\limits_{\substack{j=1 \\ j\neq i}}^M \mathrm{d}_{ij}^{-2}+ \mathcal{O}(\delta^4).  
\end{align}
We now recall that the coefficients $\rho_i$ and $\mathrm{k}_i$ are piece-wise constants, with one constant being outside $\Omega_i$. That is, $\rho(\mathrm{x})\equiv \rho_\mathrm{m}$ and $\mathrm{k}(\mathrm{x}) \equiv \mathrm{k}_\mathrm{m}$, and other constants $\rho_\mathrm{c}$ and $\mathrm{k}_\mathrm{c}$ inside $\Omega_i$, satisfying the following scaling properties:
\begin{align}\label{scalingmulti}
    \rho_\mathrm{c} = \overline{\rho}_\mathrm{c}\delta^2, \quad \mathrm{k}_\mathrm{c} = \overline{\mathrm{k}}_\mathrm{c}\delta^2 \quad \text{and} \quad \frac{\rho_\mathrm{c}}{\mathrm{k}_\mathrm{c}} \sim 1 \ \text{as}\ \delta\ll1.
\end{align}
Therefor, after a short calculation, we obtain $1 - \gamma_i\frac{\mathrm{k}_\mathrm{c}\rho_\mathrm{m}}{\rho_\mathrm{c}} = \frac{\mathrm{k}_\mathrm{c}}{\rho_\mathrm{c}}\frac{1}{\mathrm{c}_0^2}$, where $\gamma_i:=\beta_i-\alpha_i \frac{\rho_\mathrm{c}}{\mathrm{k}_\mathrm{c}}$ with $\alpha_i:=\frac{1}{\rho_\mathrm{c}}-\frac{1}{\rho_\mathrm{m}}$ and $\beta_i:= \frac{1}{\mathrm{k}_\mathrm{c}}-\frac{1}{\mathrm{k}_\mathrm{m}}.$
\newline

Next, the vector $\displaystyle\Big(\int_{\partial\Omega_i} \partial_\nu u^{(i)}\Big)_{i=1}^\mathrm{M}$ satisfies the following differential equation:
\begin{align}
    \frac{\rho_\mathrm{m}}{2}\alpha_i\frac{\rho_\mathrm{c}}{\mathrm{k}_\mathrm{c}}\mathrm{A}_{\partial\Omega_i}\int_{\partial\Omega_i}\partial_t^2\partial_\nu u^{(i)} \nonumber&\textcolor{black}{+}\textcolor{black}{\rho_\mathrm{m}}\alpha_j\delta^3\frac{\rho_\mathrm{c}}{\mathrm{k}_\mathrm{c}}\sum\limits_{\substack{j=1 \\ j\neq i}}^M\frac{1}{|\mathrm{z}_i-\mathrm{z}_j|} \int_{\partial\Omega_j}\partial_\mathrm{t}^2\partial_\nu u^{(j)}(\mathrm{y},\mathrm{t}-\mathrm{c}_0^{-1}|\mathrm{z}_i-\mathrm{z}_j|) + \int_{\partial\Omega_i} \partial_\nu u^{(i)} \\ &= \frac{\rho_\mathrm{c}}{\mathrm{k}_\mathrm{c}}\mathrm{c}_0^2\int_{\partial\Omega_i} \partial_\nu u^\textbf{in} + \textbf{Error},
\end{align}
where 
\begin{align}
    \textbf{Error}:= \delta^5 \sum\limits_{\substack{j=1 \\ j\neq i}}^M \mathrm{d}_{ij}^{-2} + \delta^7 \sum\limits_{\substack{j=1 \\ j\neq i}}^M \mathrm{d}_{ij}^{-2} + \mathcal{O}(\delta^4).
\end{align}
In order to proceed, it is imperative to establish the existence of a solution to the aforementioned system of differential equations.

\subsection{Solvability of the System of Differential Equations}  % Sub-Section

Let us first recall that $\bm{\mathrm{Y}}:=\big(\mathrm{Y}_i\big)_{i=1}^\mathrm{M}$ is the vector solution to the following non-homogeneous second-order matrix differential equation with initial conditions:
\begin{align}\label{matrixmulti1}
\begin{cases}\displaystyle
    \mathrm{d}_i\frac{\mathrm{d}^2}{\mathrm{d}\mathrm{t}^2}\mathrm{Y}_i(\mathrm{t}) + \mathrm{Y}_i(\mathrm{t}) \textcolor{black}{+} \sum\limits_{\substack{j=1 \\ j\neq i}}^M\mathrm{q}_{ij} \frac{\mathrm{d}^2}{\mathrm{d}\mathrm{t}^2}\mathrm{Y}_j(\mathrm{t}-\mathrm{c}_0^{-1}|\mathrm{z}_i-\mathrm{z}_j|) = \frac{\rho_\mathrm{c}}{\mathrm{k}_\mathrm{c}}\mathrm{c}_0^2\int_{\partial\Omega_i} \partial_\nu u^\textbf{in} \mbox{ in } (0, \mathrm{T}),
     \\ \mathrm{Y}_i(\mathrm{0}) = \frac{\mathrm{d}}{\mathrm{d}\mathrm{t}}\mathrm{Y}_i(\mathrm{0}) = 0,
\end{cases}
\end{align}
where $\mathrm{d}_i$ is defined as $\mathrm{d}_i:= \frac{\rho_\mathrm{m}}{2}\alpha_i\frac{\rho_\mathrm{c}}{\mathrm{k}_\mathrm{c}} \mathrm{A}_{\partial\Omega_i}$. Furthermore, $\bm{\mathrm{Q}}=\big(\mathrm{q}_{ij}\big)_{i,j=1}^\mathrm{M}$ is given by
\begin{align}\label{interaction-matrix1}
\mathrm{q}_{ij}=\begin{cases}
0 & \text{for}\ i=j,\\
\frac{\mathrm{b}_j}{|\mathrm{z}_i-\mathrm{z}_j|} & \text{for}\ i\ne j,
\end{cases}
\end{align}
where $\mathrm{b}_j$ is defined as $\mathrm{b}_j:= \frac{\rho_\mathrm{m}}{2}\alpha_j\delta^3\frac{\rho_\mathrm{c}}{\mathrm{k}_\mathrm{c}}$. We also define $\bm{\mathrm{B}}=\big(\mathrm{B}_i\big)_{i=1}^\mathrm{M}:= \displaystyle\Big(\int_{\partial\Omega_i} \partial_\nu u^\textbf{in}\Big)_{i=1}^\mathrm{M}$, and \\ $\displaystyle\mathrm{A}_{\partial\Omega_i} = \frac{1}{|\partial \Omega_i|}\int_{\partial\Omega_i}\int_{\partial\Omega_i}\frac{(\mathrm{x}-\mathrm{y})\cdot\nu_\mathrm{x}}{|\mathrm{x}-\mathrm{y}|}d\sigma_\mathrm{x}d\sigma_\mathrm{y}.$
\newline
In order to discuss the existence and uniqueness of the solution of the equation (\ref{matrixmulti1}), we rewrite the system above as follows:
\begin{align} \label{matrix2}
    \begin{cases}
             \mathcal{A}\frac{\mathrm{d}^2}{\mathrm{d}\mathrm{t}^2}\bm{\mathrm{Y}}(\mathrm{t}) + \bm{\mathrm{Y}}(\mathrm{t}) =  \bm{\mathrm{B}}(\mathrm{t}) \mbox{ in } (0, \mathrm{T}),
             \\ \bm{\mathrm{Y}}(\mathrm{0}) = \frac{\mathrm{d}}{\mathrm{d}\mathrm{t}}\bm{\mathrm{Y}}(\mathrm{0}) = 0,   
    \end{cases}
\end{align}
where, we define the operator $\mathcal{A}: (\mathrm{L}_\ell^2)^M\to (\mathrm{L}_\ell^2)^M$ as
\begin{align}
    \mathcal{A} = \mathcal{A}(t):=
     \begin{pmatrix}
        \mathrm{d}_1 & \dots  & \mathrm{q}_{1\mathrm{M}}\mathcal{T}_{-\mathrm{c}_0^{-1}|\mathrm{z}_i-\mathrm{z}_j|}\\
       \vdots & \ddots & \vdots\\
       \mathrm{q}_{\mathrm{M}1}\mathcal{T}_{-\mathrm{c}_0^{-1}|\mathrm{z}_i-\mathrm{z}_j|} & \dots  & \mathrm{d}_\mathrm{M} 
    \end{pmatrix}, 
\end{align}
with $\mathrm{L}_\ell^2:= \{ f \in \mathrm{L}^2(-\ell,\mathrm{T}): f=0\; \text{in}\ (-\ell,0)\}$, the translation operators $\mathcal{T}_{-\mathrm{c}_0^{-1}|\mathrm{z}_i-\mathrm{z}_j|}$, i.e. $\mathcal{T}_{-\mathrm{c}_0^{-1}|\mathrm{z}_i-\mathrm{z}_j|}(f)(t):=f(t-\mathrm{c}_0^{-1}|\mathrm{z}_i-\mathrm{z}_j|)$, and $\ell:=\max_{i\neq j}{\mathrm{c}^{-1}_0\vert z_i- z_j\vert}$.
\newline
Using similar techniques as discussed in \cite[Lemma 3.3]{sini-haibing-yao}, it can be shown that $\mathcal{A}$ is an invertible operator from $\mathrm{L}_\ell^2\to \mathrm{L}_\ell^2$, under the condition \textcolor{black}{$\max\limits_{1\le i\le M}\sum\limits_{\substack{j=1 \\ j\neq i}}^M\mathrm{q}_{ij}<d_i$ for $i=1,2,...,M$}. As a result, the system of differential equations reduces to the following system:
\begin{align} \label{matrix3}
    \begin{cases}
             \frac{\mathrm{d}^2}{\mathrm{d}\mathrm{t}^2}\bm{\mathrm{Y}}(\mathrm{t}) + \mathcal{A}^{-1}\bm{\mathrm{Y}}(\mathrm{t}) =  \mathcal{A}^{-1}\bm{\mathrm{B}}(\mathrm{t}) \mbox{ in } (0, \mathrm{T}),
             \\ \bm{\mathrm{Y}}(\mathrm{0}) = \frac{\mathrm{d}}{\mathrm{d}\mathrm{t}}\bm{\mathrm{Y}}(\mathrm{0}) = 0.  
    \end{cases}
\end{align}
Next, we can reduce the above second-order system of ordinary differential equations to a system of first-order ordinary differential equations by considering the unknown  $\bm{\mathrm{Z}}(\mathrm{t}) = \begin{pmatrix}
    \bm{\mathrm{Y}}(\mathrm{t}) \\
    \frac{d}{d\mathrm{t}}\bm{\mathrm{Y}}(\mathrm{t})
\end{pmatrix}$. Therefore, with this we obtain the following system of first order non-homogeneous ordinary differential equations with the zero initial condition:
\begin{align}\label{reduced-ode}
    \begin{cases}
       \frac{d}{d\mathrm{t}}\bm{\mathrm{Z}}(\mathrm{t}) = \begin{pmatrix}
        0 & \mathrm{I} \\
        - \bm{\mathcal{A}}^{-1} & 0
    \end{pmatrix}\bm{\mathrm{Z}}(\mathrm{t}) + \begin{pmatrix}
        0 \\
        \mathcal{A}^{-1}\bm{\mathrm{B}}(\mathrm{t})
    \end{pmatrix}\\
        \bm{\mathrm{Z}}(0) = \bm{0}.
    \end{cases}
\end{align}
Let us now denote the operator $\bm{\mathrm{R}} := \begin{pmatrix}
        0 & \mathrm{I} \\
        - \bm{\mathcal{A}}^{-1} & 0
    \end{pmatrix}.$  In the next step, we show that $\bm{\mathrm{R}}$ is a bounded operator. To show that the operator $\mathbf{R} = \begin{pmatrix} 0 & \mathbf{I} \\ -\bm{\mathcal{A}}^{-1} & 0 \end{pmatrix}$ is bounded, where $\bm{\mathcal{A}}^{-1}$ is a bounded operator from $L_\ell^2$ to $\mathrm{L}_\ell^2$, we need to demonstrate that there exists a positive constant C such that $\|\mathbf{R} \mathbf{x}\|_{\mathrm{L}_\ell^2} \leq C \|\mathbf{x}\|_{\mathrm{L}_\ell^2}$ for all $\mathbf{x}$ in $\mathrm{L}_\ell^2$.
\newline
Let $\mathbf{x} = \begin{pmatrix} \mathbf{x}_1 \\ \mathbf{x}_2 \end{pmatrix}$ be an arbitrary vector in $\mathrm{L}^2_\ell$, where $\mathbf{x}_1$ and $\mathbf{x}_2$ are elements in the corresponding sub-spaces. Applying the operator $\mathbf{R}$ to $\mathbf{x}$, we have:
\begin{equation*}
\mathbf{R} \mathbf{x} = \begin{pmatrix} 0 & \mathbf{I} \\ -\bm{\mathcal{A}}^{-1} & 0 \end{pmatrix} \begin{pmatrix} \mathbf{x}_1 \\ \mathbf{x}_2 \end{pmatrix} = \begin{pmatrix} \mathbf{x}_2 \\ -\bm{\mathcal{A}}^{-1} \mathbf{x}_1 \end{pmatrix}.    
\end{equation*}
Now, let us calculate the $\mathrm{L}^2_\ell$ norm of $\mathbf{R} \mathbf{x}$:
$$
\|\mathbf{R} \mathbf{x}\|_{\mathrm{L}_\ell^2} = \left( \int^T_0 \|\begin{pmatrix} \mathbf{x}_2 \\ -\bm{\mathcal{A}}^{-1} \mathbf{x}_1 \end{pmatrix}\|^2 \mathrm{d}t \right)^{1/2} = \left( \int^T_0 (\|\mathbf{x}_2\|^2 + \|\bm{\mathcal{A}}^{-1} \mathbf{x}_1\|^2) \mathrm{d}t \right)^{1/2}.
$$
Since $\bm{\mathcal{A}}^{-1}$ is a bounded operator from $\mathrm{L}^2_\ell$ to $\mathrm{L}_\ell^2$, there exists a positive constant $K$ such that $\|\bm{\mathcal{A}}^{-1} \mathbf{x}_1\|_{\mathrm{L}^2_\ell} \leq K \|\mathbf{x}_1\|_{\mathrm{L}^2_\ell}$ for all $\mathbf{x}_1$ in $\mathrm{L}^2_\ell$. Therefore, we can write:
$$
\|\mathbf{R} \mathbf{x}\|_{\mathrm{L}^2_\ell} = \left( \int_0^T (\|\mathbf{x}_2\|^2 + \|\bm{\mathcal{A}}^{-1} \mathbf{x}_1\|_{\mathrm{L}^2_\ell}2) \mathrm{d}t \right)^{1/2} \leq \left( \int_0^T (\|\mathbf{x}_2\|^2 + K^2 \|\mathbf{x}_1\|^2) \mathrm{d}t \right)^{1/2}.
$$
By the triangle inequality,  we have $\|\mathbf{R} \mathbf{x}\|_{\mathrm{L}^2_\ell}^2 \leq (\|\mathbf{x}_2\|_{\mathrm{L}^2_\ell}^2 + K \|\mathbf{x}_1\|_{\mathrm{L}^2_\ell}^2)$. Then, we can write $\|\mathbf{R} \mathbf{x}\|_{\mathrm{L}^2_\ell}^2 \leq C \|\mathbf{x}\|_{\mathrm{L}^2_\ell}^2$ for all $\mathbf{x}$ in $\mathrm{L}_\ell^2$, where $C$ is a positive constant. This establishes that the operator $\mathbf{R}$ is bounded as $\bm{\mathcal{A}}^{-1}$ is a bounded operator from $\mathrm{L}_\ell^2$ to $\mathrm{L}_\ell^2$. In addition, it satisfies the hypothesis in Lemma 3.5.2 in \cite{ode}. Hence it ensures the existence, uniqueness and continuity, with respect to source term, of the solution to (\ref{reduced-ode}) and then of the one of (\ref{matrix2}) (or (\ref{matrixmulti1})). \qed
%------------------------------------------------------------------------------------
%------------------------------------------------------------------------------------

\subsection{End of the Proof of Theorem \ref{mainthmulti}}\label{End-proof}

We start with the following integral representation for $(\mathrm{x},\mathrm{t})\in \mathbb R^3\setminus\overline{\cup_{i=1}^\mathrm{M}\Omega_i}\times(0,\mathrm{T})$
\begin{align}\label{approximation}
    &\nonumber u^\mathrm{s}(\mathrm{x},\mathrm{t})
   \\ &=\nonumber -  \sum_{i=1}^\mathrm{M}\gamma_i\int_{\Omega_i}\frac{\rho_\mathrm{m}}{4\pi|\mathrm{x}-\mathrm{y}|} u^{(i)}_{\mathrm{t}\mathrm{t}}(\mathrm{y},\mathrm{t}-\mathrm{c}_0^{-1}|\mathrm{x}-\mathrm{y}|)d\mathrm{y} - \sum_{i=1}^\mathrm{M}\alpha_i \int_{\partial \Omega_i} \frac{\rho_\mathrm{m}}{4\pi|\mathrm{x}-\mathrm{y}|}\partial_\nu u^{(i)}(\mathrm{y}, \mathrm{t}- \mathrm{c}_0^{-1}|\mathrm{x}-\mathrm{y}|)d\sigma_\mathrm{y}
    \\ \nonumber &= \sum_{i=1}^\mathrm{M}\frac{\rho_\mathrm{m}\gamma_i}{4\pi|\mathrm{x}-\mathrm{z}_i|}\int_{\Omega_i} u^{(i)}_{\mathrm{t}\mathrm{t}}(\mathrm{y},\mathrm{t}-\mathrm{c}_0^{-1}|\mathrm{x}-\mathrm{z}_i|)d\mathrm{y} 
    + \sum_{i=1}^\mathrm{M}\frac{\rho_\mathrm{m}\gamma_i}{4\pi} \int_{\Omega_i} u^{(i)}_{\mathrm{t}\mathrm{t}}(\mathrm{y},\mathrm{t}-\mathrm{c}_0^{-1}|\mathrm{x}-\mathrm{z}_i|) \Big(\frac{1}{|\mathrm{x}-\mathrm{y}|}- \frac{1}{|\mathrm{x}-\mathrm{z}_i|}\Big)d\mathrm{y}
    \\ \nonumber &-\sum_{i=1}^\mathrm{M}\frac{\rho_\mathrm{m}\gamma_i}{4\pi} \int_{\Omega_i}\frac{u^{(i)}_{\mathrm{t}\mathrm{t}}(\mathrm{y},\mathrm{t}-\mathrm{c}_0^{-1}|\mathrm{x}-\mathrm{z}_i|)- u^{(i)}_{\mathrm{t}\mathrm{t}}(\mathrm{y},\mathrm{t}-\mathrm{c}_0^{-1}|\mathrm{x}-\mathrm{y}|)}{|\mathrm{x}-\mathrm{y}|}d\mathrm{y} 
    \\ \nonumber &+ \sum_{i=1}^\mathrm{M}\frac{\rho_\mathrm{m}\alpha_i}{4\pi} \int_{\partial\Omega_i}\frac{1}{|\mathrm{x}-\mathrm{y}|} \partial_\nu u (\mathrm{y},\mathrm{t}-\mathrm{c}_0^{-1}|\mathrm{x}-\mathrm{z}_i|) d\sigma_\mathrm{y}
    \\ \nonumber &-\sum_{i=1}^\mathrm{M}\frac{\rho_\mathrm{m}\alpha_i}{4\pi} \int_{\partial\Omega_i}\frac{\partial_\nu u^{(i)}(\mathrm{y},\mathrm{t}-\mathrm{c}_0^{-1}|\mathrm{x}-\mathrm{z}_i|)-\partial_\nu u^{(i)}(\mathrm{y},\mathrm{t}-\mathrm{c}_0^{-1}|\mathrm{x}-\mathrm{y}|)}{|\mathrm{x}-\mathrm{y}|}d\mathrm{y} 
    \\ \nonumber &= \sum_{i=1}^\mathrm{M}\frac{\rho_\mathrm{m}\gamma_i}{4\pi|\mathrm{x}-\mathrm{z}_i|}\int_{\Omega_i} u^{(i)}_{\mathrm{t}\mathrm{t}}(\mathrm{y},\mathrm{t}-\mathrm{c}_0^{-1}|\mathrm{x}-\mathrm{z}_i|)d\mathrm{y} 
    + \sum_{i=1}^\mathrm{M}\frac{\rho_\mathrm{m}\gamma_i}{4\pi} \int_{\Omega_i} u^{(i)}_{\mathrm{t}\mathrm{t}}(\mathrm{y},\mathrm{t}-\mathrm{c}_0^{-1}|\mathrm{x}-\mathrm{z}_i|) \Big(\frac{1}{|\mathrm{x}-\mathrm{y}|}- \frac{1}{|\mathrm{x}-\mathrm{z}_i|}\Big)d\mathrm{y}
    \\ \nonumber &-\sum_{i=1}^\mathrm{M}\frac{\rho_\mathrm{m}\gamma_i}{4\pi} \int_{\Omega_i}\frac{u^{(i)}_{\mathrm{t}\mathrm{t}}(\mathrm{y},\mathrm{t}-\mathrm{c}_0^{-1}|\mathrm{x}-\mathrm{z}_i|)- u^{(i)}_{\mathrm{t}\mathrm{t}}(\mathrm{y},\mathrm{t}-\mathrm{c}_0^{-1}|\mathrm{x}-\mathrm{y}|)}{|\mathrm{x}-\mathrm{y}|}d\mathrm{y} 
    \\ \nonumber &-\sum_{i=1}^\mathrm{M}\frac{\rho_\mathrm{m}\alpha_i}{4\pi} \int_{\partial\Omega_i}\frac{\partial_\nu u^{(i)}(\mathrm{y},\mathrm{t}-\mathrm{c}_0^{-1}|\mathrm{x}-\mathrm{z}_i|)-\partial_\nu u^{(i)}(\mathrm{y},\mathrm{t}-\mathrm{c}_0^{-1}|\mathrm{x}-\mathrm{y}|)}{|\mathrm{x}-\mathrm{y}|}d\sigma_\mathrm{y}
    \\ \nonumber &+ \sum_{i=1}^\mathrm{M}\frac{\rho_\mathrm{m}\alpha_i}{4\pi} \frac{1}{{|\partial\Omega_i|}}\int_{\partial\Omega_i}\frac{1}{|\mathrm{x}-\mathrm{y}|}\int_{\partial\Omega_i}\partial_\nu u(\mathrm{y},\mathrm{t}-\mathrm{c}_0^{-1}|\mathrm{x}-\mathrm{z}_i|)d\sigma_\mathrm{y} 
    \\ &+ \sum_{i=1}^\mathrm{M}\frac{\rho_\mathrm{m}\alpha_i}{4\pi} \int_{\partial\Omega_i}\Big(\frac{1}{|\mathrm{x}-\mathrm{y}|}-\frac{1}{|\partial\Omega_i|}\int_{\partial\Omega_i}\frac{1}{|\mathrm{x}-\mathrm{y}|}\Big) \partial_\nu u^{(i)} (\mathrm{y},\mathrm{t}-\mathrm{c}_0^{-1}|\mathrm{x}-\mathrm{z}_i|) d\sigma_\mathrm{y}.
\end{align}
Given $\mathrm{z}_i \in \Omega_i$ and $\mathrm{x}\in \mathbb{R}^3\setminus\overline{\bigcup_{i=1}^\mathrm{M}\Omega_i}$, the following estimates are derived using Taylor's series expansion:
\begin{align}\label{t2multi}
 u^{(i)}_{\mathrm{t}\mathrm{t}}(\mathrm{y},\mathrm{t}-\mathrm{c}_0^{-1}|\mathrm{x}-\mathrm{y}|)- u^{(i)}_{\mathrm{t}\mathrm{t}}(\mathrm{y},\mathrm{t}-\mathrm{c}_0^{-1}|\mathrm{x}-\mathrm{z}_i|) =  \mathrm{c}_0^{-1} (\mathrm{y}-\mathrm{z}_i)\nabla|\mathrm{x}-\mathrm{z}^*|\partial_\mathrm{t}^3 u^{(i)}(\mathrm{y},\mathrm{t}_0^*),
\end{align}
\begin{align}\label{ttt2}
 \partial_\nu u^{(i)}(\mathrm{y},\mathrm{t}-\mathrm{c}_0^{-1}|\mathrm{x}-\mathrm{y}|)- \partial_\nu u^{(i)}(\mathrm{y},\mathrm{t}-\mathrm{c}_0^{-1}|\mathrm{x}-\mathrm{z}_i|) =  \mathrm{c}_0^{-1} (\mathrm{y}-\mathrm{z}_i)\nabla|\mathrm{x}-\mathrm{z}^*|\partial_\mathrm{t}\partial_\nu u^{(i)}(\mathrm{y},\mathrm{t}_0^*),
\end{align}
and
\begin{align}\label{tt2multi}
\frac{1}{|\mathrm{x}-\mathrm{y}|}-\frac{1}{|\mathrm{x}-\mathrm{z}_i|} = (\mathrm{y}-\mathrm{z}_i)\nabla\frac{1}{|\mathrm{x}-\mathrm{z}^*|},
\end{align}
where, $ \mathrm{t}_0^*\in (\mathrm{t}-\mathrm{c}_0^{-1}|\mathrm{x}-\mathrm{y}|,\mathrm{t}-\mathrm{c}_0^{-1}|\mathrm{x}-\mathrm{z}_i|)$ and $\mathrm{z}^* \in \Omega_i.$ 
\newline

We proceeded to estimate the following terms, beginning with $\textbf{err}_{(1)}$:
\begin{align}\label{3.12multi}
   \textbf{err}_{(1)}:= \Big| \sum_{i=1}^\mathrm{M}\frac{\rho_\mathrm{m}\gamma_i}{4\pi|\mathrm{x}-\mathrm{z}_i|}\int_{\Omega_i} u^{(i)}_{\mathrm{t}\mathrm{t}}(\mathrm{y},\mathrm{t}-\mathrm{c}_0^{-1}|\mathrm{x}-\mathrm{z}_i|)d\mathrm{y}\Big| &\lesssim \nonumber \mathrm{M}\frac{\gamma_i\rho_\mathrm{m}}{4\pi}| \delta^\frac{3}{2}\mathrm{x}-\mathrm{z}_i|^{-1}\Vert \partial_\mathrm{t}^2u^{(i)}(\cdot,\mathrm{t})\Vert_{\mathrm{L}^2(\Omega_i)}.
\end{align}
We now consider $\bm{\xi}= \max\limits_{1\le i\le \mathrm{M}}\textbf{dist}(\mathrm{x},\mathrm{z}_i).$ Thus, we obtain
\begin{align}
    \textbf{err}_{(1)} = \mathcal{O}\big(\mathrm{M}\delta^3\bm{\xi}^{-1}\big).
\end{align}
Using the expression derived in equation (\ref{tt2multi}), we obtain:
\begin{align}\nonumber
 \textbf{err}_{(2)}&:= \Big|\sum_{i=1}^\mathrm{M}\frac{\rho_\mathrm{m}\gamma_i}{4\pi} \int_{\Omega_i}    u^{(i)}_{\mathrm{t}\mathrm{t}}(\mathrm{y},\mathrm{t}-\mathrm{c}_0^{-1}|\mathrm{x}-\mathrm{z}_i|) \Big(\frac{1}{|\mathrm{x}-\mathrm{y}|}- \frac{1}{|\mathrm{x}-\mathrm{z}_i|}\Big)d\mathrm{y}\Big| 
 \\ &\lesssim \nonumber \mathrm{M}\frac{\gamma_i\rho_\mathrm{m}}{4\pi}\Big|\nabla\frac{1}{|\mathrm{x}-\mathrm{z}^*|}\Big|\Vert \cdot-\mathrm{z}_i\Vert_{\mathrm{L}^2(\Omega_i)} \Vert \partial_\mathrm{t}^2u^{(i)}(\cdot,\mathrm{t})\Vert_{\mathrm{L}^2(\Omega_i)}
    \\ &= \mathcal{O}\Big(\mathrm{M}\delta^\frac{5}{2} |\mathrm{x}-\mathrm{z}_i|^{-2}\Vert \partial_\mathrm{t}^2u^{(i)}(\cdot,\mathrm{t})\Vert_{\mathrm{L}^2(\Omega_i)}\Big)
    = \mathcal{O}\Big(\mathrm{M}\delta^4 \bm{\xi}^{-2}\Big) .
\end{align}
We use the estimates (\ref{t2multi}) and (\ref{tt2multi}) to evaluate the following term:
\begin{align}
   \nonumber&\textbf{err}_{(3)}\\ \nonumber&:= \Big|\sum_{i=1}^\mathrm{M}\frac{\rho_\mathrm{m}\gamma_i}{4\pi} \int_{\Omega_i}\frac{u^{(i)}_{\mathrm{t}\mathrm{t}}(\mathrm{y},\mathrm{t}-\mathrm{c}_0^{-1}|\mathrm{x}-\mathrm{z}_i|)- u^{(i)}_{\mathrm{t}\mathrm{t}}(\mathrm{y},\mathrm{t}-\mathrm{c}_0^{-1}|\mathrm{x}-\mathrm{y}|)}{|\mathrm{x}-\mathrm{y}|}d\mathrm{y}\Big| 
   \\ &\lesssim\nonumber\mathrm{M}\frac{\gamma_i\rho_\mathrm{m}}{4\pi}\Big[ \frac{1}{|\mathrm{x}-\mathrm{z}_i|}\Big|\nabla|\mathrm{x}-\mathrm{z}^*|\Big| \Vert \cdot-\mathrm{z}_i\Vert_{\mathrm{L}^2(\Omega_i)} \Vert \partial_\mathrm{t}^3u^{(i)}(\cdot,\mathrm{t}_0^*)\Vert_{\mathrm{L}^2(\Omega_i)} 
   \\ &\nonumber+ \Big|\nabla|\mathrm{x}-\mathrm{z}^*|\Big|\Big|\nabla\frac{1}{|\mathrm{x}-\mathrm{z}^*|}\Big| \Vert (\cdot-\mathrm{z}_i)^2\Vert_{\mathrm{L}^2(\Omega)} \Vert \partial_\mathrm{t}^3u^{(i)}(\cdot,\mathrm{t}_0^*)\Vert_{\mathrm{L}^2(\Omega_i)}\Big]
    \\ &= \mathcal{O}(\mathrm{M}\delta^\frac{5}{2} |\mathrm{x}-\mathrm{z}_i|^{-1}\Vert \partial_\mathrm{t}^3u^{(i)}(\cdot,\mathrm{t}_0^*)\Vert_{\mathrm{L}^2(\Omega_i )})= \mathcal{O}\big(\mathrm{M}\delta^4\bm{\xi}^{-1}\big).
\end{align}
Subsequently, in a similar fashion as described above, we proceed to evaluate the subsequent term.
\begin{align}
 \nonumber&\textbf{err}_{(4)}\\ \nonumber&:= \Big|\sum_{i=1}^\mathrm{M}\frac{\rho_\mathrm{m}\alpha_i}{4\pi} \int_{\partial\Omega_i}\frac{\partial_\nu u^{(i)}(\mathrm{y},\mathrm{t}-\mathrm{c}_0^{-1}|\mathrm{x}-\mathrm{z}_i|)-\partial_\nu u^{(i)}(\mathrm{y},\mathrm{t}-\mathrm{c}_0^{-1}|\mathrm{x}-\mathrm{y}|)}{|\mathrm{x}-\mathrm{y}|}d\sigma_\mathrm{y}\Big|  
 \\ &\lesssim\nonumber\mathrm{M}\frac{\alpha_i\rho_\mathrm{m}}{4\pi}\Bigg[ \delta\frac{1}{|\mathrm{x}-\mathrm{z}_i|}\Big|\nabla|\mathrm{x}-\mathrm{z}^*|\Big| \Vert 1\Vert_{\mathrm{L}^2(\partial\Omega_i)} \Vert \partial_\mathrm{t}\partial_\nu u^{(i)}(\cdot,\mathrm{t}_0^*)\Vert_{\mathrm{H}^{-\frac{1}{2}}(\partial\Omega_i)} 
 \\ &\nonumber+ \delta^2\Big|\nabla|\mathrm{x}-\mathrm{z}^*|\Big|\Big|\nabla\frac{1}{|\mathrm{x}-\mathrm{z}^*|}\Big| \Vert 1\Vert_{\mathrm{L}^2(\partial\Omega_i)} \Vert \partial_\mathrm{t}\partial_\nu u^{(i)}(\cdot,\mathrm{t}_0^*)\Vert_{\mathrm{H}^{-\frac{1}{2}}(\partial\Omega_i)}\Bigg]
    \\ &= \mathcal{O}(\mathrm{M}|\mathrm{x}-\mathrm{z}_i|^{-1}\Vert \partial_\mathrm{t}\partial_\nu u^{(i)}(\cdot,\mathrm{t}_0^*)\Vert_{\mathrm{H}^{-\frac{1}{2}}(\partial\Omega_i)})= \mathcal{O}\big(\mathrm{M}\delta^2\bm{\xi}^{-1}\big).
\end{align}
Finally, in accordance with the findings presented in \cite{AM2}, we can demonstrate that
\begin{align}
    \textbf{err}_{(5)}\nonumber&:=\Bigg|\sum_{i=1}^\mathrm{M}\frac{\rho_\mathrm{m}\alpha_i}{4\pi} \int_{\partial\Omega_i}\Big(\frac{1}{|\mathrm{x}-\mathrm{y}|}-\frac{1}{|\partial\Omega_i|}\int_{\partial\Omega_i}\frac{1}{|\mathrm{x}-\mathrm{y}|}\Big) \partial_\nu u^{(i)} (\mathrm{y},\mathrm{t}-\mathrm{c}_0^{-1}|\mathrm{x}-\mathrm{z}_i|) d\sigma_\mathrm{y}\Bigg|
    \\ &= \mathcal{O}\big(\mathrm{M}\delta^2\bm{\xi}^{-1}\big).
\end{align}
Therefore, upon incorporating the estimations of $\textbf{err}_{(1)}$, $\textbf{err}_{(2)}$, $\textbf{err}_{(3)}$, $\textbf{err}_{(4)}$, and $\textbf{err}_{(5)}$ into (\ref{approximation}), we obtain:
\begin{align}
 u^\mathrm{s}(\mathrm{x},\mathrm{t}) = \sum_{i=1}^\mathrm{M}\frac{\rho_\mathrm{m}\alpha_i}{4\pi} \frac{1}{{|\partial\Omega_i|}}\int_{\partial\Omega_i}\frac{1}{|\mathrm{x}-\mathrm{y}|}\int_{\partial\Omega_i}\partial_\nu u(\mathrm{y},\mathrm{t}-\mathrm{c}_0^{-1}|\mathrm{x}-\mathrm{z}_i|)d\sigma_\mathrm{y} + \mathcal{O}\big(\mathrm{M}\delta^2\bm{\xi}^{-1}\big).    
\end{align}
%------------------------------------------------------------------------------------
%------------------------------------------------------------------------------------

\section{A Priori Estimates}\label{apri}
\begin{proposition}\label{p1multi}
For $u = u^\textbf{in}+u^\mathrm{s}$ as the solution of (\ref{math-modelmulti}) we have the following estimates
\begin{align}
\Vert \partial_\mathrm{t}^\mathrm{n}\nabla u^{(i)}(\cdot,\mathrm{t})\Vert_{\mathrm{L}^2(\Omega_i)} \lesssim \delta^{\frac{3}{2}},\quad \mathrm{t}\in [0,\mathrm{T}], \quad \mathrm{n}=0,1,\ldots.
\end{align}
\end{proposition}
\begin{proof}
See Section \cite[Sec. 3]{AM2} for the proof.
\end{proof}
Afterwards, it is necessary to enhance the estimated value for $\Vert \partial_\mathrm{t}^\mathrm{n}\nabla u^{(i)}(\cdot,\mathrm{t})\Vert_{\mathrm{L}^2(\Omega_i)}$ in order to proceed with our mathematical analysis. Hence, we present the following lemma to address this issue.
\begin{lemma}\label{l1}
Provided that the distribution of resonating micro-bubbles satisfies the condition
\begin{align}
        \frac{\rho_\mathrm{m}}{4\pi}\textbf{Vol}(\mathrm{B}_j)\delta^6\max_{\mathrm{n}\in\mathbb N}\max_{1\le i\le\mathrm{M}} \sum_{j\ne i}\frac{\alpha_\mathrm{j}^2}{|1+\alpha_i\lambda^{(3)}_\mathrm{n}|^2} d^{-6}_{ij} <1,
\end{align}
the solution of the system of integral equation (\ref{system}) can be estimated as follows:
\begin{align}
     \Big\Vert \frac{\partial^n}{\partial t^n}\nabla u^{(i)}(\cdot,t)\Big\Vert_{\mathrm{L}^2(\Omega_i)} \lesssim \delta^\frac{5}{2},\; \text{for}\; i=1,2,\ldots,\mathrm{M}\; \text{and}\ \mathrm{n}\in \mathbb Z^+.
\end{align}
\end{lemma}
\begin{proof}
We present the following equation
\begin{align}\nonumber
    \nabla u^{(i)} - \sum_{i=1}^\mathrm{M}\alpha_i \nabla\text{div}\int_{\Omega_i} \frac{\rho_\mathrm{m}}{4\pi|\mathrm{x}-\mathrm{y}|}\nabla u^{(i)}(\mathrm{y},\mathrm{t}-\mathrm{c}_0^{-1}|\mathrm{x}-\mathrm{y}|)d\mathrm{y}\nonumber &=\nabla u^\textbf{in} 
    \\ \nonumber&- \sum_{i=1}^\mathrm{M}\beta_i\nabla \int_{\Omega_i} \frac{\rho_\mathrm{m}}{4\pi|\mathrm{x}-\mathrm{y}|}u^{(i)}_{\mathrm{t}\mathrm{t}}(\mathrm{y},\mathrm{t}-\mathrm{c}_0^{-1}|\mathrm{x}-\mathrm{y}|)d\mathrm{y},
\end{align} 
which we rewrite as follows for $\mathrm{x}\in \Omega_i$:
\begin{align}\nonumber
    \\ &\nabla u^{(i)} - \alpha_i \nabla\text{div}\int_{\Omega_i} \frac{\rho_\mathrm{m}}{4\pi|\mathrm{x}-\mathrm{y}|}\nabla u^{(i)}(\mathrm{y},\mathrm{t}-\mathrm{c}_0^{-1}|\mathrm{x}-\mathrm{y}|)d\mathrm{y} \\ \nonumber&=\nabla u^\textbf{in} + \sum_{\mathrm{j}\ne i}^\mathrm{M}\alpha_\mathrm{j} \nabla\text{div}\int_{\Omega_\mathrm{j}} \frac{\rho_\mathrm{m}}{4\pi|\mathrm{x}-\mathrm{y}|}\nabla u^{(\mathrm{j})}(\mathrm{y},\mathrm{t}-\mathrm{c}_0^{-1}|\mathrm{x}-\mathrm{y}|)d\mathrm{y} - \beta_i\nabla \int_{\Omega_i} \frac{\rho_\mathrm{m}}{4\pi|\mathrm{x}-\mathrm{y}|}u^{(i)}_{\mathrm{t}\mathrm{t}}(\mathrm{y},\mathrm{t}-\mathrm{c}_0^{-1}|\mathrm{x}-\mathrm{y}|)d\mathrm{y}\\ \nonumber&- \sum_{\mathrm{j}\ne i}^\mathrm{M}\beta_\mathrm{j}\nabla \int_{\Omega_\mathrm{j}} \frac{\rho_\mathrm{m}}{4\pi|\mathrm{x}-\mathrm{y}|}u^{(\mathrm{j})}_{\mathrm{t}\mathrm{t}}(\mathrm{y},\mathrm{t}-\mathrm{c}_0^{-1}|\mathrm{x}-\mathrm{y}|)d\mathrm{y}.
\end{align} 
Subsequently, a Taylor series expansion with respect to time leads to:
\begin{align}
&\nonumber\nabla u^{(i)} - \alpha_i \nabla\text{div}\int_{\Omega_i} \frac{\rho_\mathrm{m}}{4\pi|\mathrm{x}-\mathrm{y}|}\nabla u^{(i)}(\mathrm{y},\mathrm{t})d\mathrm{y} \\ \nonumber&=\nabla u^\textbf{in} + \frac{\alpha_i\rho_\mathrm{m}}{4\pi} \nabla\text{div}\int_{\Omega_i} |\mathrm{x}-\mathrm{y}|\partial_\mathrm{t}^2\nabla u^{(i)}(\mathrm{y},\mathrm{t}_1^*)d\mathrm{y} - \beta_i\nabla \int_{\Omega_i} \frac{\rho_\mathrm{m}}{4\pi|\mathrm{x}-\mathrm{y}|}u^{(i)}_{\mathrm{t}\mathrm{t}}(\mathrm{y},\mathrm{t})d\mathrm{y}
\\ \nonumber&- \frac{\beta_i\rho_\mathrm{m}}{4\pi}\nabla \int_{\Omega_i} |\mathrm{x}-\mathrm{y}|\partial_\mathrm{t}^4u^{(i)}(\mathrm{y},\mathrm{t}_3^*)d\mathrm{y} 
+ \sum_{\mathrm{j}\ne i}^\mathrm{M}\alpha_\mathrm{j} \nabla\text{div}\int_{\Omega_\mathrm{j}} \frac{\rho_\mathrm{m}}{4\pi|\mathrm{x}-\mathrm{y}|}\nabla u^{(\mathrm{j})}(\mathrm{y},\mathrm{t})d\mathrm{y} 
\\ \nonumber&- \sum_{\mathrm{j}\ne i}^\mathrm{M}\beta_\mathrm{j}\nabla \int_{\Omega_\mathrm{j}} \frac{\rho_\mathrm{m}}{4\pi|\mathrm{x}-\mathrm{y}|}u^{(\mathrm{j})}_{\mathrm{t}\mathrm{t}}(\mathrm{y},\mathrm{t})d\mathrm{y}
- \sum_{\mathrm{j}\ne i}^\mathrm{M}\frac{\beta_\mathrm{j}\rho_\mathrm{m}}{4\pi}\nabla \int_{\Omega_\mathrm{j}} |\mathrm{x}-\mathrm{y}|\partial_\mathrm{t}^4u^{(\mathrm{j})}(\mathrm{y},\mathrm{t}_4^*)d\mathrm{y}
\\ \nonumber&+ \sum_{\mathrm{j}\ne i}^\mathrm{M}\frac{\alpha_\mathrm{j}\rho_\mathrm{m}}{4\pi} \nabla\text{div}\int_{\Omega_\mathrm{j}} |\mathrm{x}-\mathrm{y}|\partial_\mathrm{t}^2\nabla u^{(i)}(\mathrm{y},\mathrm{t}_2^*)d\mathrm{y},   
\end{align}
where $\mathrm{t}_2^*, \mathrm{t}_3^*, \mathrm{t}_4^* \in \big(\mathrm{t}-\mathrm{c}^{-1}_0|\mathrm{x}-\mathrm{y}|,\mathrm{t}\big).$
Let us now denote $\mathcal{G}^{(0)}(\mathrm{x},\mathrm{y}):= \frac{\rho_\mathrm{m}}{4\pi|\mathrm{x}-\mathrm{y}|}$, $\mathcal{H}^{(0)}(\mathrm{x},\mathrm{y}):= \frac{\rho_\mathrm{m}}{4\pi}|\mathrm{x}-\mathrm{y}|$ and we define the Magnetization operator as follows:
\begin{align}\nonumber
    \bm{\mathrm{M}}^{(0)}_{\mathrm{B}}\big[f\big](\xi) := \nabla \int_{\mathrm{B}_i}\mathop{\nabla}\limits_{\eta} \frac{1}{4\pi|\xi-\eta|} \cdot f(\eta)d\eta.
\end{align}
\begin{align}
    \nabla u^{(i)} + \alpha_i \bm{\mathrm{M}}^{(0)}_{\Omega_i}\Big[\nabla u^{(i)}\Big] 
     \nonumber&= \nabla u^{\textbf{in}}+ \alpha_i \nabla\text{div}\int_{\Omega_i} \mathcal{H}^{(0)}(\mathrm{x},\mathrm{y})\partial_\mathrm{t}^2\nabla u^{(i)}(\mathrm{y},\mathrm{t}_1^*)d\mathrm{y}
    - \beta_i\nabla \int_{\Omega_i} \frac{\rho_\mathrm{m}}{4\pi|\mathrm{x}-\mathrm{y}|}u^{(i)}_{\mathrm{t}\mathrm{t}}(\mathrm{y},\mathrm{t})d\mathrm{y}
    \\ \nonumber&- \beta_i\nabla \int_{\Omega_i} \mathcal{H}^{(0)}(\mathrm{x},\mathrm{y})\partial_\mathrm{t}^4u^{(i)}(\mathrm{y},\mathrm{t}^*_3)d\mathrm{y} 
    +\sum_{\mathrm{j}\ne i}^\mathrm{M}\alpha_\mathrm{j} \nabla\text{div}\;\mathcal{G}^{(0)}(\mathrm{z}_i,\mathrm{z}_\mathrm{j})\int_{\Omega_\mathrm{j}} \nabla u^{(\mathrm{j})}(\mathrm{y},\mathrm{t})d\mathrm{y}
     \\ \nonumber&+ \sum_{\mathrm{j}\ne i}^\mathrm{M}\alpha_\mathrm{j} \nabla\text{div} \int_0^1\mathop{\nabla}\limits_{\mathrm{x}}\mathcal{G}^{(0)}\big(\mathrm{z}_i+\vartheta_1(\mathrm{x}-\mathrm{z}_i);\mathrm{z}_\mathrm{j}\big)\cdot (\mathrm{x}-\mathrm{z}_i)d\vartheta_1\int_{\Omega_\mathrm{j}}\nabla u^{(\mathrm{j})}(\mathrm{y},\mathrm{t})d\mathrm{y}
    \\ \nonumber&+ \sum_{\mathrm{j}\ne i}^\mathrm{M}\alpha_\mathrm{j} \nabla\text{div} \int_0^1\mathop{\nabla}\limits_{\mathrm{y}}\mathcal{G}^{(0)}\big(\mathrm{x};\mathrm{z}_\mathrm{j}+\vartheta_1(\mathrm{y}-\mathrm{z}_\mathrm{j})\big)\cdot (\mathrm{y}-\mathrm{z}_\mathrm{j})d\vartheta_1\int_{\Omega_\mathrm{j}}\nabla u^{(\mathrm{j})}(\mathrm{y},\mathrm{t})d\mathrm{y}
    \\ \nonumber&- \sum_{\mathrm{j}\ne i}^\mathrm{M}\beta_\mathrm{j} \nabla\mathcal{G}^{(0)}(\mathrm{z}_i,\mathrm{z}_\mathrm{j}) \int_{\Omega_\mathrm{j}}\partial^2_\mathrm{t} u^{(\mathrm{j})}(\mathrm{y},\mathrm{t})d\mathrm{y}
    \\ \nonumber&- \sum_{\mathrm{j}\ne i}^\mathrm{M}\beta_\mathrm{j} \nabla\int_0^1\mathop{\nabla}\limits_{\mathrm{x}}\mathcal{G}^{(0)}\big(\mathrm{z}_i+\vartheta_2(\mathrm{x}-\mathrm{z}_i);\mathrm{z}_\mathrm{j}\big)\cdot (\mathrm{x}-\mathrm{z}_i)d\vartheta_2\int_{\Omega_\mathrm{j}} \partial^2_\mathrm{t}u^{(\mathrm{j})}(\mathrm{y},\mathrm{t})d\mathrm{y}
    \\ \nonumber&- \sum_{\mathrm{j}\ne i}^\mathrm{M}\beta_\mathrm{j} \nabla\int_0^1\mathop{\nabla}\limits_{\mathrm{y}}\mathcal{G}^{(0)}\big(\mathrm{x};\mathrm{z}_\mathrm{j}+\vartheta_2(\mathrm{y}-\mathrm{z}_\mathrm{j})\big)\cdot (\mathrm{y}-\mathrm{z}_\mathrm{j})d\vartheta_2\int_{\Omega_\mathrm{j}}\partial^2_\mathrm{t} u^{(\mathrm{j})}(\mathrm{y},\mathrm{t})d\mathrm{y}
    \\ \nonumber&- \sum_{\mathrm{j}\ne i}^\mathrm{M}\beta_\mathrm{j}\nabla\;\mathcal{H}^{(0)}\big(\mathrm{z}_i;\mathrm{z}_\mathrm{j}\big) \int_{\Omega_\mathrm{j}} \partial_\mathrm{t}^4 u^{(\mathrm{j})}(\mathrm{y},\mathrm{t}_4^*)d\mathrm{y}
    \\ \nonumber&- \sum_{\mathrm{j}\ne i}^\mathrm{M}\beta_\mathrm{j}\nabla\int_{\Omega_\mathrm{j}} \int_0^1\mathop{\nabla}\limits_{\mathrm{x}}\mathcal{H}^{(0)}\Big(\mathrm{z}_i+\vartheta_3(\mathrm{x}-\mathrm{z}_i);\mathrm{z}_\mathrm{j}\Big)\cdot(\mathrm{x}-\mathrm{z}_i)d\vartheta_3\; \partial_\mathrm{t}^4 u^{(\mathrm{j})}(\mathrm{y},\mathrm{t}_4^*)d\mathrm{y}
    \\ \nonumber&- \sum_{\mathrm{j}\ne i}^\mathrm{M}\beta_\mathrm{j}\nabla \int_{\Omega_\mathrm{j}} \int_0^1\mathop{\nabla}\limits_{\mathrm{y}}\mathcal{H}^{(0)}\Big(\mathrm{z}_i;\mathrm{z}_\mathrm{j}+\vartheta_3(\mathrm{y}-\mathrm{z}_\mathrm{j})\Big)\cdot(\mathrm{y}-\mathrm{z}_\mathrm{j})d\vartheta_3\; \partial_\mathrm{t}^4 u^{(\mathrm{j})}(\mathrm{y},\mathrm{t}_4^*)d\mathrm{y}
    \\ \nonumber&+ \sum_{\mathrm{j}\ne i}^\mathrm{M}\alpha_\mathrm{j}\nabla\text{div}\;\mathcal{H}^{(0)}\big(\mathrm{z}_i;\mathrm{z}_\mathrm{j}\big) \int_{\Omega_\mathrm{j}} \partial_\mathrm{t}^2 \nabla u^{(\mathrm{j})}(\mathrm{y},\mathrm{t}_2^*)d\vartheta_4d\mathrm{y}
    \\ \nonumber&+ \sum_{\mathrm{j}\ne i}^\mathrm{M}\alpha_\mathrm{j}\nabla\text{div} \int_{\Omega_\mathrm{j}} \int_0^1\mathop{\nabla}\limits_{\mathrm{x}}\mathcal{H}^{(0)}\Big(\mathrm{z}_i+\vartheta_4(\mathrm{x}-\mathrm{z}_i);\mathrm{z}_\mathrm{j}\Big)\cdot(\mathrm{x}-\mathrm{z}_i)d\vartheta_4\; \partial_\mathrm{t}^2\nabla u^{(\mathrm{j})}(\mathrm{y},\mathrm{t}_2^*)d\mathrm{y}
    \\ \nonumber&+ \sum_{\mathrm{j}\ne i}^\mathrm{M}\alpha_\mathrm{j}\nabla\text{div} \int_{\Omega_\mathrm{j}} \int_0^1\mathop{\nabla}\limits_{\mathrm{y}}\mathcal{H}^{(0)}\Big(\mathrm{z}_i;\mathrm{z}_\mathrm{j}+\vartheta_4(\mathrm{y}-\mathrm{z}_\mathrm{j})\Big)\cdot(\mathrm{y}-\mathrm{z}_\mathrm{j})d\vartheta_4\; \partial_\mathrm{t}^2\nabla u^{(\mathrm{j})}(\mathrm{y},\mathrm{t}_2^*)d\mathrm{y}
\end{align}
\vspace{-5pt}
Then, we rewrite the aforementioned equation in the scaled domain $\mathrm{B}$ as follows:
\begin{align}
    \widetilde{\nabla u}^{(i)} + \alpha_i \bm{\mathrm{M}}^{(0)}_{\mathrm{B}_i}\Big[\widetilde{\nabla u}^{(i)}\Big] 
     \nonumber&= \widetilde{\nabla u}^{\textbf{in}} 
    + \alpha_i\delta^2\nabla\text{div}\int_{\mathrm{B}_i} \mathcal H^{(0)}(\xi,\eta)\partial_\mathrm{t}^2\widetilde{\nabla u}^{(i)}(\eta,\mathrm{t}_1^*)d\eta 
    - \beta_i\delta\nabla \int_{\mathrm{B}_i} \frac{\rho_\mathrm{m}}{4\pi|\xi-\eta|}\tilde{u}^{(i)}_{\mathrm{t}\mathrm{t}}(\eta,\mathrm{t})d\eta
    \\ \nonumber&- \beta_i\delta^3\nabla \int_{\mathrm{B}_i} \mathcal H^{(0)}(\xi,\eta)\partial_\mathrm{t}^4\tilde{u}^{(i)}(\eta,\mathrm{t}_3^*)d\eta
    + \sum_{\mathrm{j}\ne i}^\mathrm{M}\alpha_\mathrm{j} \nabla\text{div}\ \mathcal{G}^{(0)}(\mathrm{z}_i,\mathrm{z}_\mathrm{j})\int_{\Omega_\mathrm{j}} \nabla u^{(\mathrm{j})}(\mathrm{y},\mathrm{t})d\mathrm{y}
    \\ \nonumber&+ \delta^4\sum_{\mathrm{j}\ne i}^\mathrm{M}\alpha_\mathrm{j} \nabla\text{div} \int_0^1\mathop{\nabla}\limits_{\xi}\mathcal{G}^{(0)}\big(\mathrm{z}_i+\vartheta_1\xi\delta;\mathrm{z}_\mathrm{j}\big)\cdot \xi d\vartheta_1\int_{\mathrm{B}_\mathrm{j}}\widetilde{\nabla u}^{(\mathrm{j})}(\eta,\mathrm{t})d\eta
    \\ \nonumber&+ \delta^4\sum_{\mathrm{j}\ne i}^\mathrm{M}\alpha_\mathrm{j} \nabla\text{div} \int_0^1\mathop{\nabla}\limits_{\eta}\mathcal{G}^{(0)}\big(\mathrm{z}_i;\mathrm{z}_\mathrm{j}+\vartheta_1\eta\delta\big)\cdot \eta d\vartheta_1\int_{\mathrm{B}_\mathrm{j}}\widetilde{\nabla u}^{(\mathrm{j})}(\eta,\mathrm{t})d\eta 
    \\ \nonumber&-\sum_{\mathrm{j}\ne i}^\mathrm{M}\beta_\mathrm{j} \nabla\mathcal{G}^{(0)}(\mathrm{z}_i,\mathrm{z}_\mathrm{j})\int_{\Omega_\mathrm{j}} \partial^2_\mathrm{t} u^{(\mathrm{j})}(\mathrm{y},\mathrm{t})d\mathrm{y} 
    \\ \nonumber&- \delta^4\sum_{\mathrm{j}\ne i}^\mathrm{M}\beta_\mathrm{j} \nabla\int_0^1\mathop{\nabla}\limits_{\xi}\mathcal{G}^{(0)}\big(\mathrm{z}_i+\vartheta_2\xi\delta;\mathrm{z}_\mathrm{j}\big)\cdot \xi d\vartheta_2\int_{\mathrm{B}_\mathrm{j}} \partial^2_\mathrm{t}\tilde{u}^{(\mathrm{j})}(\eta,\mathrm{t})d\eta
    \\ \nonumber&- \delta^4\sum_{\mathrm{j}\ne i}^\mathrm{M}\beta_\mathrm{j} \nabla\int_0^1\mathop{\nabla}\limits_{\eta}\mathcal{G}^{(0)}\big(\mathrm{z}_i;\mathrm{z}_\mathrm{j}+\vartheta_2\eta\delta\big)\cdot \eta d\vartheta_2\int_{\mathrm{B}_\mathrm{j}}\partial^4_\mathrm{t} \tilde{u}^{(\mathrm{j})}(\eta,\mathrm{t})d\eta
     \\ \nonumber&- \sum_{\mathrm{j}\ne i}^\mathrm{M}\beta_\mathrm{j}\nabla \mathcal H^{(0)}(\mathrm{z}_i,\mathrm{z}_\mathrm{j})\int_{\Omega_\mathrm{j}} \partial_\mathrm{t}^4\tilde{u}^{(\mathrm{j})}(\mathrm{y},\mathrm{t}_4^*)d\mathrm{y}
    \\ \nonumber&- \delta^4\sum_{\mathrm{j}\ne i}^\mathrm{M}\beta_\mathrm{j}\nabla \int_{\mathrm{B}_\mathrm{j}} \int_0^1\mathop{\nabla}\limits_{\xi}\mathcal H^{(0)}(\mathrm{z}_i+\vartheta_3\delta\xi;\mathrm{z}_\mathrm{j})\cdot\xi d\vartheta_3\;\partial_\mathrm{t}^4\tilde{u}^{(\mathrm{j})}(\eta,\mathrm{t}_4^*)d\eta
    \\ \nonumber&- \delta^4\sum_{\mathrm{j}\ne i}^\mathrm{M}\beta_\mathrm{j}\nabla \int_{\mathrm{B}_\mathrm{j}} \int_0^1\mathop{\nabla}\limits_{\eta}\mathcal H^{(0)}(\mathrm{z}_i;\mathrm{z}_\mathrm{j}+\vartheta_3\eta\delta)\cdot\eta d\vartheta_3\;\partial_\mathrm{t}^4\tilde{u}^{(\mathrm{j})}(\mathrm{y},\mathrm{t}_4^*)d\eta
    \\ \nonumber&+ \sum_{\mathrm{j}\ne i}^\mathrm{M}\alpha_\mathrm{j}\nabla\text{div}\ \mathcal H^{(0)}(\mathrm{z}_i,\mathrm{z}_\mathrm{j})\int_{\Omega_\mathrm{j}} \partial_\mathrm{t}^2 \nabla u^{(\mathrm{j})}(\mathrm{y},\mathrm{t}_2^*)d\mathrm{y}
    \\ \nonumber&+\delta^4\sum_{\mathrm{j}\ne i}^\mathrm{M}\alpha_\mathrm{j}\nabla\text{div} \int_{\mathrm{B}_\mathrm{j}} \int_0^1\mathop{\nabla}\limits_{\xi}\mathcal H^{(0)}(\mathrm{z}_i+\vartheta_4\delta\xi;\mathrm{z}_\mathrm{j})\cdot\xi d\vartheta_4\;\partial_\mathrm{t}^2\widetilde{\nabla u}^{(\mathrm{j})}(\eta,\mathrm{t}_2^*)d\eta
    \\ \nonumber&+ \delta^4\sum_{\mathrm{j}\ne i}^\mathrm{M}\alpha_\mathrm{j}\nabla\text{div} \int_{\mathrm{B}_\mathrm{j}} \int_0^1\mathop{\nabla}\limits_{\eta}\mathcal H^{(0)}(\mathrm{z}_i;\mathrm{z}_\mathrm{j}+\vartheta_4\eta\delta)\cdot\eta d\vartheta_4\;\partial_\mathrm{t}^2\widetilde{\nabla u}^{(\mathrm{j})}(\eta,\mathrm{t}_2^*)d\eta
\end{align}
We study the system of integral equations in the Hilbert space of vector-valued function $\big(\mathrm{L}^{2}(\mathrm{B})\big)^3.$ For the sake of simplicity, we use $\mathbb{L}^2(\mathrm{B}) =\big(\mathrm{L}^{2}(\mathrm{B})\big)^3$. This space can be decomposed into three sub-spaces as a direct sum as following, see \cite{raveski},
\begin{equation}\label{decompositionmulti}
    \mathbb{L}^{2} = \mathbb{H}_{0}(\text{div}\;0,\mathrm{B}) \oplus\mathbb{H}_{0}(\textbf{curl}\;0,\mathrm{B})\oplus \nabla \mathbb{H}_{\text{arm}}.
\end{equation}
Consider $\big(\mathrm{e}^{(1)}_{\mathrm{n}}\big)_{\mathrm{n} \in \mathbb{N}}$ and $\big(\mathrm{e}^{(2)}_{\mathrm{n}}\big)_{\mathrm{n} \in \mathbb{N}}$ to be any orthonormal basis of the sub-spaces $\mathbb{H}_{0}(\text{div}\;0,\mathrm{B})$ and $\mathbb{H}_{0}(\textbf{curl}\;0,\mathrm{B})$ respectively. But for the sub-space $\nabla \mathbb{H}_{\text{arm}}$, we consider the complete orthonormal basis $\big(\mathrm{e}^{(3)}_{\mathrm{n}}\big)_{n \in \mathbb{N}}$ derived as the eigenfunctions of the magnetization operator $\nabla\bm{\mathrm{M}}^{{(0)}}_{\mathrm{B}}: \nabla \mathbb{H}_{\text{arm}}\rightarrow \nabla \mathbb{H}_{\text{arm}}$, \cite{friedmanI}, with $\big(\mathrm{\lambda}^{(3)}_{\mathrm{n}}\big)_{n \in \mathbb{N}}$ as the corresponding eigenvalues.
\newline
Let us now denote the operators as follows:
\begin{align}\nonumber
\begin{cases}\displaystyle
    \mathbb A^{(\mathrm{ii})}[\widetilde{\nabla u}^{(i)}](\xi,\mathrm{t}_1^*) := \nabla\text{div}\int_{\mathrm{B}_i} \mathcal H^{(0)}(\xi,\eta)\partial_\mathrm{t}^2\widetilde{\nabla u}^{(i)}(\eta,\mathrm{t}_1^*)d\eta, \\ \nonumber \displaystyle \mathbb B^{(\mathrm{ii})}[\partial_\mathrm{t}^4\tilde{u}^{(i)}](\xi,\mathrm{t}_3^*):= \nabla \int_{\mathrm{B}_i} \mathcal H^{(0)}(\xi,\eta)\partial_\mathrm{t}^4\tilde{u}^{(i)}(\eta,\mathrm{t}_3^*)d\eta, \\ \nonumber \displaystyle
    \mathbb V^{(\mathrm{ii})}[\partial_\mathrm{t}^2\tilde{u}^{(i)}](\xi,\mathrm{t}):= \nabla \int_{\mathrm{B}_i} \frac{\rho_\mathrm{m}}{4\pi|\xi-\eta|}\tilde{u}^{(i)}_{\mathrm{t}\mathrm{t}}(\eta,\mathrm{t})d\eta\\ \nonumber\displaystyle
    \mathbb C^{(\mathrm{ij})}[\widetilde{\nabla u}^{(\mathrm{j})}](\xi,\mathrm{t}):=  \delta^{-1}\nabla\text{div} \int_{\mathrm{B}_\mathrm{j}}\mathcal{G}^{(0)}(\mathrm{z}_i,\mathrm{z}_\mathrm{j}) \widetilde{\nabla u}^{(\mathrm{j})}(\eta,\mathrm{t})d\eta
    +  \nabla\text{div} \int_0^1\mathop{\nabla}\limits_{\xi}\mathcal{G}^{(0)}\big(\mathrm{z}_i+\vartheta_1\xi\delta;\mathrm{z}_\mathrm{j}\big)\cdot \xi d\vartheta_1\int_{\mathrm{B}_\mathrm{j}}\widetilde{\nabla u}^{(\mathrm{j})}(\eta,\mathrm{t})d\eta
    \\ \nonumber\displaystyle \quad \quad \quad \quad \quad \quad \quad \quad + \nabla\text{div} \int_0^1\mathop{\nabla}\limits_{\eta}\mathcal{G}^{(0)}\big(\mathrm{z}_i;\mathrm{z}_\mathrm{j}+\vartheta_1\eta\delta\big)\cdot \eta d\vartheta_1\int_{\mathrm{B}_\mathrm{j}}\widetilde{\nabla u}^{(\mathrm{j})}(\eta,\mathrm{t})d\eta \\ \displaystyle
    \mathbb D^{(\mathrm{ij})}[\partial_\mathrm{t}^2\tilde{u}^{(\mathrm{j})}](\xi,\mathrm{t}):= -\delta^{-1}\nabla\int_{\mathrm{B}_\mathrm{j}}\mathcal{G}^{(0)}(\mathrm{z}_i,\mathrm{z}_\mathrm{j}) \partial^2_\mathrm{t} \tilde{u}^{(\mathrm{j})}(\eta,\mathrm{t})d\eta
     -  \nabla\int_0^1\mathop{\nabla}\limits_{\xi}\mathcal{G}^{(0)}\big(\mathrm{z}_i+\vartheta_2\xi\delta;\mathrm{z}_\mathrm{j}\big)\cdot \xi d\vartheta_2\int_{\mathrm{B}_\mathrm{j}} \partial^2_\mathrm{t}\tilde{u}^{(\mathrm{j})}(\eta,\mathrm{t})d\eta
    \\ \displaystyle\nonumber\quad \quad \quad \quad \quad \quad \quad \quad- \nabla\int_0^1\mathop{\nabla}\limits_{\eta}\mathcal{G}^{(0)}\big(\mathrm{z}_i;\mathrm{z}_\mathrm{j}+\vartheta_2\eta\delta\big)\cdot \eta d\vartheta_2\int_{\mathrm{B}_\mathrm{j}}\partial^4_\mathrm{t} \tilde{u}^{(\mathrm{j})}(\eta,\mathrm{t})d\eta\\ \displaystyle
    \displaystyle
    \mathbb E^{(\mathrm{ij})}[\partial_\mathrm{t}^4\tilde{u}^{(\mathrm{j})}](\xi,\mathrm{t}_4^*):= - \delta^{-1}\nabla \int_{\mathrm{B}_\mathrm{j}}\mathcal H^{(0)}(\mathrm{z}_i,\mathrm{z}_\mathrm{j}) \partial_\mathrm{t}^4\tilde{u}^{(\mathrm{j})}(\mathrm{y},\mathrm{t}_4^*)d\mathrm{y}
    - \nabla \int_{\mathrm{B}_\mathrm{j}} \int_0^1\mathop{\nabla}\limits_{\xi}\mathcal H^{(0)}(\mathrm{z}_i+\vartheta_3\delta\xi;\mathrm{z}_\mathrm{j})\cdot\xi d\vartheta_3\;\partial_\mathrm{t}^4\tilde{u}^{(\mathrm{j})}(\eta,\mathrm{t}_4^*)d\eta
    \\ \displaystyle \quad \quad \quad \quad \quad \quad \quad \quad- \nabla \int_{\mathrm{B}_\mathrm{j}} \int_0^1\mathop{\nabla}\limits_{\eta}\mathcal H^{(0)}(\mathrm{z}_i;\mathrm{z}_\mathrm{j}+\vartheta_3\eta\delta)\cdot\eta d\vartheta_3\;\partial_\mathrm{t}^4\tilde{u}^{(\mathrm{j})}(\eta,\mathrm{t}_2^*)d\eta\\ \displaystyle
    \mathbb F^{(\mathrm{ij})}[\partial_\mathrm{t}^2\widetilde{\nabla u}^{(\mathrm{j})}](\xi,\mathrm{t}_4^*):= \delta^{-1}\nabla\text{div}\int_{\mathrm{B}_\mathrm{j}} \mathcal H^{(0)}(\mathrm{z}_i,\mathrm{z}_\mathrm{j}) \partial_\mathrm{t}^2 \nabla u^{(\mathrm{j})}(\mathrm{y},\mathrm{t}_2^*)d\mathrm{y}
    \\ \displaystyle\nonumber\quad \quad \quad \quad \quad \quad \quad \quad+\nabla\text{div} \int_{\mathrm{B}_\mathrm{j}} \int_0^1\mathop{\nabla}\limits_{\xi}\mathcal H^{(0)}(\mathrm{z}_i+\vartheta_4\delta\xi;\mathrm{z}_\mathrm{j})\cdot\xi d\vartheta_4\;\partial_\mathrm{t}^2\widetilde{\nabla u}^{(\mathrm{j})}(\eta,\mathrm{t}_2^*)d\eta
    \\ \displaystyle\nonumber\quad \quad \quad \quad \quad \quad \quad \quad+ \nabla\text{div} \int_{\mathrm{B}_\mathrm{j}} \int_0^1\mathop{\nabla}\limits_{\eta}\mathcal H^{(0)}(\mathrm{z}_i;\mathrm{z}_\mathrm{j}+\vartheta_4\eta\delta)\cdot\eta d\vartheta_4\;\partial_\mathrm{t}^2\widetilde{\nabla u}^{(\mathrm{j})}(\eta,\mathrm{t}_2^*)d\eta
    \end{cases}
\end{align}
Consequently, we write
\begin{align}
    \widetilde{\nabla u}^{(i)} + \alpha_i \bm{\mathrm{M}}^{(\mathrm{ii})}_{\mathrm{B}_i}\Big[\widetilde{\nabla u}^{(i)}\Big] \nonumber&= \widetilde{\nabla u}^{\textbf{in}} + \alpha_i\delta^2\mathbb A^{(\mathrm{ii})}[\widetilde{\nabla u}^{(i)}](\xi,\mathrm{t}_1^*) + \delta^3\beta_i\mathbb B^{(\mathrm{ii})}[\partial_\mathrm{t}^4\tilde{u}^{(i)}](\xi,\mathrm{t}_3^*) + \beta_i\delta \mathbb V^{(\mathrm{ii})}[\partial_\mathrm{t}^2\tilde{u}^{(i)}](\xi,\mathrm{t}) \\ \nonumber &+ \delta^4\sum_{\mathrm{j}\ne i}^\mathrm{M}\alpha_\mathrm{j}\mathbb C^{(\mathrm{ij})}[\widetilde{\nabla u}^{(\mathrm{j})}](\xi,\mathrm{t})
    + \delta^4\sum_{\mathrm{j}\ne i}^\mathrm{M}\beta_\mathrm{j}\mathbb D^{(\mathrm{ij})}[\partial_\mathrm{t}^2\tilde{u}^{(\mathrm{j})}](\xi,\mathrm{t}) + \delta^4\sum_{\mathrm{j}\ne i}^\mathrm{M}\beta_\mathrm{j}\mathbb E^{(\mathrm{ij})}[\partial_\mathrm{t}^4\tilde{u}^{(\mathrm{j})}](\xi,\mathrm{t}_4^*) \\  &+ \delta^4 \sum_{\mathrm{j}\ne i}^\mathrm{M}\alpha_\mathrm{j}\mathbb F^{(\mathrm{ij})}[\partial_\mathrm{t}^2\widetilde{\nabla u}^{(\mathrm{j})}](\xi,\mathrm{t}_2^*)
\end{align}
We also know that the Magnetization operator is self-adjoint and bounded, which satisfies the followings
\begin{align}
    \mathbb{M}^{(0)}\Big|_{\mathbb{H}_{0}(\text{div}\;0,\mathrm{B})} = 0, \quad \text{and} \quad \mathbb{M}^{(0)}\Big|_{\mathbb{H}_{0}(\textbf{curl}\; 0,\mathrm{B})} = \mathrm{I}.
\end{align}
From the decomposition (\ref{decompositionmulti}), we define $\overset{1}{\mathbb{P}}, \overset{2}{\mathbb{P}}$ and $\overset{3}{\mathbb{P}}$ to be the natural projectors as follows
\begin{align}
\overset{1}{\mathbb{P}}:= \mathbb{L}^2 \to \mathbb{H}_{0}(\text{div}\; 0,\mathrm{B}), \;   \overset{2}{\mathbb{P}}:= \mathbb{L}^2 \to \mathbb{H}_{0}(\textbf{curl}\;0,\mathrm{B}), \; \text{and}\; \overset{3}{\mathbb{P}}:= \mathbb{L}^2 \to \nabla \mathbb{H}_{\textbf{arm}}. 
\end{align}
Therefore, with these properties, it is natural to see that
\begin{align}\nonumber
    \big\langle \widetilde{\nabla u}^{(i)};\mathrm{e}^{(1)}_\mathrm{n}\big\rangle = 0, \big\langle \bm{\mathrm{M}}^{(\mathrm{ii})}_{\mathrm{B}_i}\Big[\widetilde{\nabla u}^{(i)}\Big];\mathrm{e}^{(1)}_\mathrm{n}\big\rangle = 0, \big\langle \widetilde{\nabla u}^{\textbf{in}};\mathrm{e}^{(1)}_\mathrm{n}\big\rangle = 0, \big\langle \mathbb A^{(\mathrm{ii})}[\widetilde{\nabla u}^{(i)}];\mathrm{e}^{(1)}_\mathrm{n}\big\rangle = 0, \big\langle \mathbb B^{(\mathrm{ii})}[\partial_\mathrm{t}^4\tilde{u}^{(i)}];\mathrm{e}^{(1)}_\mathrm{n}\big\rangle = 0.
\end{align}
Then, using the similar techniques as in (\cite{AM2}), we deduce
\begin{align}
    &\nonumber\big\Vert \widetilde{\nabla u}^{(i)}\big\Vert^2_{\mathbb{L}^2(\mathrm{B}_i)} \\ \nonumber&= \sum_{r=1}^3 \big\Vert \overset{r}{\mathbb{P}}(\widetilde{\nabla u}^{(i)})\big\Vert^2_{\mathbb{L}^2(\mathrm{B}_i)}
    \\ \nonumber &\lesssim \frac{1}{|1+\alpha_i\lambda^{(3)}_\mathrm{n}|^2}\big\Vert \widetilde{\nabla u}^{\textbf{in}}\big\Vert^2_{\mathbb{L}^2(\mathrm{B}_i)} + \frac{\alpha_i^2}{|1+\alpha_i\lambda^{(3)}_\mathrm{n}|^2}\delta^4\big\Vert \mathbb A^{(\mathrm{ii})}[\widetilde{\nabla u}^{(i)}]\big\Vert^2_{\mathbb{L}^2(\mathrm{B}_i)} + \frac{\beta_i^2}{|1+\alpha_i\lambda^{(3)}_\mathrm{n}|^2}\delta^6\big\Vert\mathbb B^{(\mathrm{ii})}[\partial_\mathrm{t}^4\tilde{u}^{(i)}]\big\Vert^2_{\mathbb{L}^2(\mathrm{B}_i)} 
    \\ \nonumber &+ \frac{\beta_i^2}{|1+\alpha_i\lambda^{(3)}_\mathrm{n}|^2}\delta^2\big\Vert\mathbb V^{(\mathrm{ii})}[\partial_\mathrm{t}^2\tilde{u}^{(i)}]\big\Vert^2_{\mathbb{L}^2(\mathrm{B}_i)} + \delta^8\sum_{\mathrm{j}\ne i}^\mathrm{M}\frac{\alpha_\mathrm{j}^2}{|1+\alpha_i\lambda^{(3)}_\mathrm{n}|^2}\big\Vert\mathbb C^{(\mathrm{ij})}[\widetilde{\nabla u}^{(\mathrm{j})}]\big\Vert^2_{\mathbb{L}^2(\mathrm{B}_i)}
    \\  &\nonumber+ \delta^8\sum_{\mathrm{j}\ne i}^\mathrm{M}\frac{\beta_\mathrm{j}^2}{|1+\alpha_i\lambda^{(3)}_\mathrm{n}|^2}\big\Vert\mathbb D^{(\mathrm{ij})}[\partial_\mathrm{t}^2\tilde{u}^{(\mathrm{j})}]\big\Vert^2_{\mathbb{L}^2(\mathrm{B}_i)}
    + \delta^8\sum_{\mathrm{j}\ne i}^\mathrm{M}\frac{\beta_\mathrm{j}^2}{|1+\alpha_i\lambda^{(3)}_\mathrm{n}|^2}\big\Vert\mathbb E^{(\mathrm{ij})}[\partial_\mathrm{t}^4\tilde{u}^{(\mathrm{j})}]\big\Vert^2_{\mathbb{L}^2(\mathrm{B}_i)} 
    \\  &+ \delta^8\sum_{\mathrm{j}\ne i}^\mathrm{M}\frac{\alpha_\mathrm{j}^2}{|1+\alpha_i\lambda^{(3)}_\mathrm{n}|^2}\big\Vert\mathbb F^{(\mathrm{ij})}[\partial_\mathrm{t}^2\tilde{\nabla u}^{(\mathrm{j})}]\big\Vert^2_{\mathbb{L}^2(\mathrm{B}_i)}
\end{align}
We then use the continuity of the operators $A^{(\mathrm{ii})}$, $\mathbb B^{(\mathrm{ii})}$ and $\mathbb V^{(\mathrm{ii})}$ to deduce
\begin{align}\label{multiesti}
    \big\Vert \widetilde{\nabla u}^{(i)}\big\Vert^2_{\mathbb{L}^2(\mathrm{B}_i)} 
    &\lesssim \nonumber\frac{1}{|1+\alpha_i\lambda^{(3)}_\mathrm{n}|^2}\big\Vert \widetilde{\nabla u}^{\textbf{in}}\big\Vert^2_{\mathbb{L}^2(\mathrm{B}_i)} + \frac{\alpha_i^2}{|1+\alpha_i\lambda^{(3)}_\mathrm{n}|^2}\delta^4\big\Vert \widetilde{\nabla u}^{(i)}\big\Vert^2_{\mathbb{L}^2(\mathrm{B}_i)} + \frac{\beta_i^2}{|1+\alpha_i\lambda^{(3)}_\mathrm{n}|^2}\delta^6\big\Vert\partial_\mathrm{t}^4\tilde{u}^{(i)}\big\Vert^2_{\mathbb{L}^2(\mathrm{B}_i)} 
    \\ \nonumber &+ \frac{\beta_i^2}{|1+\alpha_i\lambda^{(3)}_\mathrm{n}|^2}\delta^2\big\Vert\mathbb \partial_\mathrm{t}^2\tilde{u}^{(i)}\big\Vert^2_{\mathbb{L}^2(\mathrm{B}_i)} 
    + \delta^8\sum_{\mathrm{j}\ne i}^\mathrm{M}\frac{\alpha_\mathrm{j}^2}{|1+\alpha_i\lambda^{(3)}_\mathrm{n}|^2}\big\Vert\mathbb C^{(\mathrm{ij})}[\widetilde{\nabla u}^{(\mathrm{j})}]\big\Vert^2_{\mathbb{L}^2(\mathrm{B}_i)} 
    \\ \nonumber &+ \delta^8\sum_{\mathrm{j}\ne i}^\mathrm{M}\frac{\beta_\mathrm{j}^2}{|1+\alpha_i\lambda^{(3)}_\mathrm{n}|^2}\big\Vert\mathbb D^{(\mathrm{ij})}[\partial_\mathrm{t}^2\tilde{u}^{(\mathrm{j})}]\big\Vert^2_{\mathbb{L}^2(\mathrm{B}_i)}
    + \delta^8\sum_{\mathrm{j}\ne i}^\mathrm{M}\frac{\beta_\mathrm{j}^2}{|1+\alpha_i\lambda^{(3)}_\mathrm{n}|^2}\big\Vert\mathbb E^{(\mathrm{ij})}[\partial_\mathrm{t}^4\tilde{u}^{(\mathrm{j})}]\big\Vert^2_{\mathbb{L}^2(\mathrm{B}_i)}  
    \\ &+ \delta^8\sum_{\mathrm{j}\ne i}^\mathrm{M}\frac{\alpha_\mathrm{j}^2}{|1+\alpha_i\lambda^{(3)}_\mathrm{n}|^2}\big\Vert\mathbb F^{(\mathrm{ij})}[\partial_\mathrm{t}^2\tilde{\nabla u}^{(\mathrm{j})}]\big\Vert^2_{\mathbb{L}^2(\mathrm{B}_i)}
\end{align}
Next step is to estimate $\big\Vert\mathbb C^{(\mathrm{ij})}[\widetilde{\nabla u}^{(\mathrm{j})}]\big\Vert^2_{\mathbb{L}^2(\mathrm{B}_i)}$. In order to do so, we have for $\mathrm{x}\in \Omega_i$
\begin{align}
\nonumber
    \Big| \mathbb C^{(\mathrm{ij})}[\widetilde{\nabla u}^{(\mathrm{j})}] \Big| &= \Bigg| \delta^{-1}\int_{\mathrm{B}_\mathrm{j}}\nabla\text{div}\ \mathcal{G}^{(0)}(\mathrm{z}_i,\mathrm{z}_\mathrm{j}) \widetilde{\nabla u}^{(\mathrm{j})}(\eta,\mathrm{t})d\eta
    + \int_{\mathrm{B}_\mathrm{j}}\nabla\text{div} \int_0^1\mathop{\nabla}\limits_{\xi}\mathcal{G}^{(0)}\big(\mathrm{z}_i+\vartheta_1\xi\delta;\mathrm{z}_\mathrm{j}\big)\cdot \xi d\vartheta_1\widetilde{\nabla u}^{(\mathrm{j})}(\eta,\mathrm{t})d\eta
    \\ \nonumber\displaystyle &+ \int_{\mathrm{B}_\mathrm{j}}\nabla\text{div} \int_0^1\mathop{\nabla}\limits_{\eta}\mathcal{G}^{(0)}\big(\mathrm{z}_i;\mathrm{z}_\mathrm{j}+\vartheta_1\eta\delta\big)\cdot \eta d\vartheta_1\widetilde{\nabla u}^{(\mathrm{j})}(\eta,\mathrm{t})d\eta \Bigg|
    \\ \nonumber &\lesssim \Bigg|\delta^{-1}\Big\langle \nabla\text{div}\ \mathcal{G}^{(0)}(\mathrm{z}_i,\mathrm{z}_\mathrm{j});\widetilde{\nabla u}^{(\mathrm{j})}(\eta,\mathrm{t})\Big\rangle\Bigg| + \Bigg|\Big\langle \nabla\text{div} \int_0^1\mathop{\nabla}\limits_{\xi}\mathcal{G}^{(0)}\big(\mathrm{z}_i+\vartheta_1\xi\delta;\mathrm{z}_\mathrm{j}\big)\cdot \xi d\vartheta_1;\widetilde{\nabla u}^{(\mathrm{j})}(\eta,\mathrm{t})\Big\rangle\Bigg|
    \\ \nonumber &+ \Bigg|\Big\langle \nabla\text{div} \int_0^1\mathop{\nabla}\limits_{\eta}\mathcal{G}^{(0)}\big(\mathrm{z}_i;\mathrm{z}_\mathrm{j}+\vartheta_1\eta\delta\big)\cdot \eta d\vartheta_1;\widetilde{\nabla u}^{(\mathrm{j})}(\eta,\mathrm{t})\Big\rangle\Bigg|
    \\ \nonumber &\lesssim \delta^{-1} \mathrm{d}_{\mathrm{ij}}^{-3} \big\Vert \widetilde{\nabla u}^{(\mathrm{j})}\big\Vert^2_{\mathbb{L}^2(\mathrm{B}_\mathrm{j})}
\end{align}
Therefore, we have
\begin{align}\label{me1}
    \big\Vert\mathbb C^{(\mathrm{ij})}[\widetilde{\nabla u}^{(\mathrm{j})}]\big\Vert^2_{\mathbb{L}^2(\mathrm{B}_i)} \nonumber&= \int_{\mathrm{B}_i}\Big| \mathbb C^{(\mathrm{ij})}[\widetilde{\nabla u}^{(\mathrm{j})}] \Big|^2 d\xi
    \\ &\lesssim \textbf{Vol}(\mathrm{B}_j)\delta^{-2} \mathrm{d}_{\mathrm{ij}}^{-6} \big\Vert \widetilde{\nabla u}^{(\mathrm{j})}\big\Vert^2_{\mathbb{L}^2(\mathrm{B}_\mathrm{j})}.
\end{align}
In a similar way, we can show that
\begin{align}\label{me2}
    \big\Vert\mathbb F^{(\mathrm{ij})}[\partial_\mathrm{t}^2\widetilde{\nabla u}^{(\mathrm{j})}]\big\Vert^2_{\mathbb{L}^2(\mathrm{B}_i)} \lesssim \delta^{-2} \mathrm{d}_{\mathrm{ij}}^{-2} \big\Vert \partial_\mathrm{t}^2\widetilde{\nabla u}^{(\mathrm{j})}\big\Vert^2_{\mathbb{L}^2(\mathrm{B}_\mathrm{j})}.
\end{align}
Moreover, we next estimate $\big\Vert\mathbb D^{(\mathrm{ij})}[\partial_\mathrm{t}^2\tilde{u}^{(\mathrm{j})}]\big\Vert^2_{\mathbb{L}^2(\mathrm{B}_i)}$. In order to do so, we have for $\mathrm{x}\in \Omega_i$
\begin{align}
    \nonumber
    \Big| \mathbb D^{(\mathrm{ij})}[\partial^2_\mathrm{t}\tilde{ u}^{(\mathrm{j})}] \Big| &= \Bigg| -\delta^{-1} \int_{\mathrm{B}_\mathrm{j}} \nabla\mathcal{G}^{(0)}(\mathrm{z}_i,\mathrm{z}_\mathrm{j})\partial^2_\mathrm{t} \tilde{u}^{(\mathrm{j})}(\eta,\mathrm{t})d\eta
     -  \int_{\mathrm{B}_\mathrm{j}}\nabla\int_0^1\mathop{\nabla}\limits_{\xi}\mathcal{G}^{(0)}\big(\mathrm{z}_i+\vartheta_2\xi\delta;\mathrm{z}_\mathrm{j}\big)\cdot \xi d\vartheta_2 \partial^2_\mathrm{t}\tilde{u}^{(\mathrm{j})}(\eta,\mathrm{t})d\eta
    \\ \displaystyle\nonumber&- \int_{\mathrm{B}_\mathrm{j}}\nabla\int_0^1\mathop{\nabla}\limits_{\eta}\mathcal{G}^{(0)}\big(\mathrm{z}_i;\mathrm{z}_\mathrm{j}+\vartheta_2\eta\delta\big)\cdot \eta d\vartheta_2\partial^2_\mathrm{t} \tilde{u}^{(\mathrm{j})}(\eta,\mathrm{t})d\eta \Bigg|
    \\ \nonumber&\lesssim \Bigg|\delta^{-1}\Big\langle \nabla\mathcal{G}^{(0)}(\mathrm{z}_i,\mathrm{z}_\mathrm{j});\widetilde{\nabla u}^{(\mathrm{j})}(\eta,\mathrm{t})\Big\rangle\Bigg| + \Bigg|\Big\langle \nabla \int_0^1\mathop{\nabla}\limits_{\xi}\mathcal{G}^{(0)}\big(\mathrm{z}_i+\vartheta_1\xi\delta;\mathrm{z}_\mathrm{j}\big)\cdot \xi d\vartheta_1;\partial^2_\mathrm{t}\tilde{ u}^{(\mathrm{j})}(\eta,\mathrm{t})\Big\rangle\Bigg|
    \\ \nonumber &+ \Bigg|\Big\langle \nabla \int_0^1\mathop{\nabla}\limits_{\eta}\mathcal{G}^{(0)}\big(\mathrm{z}_i;\mathrm{z}_\mathrm{j}+\vartheta_1\eta\delta\big)\cdot \eta d\vartheta_1;\partial^2_\mathrm{t}\tilde{u}^{(\mathrm{j})}(\eta,\mathrm{t})\Big\rangle\Bigg|
    \\ \nonumber &\lesssim \delta^{-1} \mathrm{d}_{\mathrm{ij}}^{-2} \big\Vert \partial_\mathrm{t}^2\tilde{u}^{(\mathrm{j})}\big\Vert_{\mathbb{L}^2(\mathrm{B}_\mathrm{j})}
\end{align}
Therefore, we have
\begin{align}\label{me3}
    \big\Vert\mathbb D^{(\mathrm{ij})}[\partial_\mathrm{t}^2\tilde{ u}^{(\mathrm{j})}]\big\Vert^2_{\mathbb{L}^2(\mathrm{B}_i)} \nonumber&= \int_{\mathrm{B}_i}\Big| \mathbb D^{(\mathrm{ij})}[\partial_\mathrm{t}^2\tilde{u}^{(\mathrm{j})}] \Big|^2 d\xi
    \\ &\lesssim \delta^{-2} \mathrm{d}_{\mathrm{ij}}^{-4} \big\Vert \partial_\mathrm{t}^2\tilde{ u}^{(\mathrm{j})}\big\Vert^2_{\mathbb{L}^2(\mathrm{B}_\mathrm{j})}.
\end{align}
In a similar way, we can show that
\begin{align}\label{me4}
    \big\Vert\mathbb E^{(\mathrm{ij})}[\partial_\mathrm{t}^4\tilde{u}^{(\mathrm{j})}]\big\Vert^2_{\mathbb{L}^2(\mathrm{B}_i)} \lesssim \delta^{-2} \mathrm{d}_{\mathrm{ij}}^{-2} \big\Vert \partial_\mathrm{t}^4\tilde{u}^{(\mathrm{j})}\big\Vert^2_{\mathbb{L}^2(\mathrm{B}_\mathrm{j})}.
\end{align}
Consequently, inserting (\ref{me1}), (\ref{me2}), (\ref{me3}), and (\ref{me4}) in (\ref{multiesti}), we obtain
\begin{align}
    \big\Vert \widetilde{\nabla u}^{(i)}\big\Vert^2_{\mathbb{L}^2(\mathrm{B}_i)} 
     &\lesssim\nonumber \frac{1}{|1+\alpha_i\lambda^{(3)}_\mathrm{n}|^2}\big\Vert \widetilde{\nabla u}^{\textbf{in}}\big\Vert^2_{\mathbb{L}^2(\mathrm{B}_i)} + \frac{\alpha_i^2}{|1+\alpha_i\lambda^{(3)}_\mathrm{n}|^2}\delta^4\big\Vert \widetilde{\nabla u}^{(i)}\big\Vert^2_{\mathbb{L}^2(\mathrm{B}_i)} + \frac{\beta_i^2}{|1+\alpha_i\lambda^{(3)}_\mathrm{n}|^2}\delta^6\big\Vert\partial_\mathrm{t}^4\tilde{u}^{(i)}\big\Vert^2_{\mathbb{L}^2(\mathrm{B}_i)} 
    \\ \nonumber &+ \frac{\beta_i^2}{|1+\alpha_i\lambda^{(3)}_\mathrm{n}|^2}\delta^2\big\Vert\mathbb \partial_\mathrm{t}^2\tilde{u}^{(i)}\big\Vert^2_{\mathbb{L}^2(\mathrm{B}_i)} 
    + \delta^6\textbf{Vol}(\mathrm{B}_j)\sum_{\mathrm{j}\ne i}^\mathrm{M}\frac{\alpha_\mathrm{j}^2}{|1+\alpha_i\lambda^{(3)}_\mathrm{n}|^2}\mathrm{d}_{\mathrm{ij}}^{-6} \big\Vert \widetilde{\nabla u}^{(\mathrm{j})}\big\Vert^2_{\mathbb{L}^2(\mathrm{B}_\mathrm{j})} 
    \\ \nonumber &+ \delta^6\sum_{\mathrm{j}\ne i}^\mathrm{M}\frac{\beta_\mathrm{j}^2}{|1+\alpha_i\lambda^{(3)}_\mathrm{n}|^2}\mathrm{d}_{\mathrm{ij}}^{-4} \big\Vert \partial_\mathrm{t}^2\tilde{ u}^{(\mathrm{j})}\big\Vert^2_{\mathbb{L}^2(\mathrm{B}_\mathrm{j})}
    + \delta^6\sum_{\mathrm{j}\ne i}^\mathrm{M}\frac{\beta_\mathrm{j}^2}{|1+\alpha_i\lambda^{(3)}_\mathrm{n}|^2}\big\Vert\mathrm{d}_{\mathrm{ij}}^{-2} \big\Vert \partial_\mathrm{t}^4\tilde{u}^{(\mathrm{j})}\big\Vert^2_{\mathbb{L}^2(\mathrm{B}_\mathrm{j})} 
   \\ &+ \delta^6\sum_{j\ne i}^\mathrm{M}\frac{\alpha_\mathrm{j}^2}{|1+\alpha_i\lambda^{(3)}_\mathrm{n}|^2} \mathrm{d}_{\mathrm{ij}}^{-2} \big\Vert \partial_\mathrm{t}^2\widetilde{\nabla u}^{(\mathrm{j})}\big\Vert^2_{\mathbb{L}^2(\mathrm{B}_\mathrm{j})}.
\end{align}
Then, as $\alpha_i\sim \delta^{-2}$ and $\beta_i\sim \delta^{-2} $ for $i=1,2,\ldots,\mathrm{M}$, and using the a priori knowledge of  $\big\Vert \partial^n_\mathrm{t}\tilde{ u}^{(i)}(\cdot,\mathrm{t})\big\Vert^2_{\mathbb{L}^2(\mathrm{B}_i)}$ for $n=0,1,\ldots$ we have
\begin{align}
    \big\Vert \widetilde{\nabla u}^{(i)}\big\Vert^2_{\mathbb{L}^2(\mathrm{B}_i)} \lesssim \delta^2 + \delta^6\textbf{Vol}(\mathrm{B}_j)\sum_{\mathrm{j}\ne i}^\mathrm{M}\mathrm{d}_{\mathrm{ij}}^{-6} \big\Vert \widetilde{\nabla u}^{(\mathrm{j})}\big\Vert^2_{\mathbb{L}^2(\mathrm{B}_\mathrm{j})}.
\end{align}
Using the similar argument, we can estimate $\big\Vert \partial^n_\mathrm{t}\widetilde{\nabla u}^{(i)}\big\Vert_{\mathbb{L}^2(\mathrm{B}_i)}$ for $n=1,2,\ldots$ as
\begin{align}
    \big\Vert \partial^n_\mathrm{t}\widetilde{\nabla u}^{(i)}\big\Vert^2_{\mathbb{L}^2(\mathrm{B}_i)} \lesssim \delta^2 + \delta^6\textbf{Vol}(\mathrm{B}_j)\sum_{\mathrm{j}\ne i}^\mathrm{M}\mathrm{d}_{\mathrm{ij}}^{-6} \big\Vert \partial^n_\mathrm{t}\widetilde{\nabla u}^{(\mathrm{j})}\big\Vert^2_{\mathbb{L}^2(\mathrm{B}_\mathrm{j})}.
\end{align}
Then, under the condition 
\begin{align}
    \frac{\rho_\mathrm{m}}{4\pi}\textbf{Vol}(\mathrm{B}_j)\delta^6\max_{\mathrm{n}\in\mathbb N}\max_{1\le i\le\mathrm{M}} \sum_{j\ne i}\frac{\alpha_\mathrm{j}^2}{|1+\alpha_i\lambda^{(3)}_\mathrm{n}|^2} d^{-6}_{ij} <1,
\end{align}
We have
\begin{align}
\big\Vert \partial^n_\mathrm{t}\widetilde{\nabla u}^{(i)}\big\Vert_{\mathbb{L}^2(\mathrm{B}_i)} \lesssim \delta.    
\end{align}
This completes the proof.
\end{proof}
\subsection{Proof of Proposition \ref{p3}}
Using the derived estimate as in Lemma \ref{l1}, we therefore deduce
\begin{align}
    \nonumber\big\Vert \partial_\nu u^{(i)}(\cdot,\mathrm{t})\big\Vert_{\mathrm{H}^{-\frac{1}{2}}(\partial\Omega_i)} 
    \nonumber &\lesssim \big\Vert \partial_\nu\tilde{ u}^{(i)}(\cdot,\mathrm{t})\big\Vert_{\mathrm{H}^{-\frac{1}{2}}(\partial\mathrm{B}_i)} 
   \\ \nonumber &\lesssim \big\Vert \nabla\tilde{ u}^{(i)}(\cdot,\mathrm{t})\big\Vert_{\mathbb{H}(\text{div},\mathrm{B}_i)}
   \\ \nonumber &\lesssim \Big(\big\Vert \nabla\tilde{ u}^{(i)}(\cdot,\mathrm{t})\big\Vert^2_{\mathrm{L}^2(\mathrm{B}_i)} + \big\Vert \Delta\tilde{ u}^{(i)}(\cdot,\mathrm{t})\big\Vert^2_{\mathrm{L}^2(\mathrm{B}_i)}\Big)^\frac{1}{2}
    \\ \nonumber &\lesssim \Big(\delta^2\underbrace{\big\Vert\widetilde{\nabla u}^{(i)}(\cdot,\mathrm{t})\big\Vert^2_{\mathrm{L}^2(\mathrm{B}_i)}}_{\sim \delta^2} + \delta^4\underbrace{\big\Vert \widetilde{\Delta u}^{(i)}(\cdot,\mathrm{t})\big\Vert^2_{\mathrm{L}^2(\mathrm{B}_i)}}_{\sim 1}\Big)^\frac{1}{2} \sim \delta^2.
\end{align}
This completes the proof. \qed
\bigskip

\noindent
\textbf{Acknowledgements.}
\newline
\textbf{Data Availability Statement.} Data sharing is not applicable to this article as no datasets were generated or analysed during the current study.
\bigskip

\noindent
\textbf{Declarations.}
\newline
\textbf{Conflict of interest.} The authors declare that they have no conflict of interest.

\end{document}